\documentclass[10pt,oneside,a4paper]{amsart}
\usepackage{amsmath,amsfonts,amsthm, amssymb}
\usepackage{array}
\usepackage[all]{xy}
\usepackage{color}
\usepackage[orange]{xcolor}

\theoremstyle{plain}
\newtheorem{theorem}{Theorem}[section]
\newtheorem{proposition}[theorem]{Proposition}
\newtheorem{lemma}[theorem]{Lemma}
\newtheorem{corollary}[theorem]{Corollary}

\theoremstyle{definition}
\newtheorem{definition}[theorem]{Definition}

\theoremstyle{remark}

\newtheorem{remark}[theorem]{Remark}
\newtheorem{example}[theorem]{Example}

\newcommand{\Mor}{\mathrm{Mor}}
\renewcommand{\lim}{\mathrm{lim}}

\newcommand{\Ext}{\mathrm{Ext}}

\newcommand{\Hom}{\mathrm{Hom}}

\newcommand{\Def}{\mathrm{Def}}

\newcommand{\Ob}{\mathrm{Ob}}

\newcommand{\AAA}{\mathfrak{a}}
\newcommand{\BBB}{\mathfrak{b}}

\newcommand{\CC}{\mathbf{C}}

\newcommand{\Mod}{\ensuremath{\mathsf{Mod}} }

\newcommand{\Fun}{\ensuremath{\mathsf{Fun}}}

\newcommand{\lra}{\longrightarrow}

\newcommand{\aaa}{\ensuremath{\mathcal{A}}}
\newcommand{\bbb}{\ensuremath{\mathcal{B}}}
\newcommand{\ccc}{\ensuremath{\mathcal{C}}}

\newcommand{\fff}{\ensuremath{\mathcal{F}}}
\newcommand{\GGG}{\ensuremath{\mathcal{G}}}

\newcommand{\nnn}{\ensuremath{\mathcal{N}}}

\newcommand{\ppp}{\ensuremath{\mathcal{P}}}

\newcommand{\uuu}{\ensuremath{\mathcal{U}}}

\newcommand{\h}{\mathrm{Hoch}}
\newcommand{\s}{\mathrm{simp}}
\newcommand{\cgs}[1]{\CC_\mathrm{GS}^{#1}(\aaa,M)}
\newcommand{\cgsn}[1]{\bar{\CC}_\mathrm{GS}^{#1}(\aaa,M)}
\newcommand{\cgsnr}[1]{\bar{\CC}_\mathrm{GS}'^{#1}(\aaa,M)}
\newcommand{\ctil}[1]{\CC_\uuu^{#1}(\tilde{\aaa},\tilde{M})}

\newcommand{\Part}[1]{\mathrm{Part}(#1)}
\newcommand{\Seq}{\mathrm{Seq}}
\newcommand{\Seqq}{\mathrm{Seqq}}
\newcommand{\simp}{\mathrm{simp}}

\allowdisplaybreaks
\numberwithin{equation}{section}

\title{The Gerstenhaber-Schack complex for prestacks}
\author{ Dinh Van Hoang}
\address[ Dinh Van Hoang]{Universiteit Antwerpen, Departement Wiskunde-Informatica, Middelheimcampus,
Middelheimlaan 1,
2020 Antwerp, Belgium}
\email{hoang.dinhvan@uantwerpen.be}

\author{Wendy Lowen}
\address[Wendy Lowen]{Universiteit Antwerpen, Departement Wiskunde-Informatica, Middelheimcampus,
Middelheimlaan 1,
2020 Antwerp, Belgium}
\email{wendy.lowen@uantwerpen.be}

\thanks{The authors acknowledge the support of the European Union for the ERC grant No 257004-HHNcdMir and the support of the Research Foundation Flanders (FWO) under Grant No G.0112.13N}

\begin{document}
\maketitle

\begin{abstract}
	Building on the work of Gerstenhaber and Schack for presheaves of algebras, we define a Gerstenhaber-Schack complex $\CC^{\bullet}_{\mathrm{GS}}(\aaa)$ for an arbitrary prestack $\aaa$, that is a pseudofunctor taking values in linear categories over a commutative ground ring. In the general case, the differential is no longer simply the sum of Hochschild and simplicial contributions as in the presheaf case, but contains additional higher components as well. If $\tilde{\aaa}$ denotes the Grothendieck construction of $\aaa$, which is a $\uuu$-graded category, we explicitly construct inverse quasi-isomorphisms between $\CC^{\bullet}_{\mathrm{GS}}(\aaa)$ and the Hochschild complex $\CC_{\uuu}(\tilde{\aaa})$. As the Homotopy Transfer Theorem applies to our construction, one can transfer the dg Lie structure present on the Hochschild complex in order to obtain an $L_{\infty}$-structure on $\CC^{\bullet}_{\mathrm{GS}}(\aaa)$, which controlls the higher deformation theory of the prestack $\aaa$.
\end{abstract}

\section{introduction}

Throughout the introduction, let $k$ be a field. In \cite{gerstenhaberschack}, \cite{gerstenhaberschack1}, \cite{gerstenhaberschack2} Gerstenhaber and Schack define the Hochschild cohomology of a presheaf $\aaa$ of $k$-algebras over a poset $\uuu$  as an Ext of bimodules
$HH^n(\aaa) = \Ext^n_{\aaa - \aaa}(\aaa, \aaa)$,
in analogy with the case of $k$-algebras. They construct a complex $\CC^{\bullet}_{\mathrm{GS}}(\aaa)$ which computes this cohomology, obtained as the totalization of a double complex with horizontal Hochschild differentials  and vertical simplicial differentials. From $\aaa$, they construct a single $k$-algebra $\aaa !$ such that 
\begin{equation}\label{aaa!}
	HH^n(\aaa) \cong HH^n(\aaa !)
\end{equation}
for the standard Hochschild cohomology of $\aaa !$ on the right hand side. Further, the authors construct two explicit cochain maps
\begin{equation}\label{tau}
	\tau: \CC^{\bullet}_{\mathrm{GS}}(\aaa) \lra \CC^{\bullet}(\aaa !) \hspace{0,5cm} \text{and} \hspace{0,5cm} \hat{\tau}:\CC^{\bullet}(\aaa !) \lra  \CC^{\bullet}_{\mathrm{GS}}(\aaa)
\end{equation}
relating their complex $\CC^{\bullet}_{\mathrm{GS}}(\aaa)$ to the Hochschild complex $\CC^{\bullet}(\aaa !)$, which they prove to be inverse quasi-isomorphisms. They present two essentially different approaches to \eqref{aaa!}, \eqref{tau} and the relationship between the two:
\begin{enumerate}
	\item[(A1)] In a first approach \cite{gerstenhaberschack}, \cite{gerstenhaberschack1}, \eqref{aaa!} follows from their (difficult) Special Cohomology Comparison Theorem (SCCT) which compares more general bimodule Ext groups. Both sides of \eqref{aaa!} are particular cases of such $\Ext$ groups, and a universal delta functor argument shows that the isomorphism \eqref{aaa!} is actually induced by the map $\tau$ in \eqref{tau}, whence the latter is a quasi-isomorphism. 
	\item[(A2)] In a second approach \cite{gerstenhaberschack2}, in case $\uuu$ is a finite poset, the authors focus on the compositions $\hat{\tau} \tau$ and $\tau \hat{\tau}$. They prove directly that $\hat{\tau} \tau = 1$, and based upon the $\Ext$ interpretation of the cohomology of $\CC(\aaa !)$, after extending $\hat{\tau} \tau$ to a natural transformation on $\CC^{\bullet}(\aaa !, -)$, a universal delta functor argument shows that $H^{\bullet}(\hat{\tau} \tau) = 1$. Thus, in this case the isomorphism \eqref{aaa!} follows without invoking the SCCT.
\end{enumerate}
Unlike in the algebra case, there is no perfect match between $HH^2(\aaa)$ and deformations of $\aaa$ as a presheaf. However, it turns out that $HH^2(\aaa)$ naturally parametrizes deformations of $\aaa$ \emph{as a twisted presheaf}, as is seen from direct inspection of the complex $\CC^{\bullet}_{\mathrm{GS}}(\aaa)$ \cite{dinhvanliulowen}.

Another way to understand the occurence of twists is by viewing a presheaf of algebras as a prestack, that is a pseudofunctor taking values in $k$-linear categories (algebras are considered as one object categories). If $\aaa$ is a prestack over a small category $\uuu$, then $\aaa$ has an associated $\uuu$-graded category $\tilde{\aaa}$, obtained through a $k$-linear version of the Grothendieck construction from \cite{SGA1}.  If $\aaa$ is a presheaf over a poset, then $\tilde{\aaa}$ and $\aaa!$ are closely related.  In \cite{lowenmap} it was shown based upon the construction of $\tilde{\aaa}$ that the appropriate $\uuu$-graded Hochschild complex $\CC^{\bullet}_{\uuu}(\tilde{\aaa})$ of $\tilde{\aaa}$ satisfies
\begin{equation}\label{Extuuu}
	H^n\CC_{\uuu}(\tilde{\aaa}) = \Ext^n_{\tilde{\aaa} - \tilde{\aaa}}(\tilde{\aaa}, \tilde{\aaa})
\end{equation}
and controls deformations of $\tilde{\aaa}$ as a $\uuu$-graded category ($\mathrm{Def}_{\uuu}(\tilde{\aaa})$) and, equivalently, deformations of $\aaa$ as a prestack ($\mathrm{Def}_{\mathrm{pre}}(\aaa)$). Further, in \cite{lowenvandenberghCCT}, Lowen and Van den Bergh prove a Cohomology Comparison Theorem (CCT) for prestacks $\aaa$. If we define $HH^n(\aaa) = \Ext^n_{\aaa - \aaa}(\aaa, \aaa)$ and $HH^n_{\uuu}(\tilde{\aaa}) = \Ext^n_{\tilde{\aaa} - \tilde{\aaa}}(\tilde{\aaa}, \tilde{\aaa})$, it follows in particular from the CCT that 
\begin{equation}\label{CCT2}
	HH^n(\aaa) \cong HH^n_{\uuu}(\tilde{\aaa}),
\end{equation}
that is, the analogue of \eqref{aaa!} holds. 

All of the above suggests that it is most natural to work at once in the context of arbitrary prestacks $\aaa$. In particular, it should be possible to define a Gerstenhaber-Schack complex $\CC^{\bullet}_{\mathrm{GS}}(\aaa)$ which is directly seen to controll prestack deformations of $\aaa$, and such that we can modify the inverse quasi-isomorphisms \eqref{tau} above to this setup. Realizing this is the main goal of this paper. In summary, we have the following picture of the references in which various relations are studied for a prestack $\aaa$, where $[\ast]$ stands for the present paper:
$$\xymatrix{
	{\Ext_{\aaa - \aaa}(\aaa, \aaa)} \ar@{~}[d]_{[11]} & {\CC^{\bullet}_{\mathrm{GS}}(\aaa)} \ar@{~}[r]^{[\ast]} \ar@{~}[d]_{[\ast]} & {\mathrm{Def}_{\mathrm{pre}}(\aaa)} \ar@{~}[d]^{[10]} \\
	{\Ext_{\tilde{\aaa} - \tilde{\aaa}}(\tilde{\aaa}, \tilde{\aaa})} \ar@{~}[r]_{[10]} & {\CC^{\bullet}_{\uuu}(\tilde{\aaa})} \ar@{~}[r]_{[10]} & {\mathrm{Def}_{\uuu}(\tilde{\aaa})} }$$

The content of the paper is as follows. After recalling basic terminology on prestacks and map-graded categories in \S \ref{parparpremap}, the complex $\CC^{\bullet}_{\mathrm{GS}}(\aaa)$ for a prestack $\aaa$ on a small category $\uuu$ is defined in \S \ref{parparGS}. As a graded module, according to \eqref{double}, $\CC^{\bullet}_{\mathrm{GS}}(\aaa)$ is the totalization of a double object which is a minor modification of the one in the presheaf case. Precisely, we put
$$\CC^{p,q}(\aaa) = \prod \Hom_k\big(\aaa(U_p)(A_{q-1}, A_{q}) \otimes \dots \otimes \aaa(U_p)(A_{0}, A_{1}),\aaa(U_0)( \sigma^{\star}A_0,\sigma^{\ast}A_q)\big).$$
Here, the product is taken over all $p$-simplices 
\begin{equation}\label{eqsigma00}
	\sigma = (\xymatrix{ {U_0} \ar[r]_-{u_1} & {U_{1}} \ar[r]_-{u_2} & {\dots} \ar[r]_-{u_{p-1}} & {U_{p-1}} \ar[r]_{u_p} & {U_p}})
\end{equation}
in the nerve of $\uuu$ and all $(q+1)$-tuples $(A_0, \dots, A_q)$ of objects in $\aaa(U_p)$. Further, if we denote, for $u: V \lra U$ in $\uuu$, the associated restriction functor by $u^{\ast}: \aaa(U) \lra \aaa(V)$, then we put $\sigma^{\ast} = (u_p \dots u_2u_1)^{\ast}$ and $\sigma^{\star} = u_1^{\ast}u_2^{\ast} \dots u_p^{\ast}$.

On the other hand, in \eqref{newdifin} we have to introduce a new, more complicated differential
\begin{equation}\label{newdifinin}
	d=d_0+d_1+\dots+d_n:\CC_{\mathrm{GS}}^{n-1}(\aaa)\lra \CC_{\mathrm{GS}}^n(\aaa)
\end{equation}
where $d_0=d_\mathrm{Hoch}$ for the horizontal Hochschild differential $d_\mathrm{Hoch}$ and $d_1=(-1)^nd_\mathrm{simp}$ for the vertical simplicial differential $d_\mathrm{simp}$.
The additional components $d_j$ of $d$, given in \eqref{defnewdif}, are necessary to make the differential square to zero, as is shown in Theorem \ref{propnewdif}. Note that the algebraic structure of the prestack $\aaa$ naturally corresponds to an element
$$(m, f, c) \in \CC^{0,2}(\aaa) \oplus \CC^{1,1}(\aaa) \oplus \CC^{2,0}(\aaa) = \CC^2_{\mathrm{GS}}(\aaa)$$
with $m$ encoding compositions, $f$ encoding restrictions, and $c$ encoding twists. Our definition of the components $d_j$ ensures the following desired result (Theorem \ref{proptwist}), of which the proof makes use of normalized reduced cochains as defined in \S \ref{parnormred}:
\begin{theorem}
The second cohomology group $H^2\CC_{\mathrm{GS}}(\aaa)$ classifies first order deformations of $\aaa$ as a prestack.
\end{theorem}
The definition of the higher components $d_j$ is combinatorial in nature. It makes essential use of the following ingredients:
\begin{itemize}
	\item So called \emph{paths} of natural transformations between $\sigma^{\star}$ and $\sigma^{\ast}$, each path building up a $(p-1)$-simplex in the nerve of $\Fun(\aaa(U_p), \aaa(U_0))$ by using one twist isomorphism in each step (the precise definition is given in the beginning of \S \ref{parnewdif}).
	\item The natural action of shuffle permutations on nerves of categories, as discussed in \S \ref{parshuffle}.
\end{itemize}
In \S \ref{parparcomp} we go on to define cochain maps 
\begin{equation}\label{eqFG}
	\fff:\CC^{\bullet}_{\mathrm{GS}}(\aaa) \lra \CC^{\bullet}_{\uuu}(\tilde{\aaa}) \hspace{0,5cm} \text{and} \hspace{0,5cm} 
	\GGG: \CC^{\bullet}_{\uuu}(\tilde{\aaa}) \lra \CC^{\bullet}_{\mathrm{GS}}(\aaa)
\end{equation}
between $\CC^{\bullet}_{\mathrm{GS}}(\aaa)$ and the $\uuu$-graded Hochschild complex $\CC_{\uuu}^\bullet(\tilde{\aaa})$ as defined in \cite{lowenmap}. 
Although the existence of those maps is inspired by the existence of the maps in \eqref{aaa!}, due to our more complicated differential on $\CC_{\mathrm{GS}}(\aaa)$, the maps in \eqref{eqFG} are actually new and considerably more complicated. Our main theorem is the following (see Proposition \ref{propGF} and Theorem \ref{homotopy}):

\begin{theorem} \label{mainintro} The maps $\fff$ and $\GGG$ are inverse quasi-isomorphisms. More precisely
\begin{enumerate}
\item $\GGG\fff(\phi)=\phi$ for any normalized reduced cochain $\phi$;
\item there is an explicit homotopy $T: \fff \GGG \sim 1$.
\end{enumerate}
\end{theorem}
In combination with \eqref{CCT2} and \eqref{Extuuu}, we thus obtain
\begin{corollary}\label{ExtGS}
	$H^n\CC_{\mathrm{GS}}(\aaa) = \Ext^n_{\aaa - \aaa}(\aaa, \aaa).$
\end{corollary}
Note that, in contrast with \cite{gerstenhaberschack1}, in our setup we do not give a direct proof of Corollary \ref{ExtGS}, whence the approach (A1) is not available to us. 

Our construction of the homotopy $T: \fff \GGG \sim 1$ in Theorem \ref{mainintro}(2) is new even in the presheaf case and has the following important consequence. By the Homotopy Transfer Theorem \cite[Theorem 10.3.9]{lodayvallette},
using $T$ we can transfer the  dg Lie algebra structure present on $\CC_{\uuu}^\bullet(\tilde{\aaa})$ (see \cite{lowenmap}) in order to obtain an $L_{\infty}$-structure on $\CC^{\bullet}_{\mathrm{GS}}(\aaa)$. This $L_{\infty}$-structure determines the higher deformation theory of $\aaa$ as a prestack, which thus becomes equivalent to the higher deformation theory of the $\uuu$-graded category $\tilde{\aaa}$ described in \cite{lowenmap}. A more detailed elaboration of this $L_{\infty}$-structure, as well as a comparison with the $L_{\infty}$ deformation complex described in the literature in an operadic context \cite{fregiermarklyau},\cite{doubek},\cite{merkulovvallette} will appear in \cite{dinhvan}.

\vspace{0,5cm}

\noindent \emph{Acknowledgement.} The second author is very grateful to Michel Van den Bergh for many interesting discussions, and in particular for his proposal of map-graded Hochschild cohomology which was originally made in the context of a local-to-global spectral sequence \cite{lowenvandenberghlocglob}.

\section{Prestacks and Map-graded categories}\label{parparpremap}

Let $k$ be a commutative ground ring. Except for the $\Ext$ interpretations of cohomologies of complexes, which are of secondary importance in the paper, all our results hold true in this generality.

In this section, we recall the notions of prestacks and map-graded categories, thus fixing terminology and notations. As described explicitly in \cite{lowenmap}, prestacks and map-graded categories constitute two different incarnations of mathematical data that are equivalent in a suitable sense. A prestack is a pseudofunctor taking values in $k$-linear categories. The terminology is used as in \cite{lowenvandenberghCCT}, \cite{dinhvanliulowen}. 

Let $\uuu$ be a small category.

\begin{definition}\label{defpseudo}
	A \emph{prestack} $\aaa=(\aaa,m,f,c)$ on $\uuu$ consists of the following data:
	\begin{itemize}
		\item for every object $U \in \uuu$, a $k$-linear category $(\aaa(U),m^U, 1^U)$ where $m^U$ is the composition of morphisms in $\aaa(U)$ and $1^U$ encodes the identity morphisms on $\aaa(U)$;
		\item for every morphism $u\colon V \lra U$ in $\uuu$, a $k$-linear functor $f^u = u^{\ast}\colon \aaa(U) \lra \aaa(V)$. For $u=1_U$, we require that $f^{1_U}=1_U$.
		\item for every couple of morphisms $v\colon W \lra V$, $u\colon V  \lra U$ in $\uuu$, a natural isomorphism
		$$c^{u,v}\colon v^{\ast}u^{\ast} \lra (uv)^{\ast}.$$
		For $u=1$ or $v=1$, we require that $c^{u,v}=1$. Moreover the natural isomorphisms have to satisfy the following coherence condition for every triple $w\colon T\lra W$, $v\colon W\lra V$, $u\colon V\lra U$:
		\begin{equation}\label{compatibility}
			c^{u,vw}(c^{v,w} \circ u^{\ast})=c^{uv,w}(w^{\ast} \circ c^{u,v}).
		\end{equation}
			\end{itemize}
\end{definition}

\begin{remark}
	A presheaf of $k$-linear categories is considered as a prestack in which $c^{u,v}=1$ for every $v\colon W \lra V$, $u\colon V  \lra U$.
\end{remark}

A prestack being a pseudofunctor, we obviously define a morphism of prestacks to be a pseudonatural transformation. 

\begin{definition}
	Consider prestacks $(\aaa, m, f, c)$ and $(\aaa', m', f', c')$ on $\uuu$. A \emph{morphism of prestacks} $(g,\tau)\colon\aaa \lra \aaa'$ consists of the following data:
	\begin{itemize}
		\item for  each $U\in \uuu$,  a functor $g^U\colon\aaa(U)\lra\aaa'(U)$;
		\item for each $u\colon V\lra U$ and $A\in\aaa(U)$, an element $$\tau^{u,A}\in\aaa'(V)(u'^*g^U(A),g^V(u^*A)$$
	\end{itemize}
	These data further satisfy the following conditions: for any $v\colon W\lra V$, $u\colon V\lra U$ and $a\in\aaa(U)(A,B)$,
	\begin{enumerate}
		\item  $m'^V(g^Vu^*(a),\tau^u)=m'^V(\tau^u,u'^* g^U(a))$;
		\item  $m'^W(\tau^{uv},c'^{u,v})=m'^W(g^W(c^{u,v}),\tau^v,v'^*(\tau^u))$;
		\item  $m'^U(\tau^{1_U},1'_U)=g^U(1_U)$.
	\end{enumerate}
\end{definition}

Let $\Mod(k)$ be the category of $k$-modules and let $\underline{\Mod}(k)$ be the constant prestack on $\uuu$ with value $\Mod(k)$. We are mainly interested in modules and bimodules.

\begin{definition}\label{defmodule}
	Let $\aaa$ be a prestack on $\uuu$. An \emph{$\aaa$-module} is a morphism of prestacks $M:\aaa^{\mathrm{op}}\lra \underline{\Mod}(k)$. More precisely, an $\aaa$-module consists of the following data:
	\begin{itemize}
		\item for every $U\in\uuu$, an $\aaa(U)$-module $M^U: \aaa(U)^{\mathrm{op}} \lra \Mod(k)$;
		\item for every $u:V\lra U$, a morphism of $\aaa(U)$-modules $M^u:M^U\lra M^Vu^*$;\\
		such that the following coherence condition holds for every $u:V\lra U, v:W\lra V$: the morphism $M^{uv}$ equals the canonical composition
		$$\xymatrix{M^U\ar[r]^-{M^u} &M^Vu^*\ar[r]^-{M^vu^*}&M^Wv^*u^* \ \ \ar[r]^-{M^W(c^{u,v})}&\ \ M^{W}{(uv)^*} }.$$	
	\end{itemize}
\end{definition}

\begin{definition} Let $\aaa$, $\bbb$ be prestacks on $\uuu$.
	An \emph{$\aaa$-$\bbb$-bimodule} is a module over $\aaa^{\mathrm{op}}\otimes\bbb$. More concretely, an $\aaa$-$\bbb$-bimodule $M$ consists of abelian groups $$M^U(B,A)$$
	for $U\in\mathrm{Ob}(\uuu)$, $A\in\mathrm{Ob}(\aaa(U)), B\in\mathrm{Ob}(\bbb(U))$, together with restriction morphisms $$M^u(B,A):M^U(B,A)\lra M^V(u^*B,u^*A)$$
	for $u:V\lra U$ in $\uuu$ satisfying the natural coherence condition obtained from Definition \ref{defmodule}.
\end{definition}

Next we turn to map-graded categories in the sense of \cite{lowenmap}, where ``map'' stands for the maps in the underlying small category $\uuu$.

\begin{definition} A  \emph{$\uuu$-graded $k$-category} $\mathfrak{a} = (\mathfrak{a}, \mu, \mathrm{id})$ consists of the following data:
	\begin{itemize}
		\item for every object $U\in\uuu$ 
		, we have a set of \emph{objects} $\mathfrak{a}(U)$;
		\item for every morphism $u : V\lra U$ in $\uuu$ and objects $A\in \mathfrak{a}(V),
		B \in \mathfrak{a}(U)$,  we have a $k$-module $\mathfrak{a}_u(A, B)$ of \emph{morphisms}.
	\end{itemize} These data are further equipped with compositions and identity morphisms in the following sense. The composition $\mu$ on $\mathfrak{a}$ consists of operations
	\begin{equation*}
		\mu^{u,v,A,B,C}:\mathfrak{a}_u(B,C)\otimes\mathfrak{a}_v(A,B)\lra\mathfrak{a}_{uv}(A,C)
	\end{equation*}
	satisfying the associativity condition 
	$$\mu^{w,uv,A,C,D}(\mu^{u,v,A,B,C}\otimes 1_{\mathfrak{a}_w(C,D))}=\mu^{wu,v,A,B,D}(1_{\mathfrak{a}_v(A,B)}\otimes \mu^{w,u,B,C,D}).$$
	The identity $\mathrm{id}$ on $\mathfrak{a}$ consists of elements $\mathrm{id}^A\in\mathfrak{a}_1(A,A)$ satisfying the condition 
	$$\mu^{u,1,A,A,B}(1_{\mathfrak{a}_u(A,B)}\otimes \mathrm{id}^A)=1_{\mathfrak{a}_u(A,B)}=\mu^{1,u,A,B,B}(\mathrm{id}^B\otimes 1_{\mathfrak{a}_u(A,B)}).$$
\end{definition}

The most natural type of modules to consider over a map-graded category turn out to be a kind of bimodules:

\begin{definition}
	Let $\mathfrak{a}$ be a $\uuu$-graded $k$-category. An \emph{$\mathfrak{a}$-bimodule} $M$ consists of $k$-modules $$M_u(A,B)$$
	for $u:V\lra U, A\in\mathfrak{a}(V),B\in\mathfrak{a}(U)$ and compositions
	$$\rho:\mathfrak{a}_u(C,D)\otimes M_v(B,C)\otimes \mathfrak{a}_w(A,B)\lra M_{uvw}(A,D)$$
	satisfying the following associativity and identity conditions:
	\begin{enumerate}
		\item $\rho(\mu\otimes 1\otimes\mu)=\rho(1\otimes \rho \otimes 1)$;
		\item $\rho(\mathrm{id}\otimes 1\otimes \mathrm{id})=1$.
	\end{enumerate}
\end{definition}

Let $(\aaa,m,f,c)$ be prestack on $\uuu$. The associated $\uuu$-graded category $(\widetilde{\aaa}, \mu, \mathrm{id})$ is defined as a $k$-linear version of the Grothendieck construction from \cite{SGA1}, more precisely: 
\begin{itemize}
	\item for each object $U\in\uuu$, we put $\widetilde{\aaa}(U)=\mathrm{Ob}(\aaa(U))$;
	\item for every morphism $u:V\lra U$ and objects $A\in\widetilde{\aaa}(V),B\in\widetilde{\aaa}(U)$, we put $$\widetilde{\aaa}_u(A,B)=\aaa(U)(A,u^*B).$$
\end{itemize}
The composition operations $$\mu:\widetilde{\aaa}_u(B,C)\otimes\widetilde{\aaa}_v(A,B)\lra\widetilde{\aaa}_{uv}(A,C)$$
are defined by setting $\mu(b,a)=m(c^{u,v,C},v^*b,a)$ for every $a\in\widetilde{\aaa}_v(A,B),b\in \widetilde{\aaa}_u(B,C)$ and the identities are given by
$\mathrm{id}^A = 1^{U,A} \in \aaa(U)(A,A) = \tilde{\aaa}_{1_U}(A,A)$ for $A \in \aaa(U)$.

There is a natural functor $$\widetilde{(-)}:\mathrm{Bimod}(\aaa)\lra\mathrm{Bimod}(\widetilde{\aaa}): M \longmapsto \widetilde{M}$$   
defined by $$\widetilde{M}_u(A,B):=M^V(A,u^*B)$$                                
for every $u:V\lra U, A\in\tilde{\aaa}(V), B\in\tilde{\aaa}(U)$. In \cite{lowenvandenberghCCT}, inspired by Gerstenhaber and Schack's Cohomology Comparison Theorem \cite{gerstenhaberschack1}, this functor is shown to induce a fully faithful functor on the level of the derived categories. In particular:

\begin{theorem}\cite[Theorem 1.1]{lowenvandenberghCCT}
	For any $M,N\in\mathrm{Bimod}(\aaa)$, we have $$\mathrm{Ext}^n_{\aaa-\aaa}(M,N)\cong \mathrm{Ext}^n_{\tilde{\aaa}-\tilde{\aaa}}(\widetilde{M},\widetilde{N})$$ for all $n$.
\end{theorem}

\section{The Gerstenhaber-Schack complex for prestacks}\label{GScomplex}\label{parparhodgeHKR}\label{parparGS}

If $\aaa$ is a presheaf of $k$-categories, then in complete analogy with the case of presheaves of $k$-algebras treated in \cite[\S 21]{gerstenhaberschack1} and \cite{gerstenhaberschack}, one defines the Gerstenhaber-Schack (GS) complex $(\CC^{\bullet}_{\mathrm{GS}}(\aaa,M),d_{\mathrm{GS}})$ for an $\aaa$-bimodule $M$ as the total complex of a double complex with $d_{\mathrm{GS}} =d_{\mathrm{Hoch}}+ d_{\mathrm{simp}}$ for the horizontal Hochschild differential $d_{\mathrm{Hoch}}$ and the vertical simplicial differential $d_{\mathrm{simp}}$. The cohomology of this complex is called Gerstenhaber-Schack (GS) cohomology and denoted
$$HH^n_{\mathrm{GS}}(\aaa, M) = H^n\CC^{\bullet}_{\mathrm{GS}}(\aaa, M).$$ We denote $\CC^{\bullet}_{\mathrm{GS}}(\aaa) =\CC^{\bullet}_{\mathrm{GS}}(\aaa, \aaa)$ and $HH^n_{\mathrm{GS}}(\aaa) = H^n\CC^{\bullet}_{\mathrm{GS}}(\aaa)$.

In the fashion of \cite[\S 2]{dinhvanliulowen} one sees that the second cohomology group $HH_{\mathrm{GS}}^2(\aaa)$ naturally classifies the first order deformations of $\aaa$ \emph{as a prestack}. Even though many prestacks of interest are in fact presheaves - for instance (restricted) structure sheaves of schemes as treated in \cite{dinhvanliulowen} - the fact that prestacks turn up naturally as deformations suggests that it is really \emph{prestacks} of which one should understand Gerstenhaber-Schack cohomology and deformations in the first place.

Our main aim in this section is to define a Gerstenhaber-Schack (GS) complex $\CC^{\bullet}_{\mathrm{GS}}(\aaa, M)$ for an arbitrary prestack $\aaa$. Contrary to what one may at first expect, the change from the presheaf case to the prestack case is a major one. Indeed, if $\aaa$ is non-trivially twisted ($c^{u,v}\ne 1$), with the natural definitions of $d_{\mathrm{Hoch}}$ and $d_{\mathrm{simp}}$ we now in general have $d_{\mathrm{simp}}^2 \neq 0$ so we do not obtain a double complex. Instead, we construct a more complicated differential  on the total double object $\CC^{\bullet}_{\mathrm{GS}}(\aaa, M)$ by adding more components to the formula. After introducing the double object $\CC^{\bullet}_{\mathrm{GS}}(\aaa, M)$ in \S \ref{parGScompl} as a slight modification of the object associated to a presheaf, in \S \ref{parnewdif} we define the new differential
\begin{equation}\label{newdifin}
	d =d_0+d_1+\dots+d_n:\CC_{\mathrm{GS}}^{n-1}(\aaa,M)\lra \CC_{\mathrm{GS}}^n(\aaa,M)
\end{equation}
where $d_0=d_\mathrm{Hoch},d_1=(-1)^nd_\mathrm{simp}$ and the higher $d_j$ are defined in \eqref{defnewdif}. The new differential is shown to square to zero in Theorem \ref{propnewdif}. The definition of the individual components makes essential use of shuffle products. In the self contained \S \ref{parshuffle}, we give a detailed description of the natural action of shuffle permutations on nerves of categories. 

In order to properly relate the GS cohomology to deformation theory, we have to turn to the complex of normalized reduced cochains, which is introduced in \S \ref{parnormred}  as a subcomplex of the GS complex and shown to be quasi-isomorphic to the latter in Propositions \ref{inclusion1}, \ref{normalreducefinal}.
 Finally, in \S \ref{pardefo} generalizing \cite[Thm 2.21]{dinhvanliulowen}, in Theorem \ref{proptwist} we prove that $HH^2_{\mathrm{GS}}$ classifies first order deformations of $\aaa$ as a prestack.

\subsection{Shuffle products}\label{parshuffle}
In this section, we discuss the natural action of shuffle permutations on nerves of categories.
Let $S_n$ be the symmetric group of permutations of $\{1, \dots, n\}$. For $n_i \geq 0$ with $\sum_{i = 1}^k n_i = n$, a permutation $\beta \in S_n$ is an \emph{$(n_i)_i$-shuffle} if the following holds:
for $1 \leq i \leq k$ and $\sum_{j = 1}^{i -1} n_j +1 \leq x \leq y \leq \sum_{j = 1}^i n_j$ we have $\beta(x) \leq \beta(y)$.
The permutation is a \emph{conditioned $(n_i)_i$-shuffle} if moreover we have 
$$\beta(\sum_{i = 1}^{l-1} n_i +1) \leq \beta(\sum_{i = 1}^{l} n_i + 1)$$ for all $1 \leq l \leq k-1$.
Let $S_{(n_i)_i} \subseteq S_n$ be the subset of all $(n_i)_i$-shuffles and $\bar{S}_{(n_i)_i} \subseteq S_{(n_i)_i}$ the subset of conditioned $(n_i)_i$-shuffles.
For any set $X$, $S_n$ obviously has an action of $X^n$. For $\beta \in S_n$ and $(x_1, \dots, x_n) \in X^n$, we define
$$\beta^{(0)}(x_1, \dots, x_n) = (x_{\beta^{}(1)}, \dots, x_{\beta^{}(n)}).$$
When working with $(n_i)_i$-shuffles, we will often consider different sets $X_i$ for $1 \leq i \leq k$ and elements $x^i = (x^i_1, \dots x^i_{n_i}) \in (X_i)^{n_i}$ for $1 \leq i \leq k$. Thus, for a permutation $\beta$, we obtain the \emph{formal shuffle} by $\beta$ of $(x^i)_i$:
\begin{equation}\label{action1}
\beta^{(0)}( (x^i_1, \dots x^i_{n_i})_i) = \beta^{(0)}(x^1_1, \dots, x^1_{n_1}, \dots, x^k_1, \dots x^k_{n_k}) \in (\coprod_{i = 1}^k X_i)^n.
\end{equation}
For instance, for $k = 2$, $\beta \in S_{m,n}$, $x = (x_1, \dots x_m) \in X^m$ and $y = (y_1, \dots y_n) \in Y^n$, we denote the formal shuffle by $\beta$ of $(x,y)$ by:
$$x \underset{\beta}{*} ^{(0)}y = \beta^{(0)}(x,y) = \beta^{(0)}(x_1, \dots, x_m, y_1, \dots, y_n).$$
In the remainder of this section, we discuss the action of shuffle permutations on nerves of categories. Consider categories $\aaa_i$ for $1 \leq i \leq k$. We now refine action \eqref{action1} to obtain a \emph{shuffle} action 
\begin{equation}\label{action2}
S_{(n_i)_i} \times \prod_{i = 1} ^k \nnn_{n_i}(\aaa_i) \lra \nnn_{n}(\prod_{i = 1}^k \aaa_i).
\end{equation}

Consider $\beta \in S_{(n_i)_i}$ and
$$a^i = (\xymatrix{{A_0^i} \ar[r]^-{a^i_{n_i}} & {A_1^i} \ar[r]^-{a^i_{n_i -1}} & {\dots}  \ar[r]^-{a^i_2} & {A^i_{n_i -1}} \ar[r]^-{a_1^i} & {A_{n_i}^i}}) \in \nnn_{n_i}(\aaa_i).$$
Note that it may occur that $n_i = 0$ and $a^i = A^i_0 \in \nnn_0(\aaa_i) = \Ob(\aaa_i)$.
For the associated elements $\underline{a}^i = (a^i_1, a^i_2, \dots, a^i_{n_i -1}, a^i_{n_i}) \in \Mor(\aaa_i)^{n_i}$, we obtain the formal shuffle
$\underline{b} = \beta^{(0)}((\underline{a}^i)_i) = (\underline{b}_{1}, \dots, \underline{b}_n)$. We now inductively associate to $\underline{b}$ an element 
$$b = \beta((a^i)_i) \in \nnn_{n}(\prod_{i = 1}^k \aaa_i)$$
 with source $\prod_{i = 1}^k A^i_0$ and target $\prod_{i = 1}^k A^i_{n_i}$. Then $b$ is called the \emph{shuffle product} by $\beta$ of $(a^i)_i$, and $\underline{b}$ is called the {\em formal sequence} of $b$.\\
 Since $\beta$ is a shuffle permutation, we have $\underline{b}_1 = a^j_1: A^j_{n_j -1} \lra A^j_{n_j}$ for some $1 \leq j \leq k$. We put $B_n = \prod_{i = 1}^k A^i_{n_i}$,
$B_{n-1} = A^1_{n_1} \times \dots \times A^j_{n_j -1} \times \dots \times A^k_{n_k}$ and
$$b_1 = (1_{A^1_{n_1}}, \dots, a^j_1, \dots, 1_{A^k_{n_k}}): B_{n-1} \lra B_n.$$
Now suppose
$$\hat{b}_l = (\xymatrix{ {B_{n-l}} \ar[r]^-{{b}_l} & {B_{n-l+1}} \ar[r]^-{{b}_{l-1}} & {\dots} \ar[r]^-{{b}_2} & {B_{n-1}} \ar[r]^-{{b}_1} & {B_n}}) \in \nnn_l(\prod_{i = 1}^k \aaa_i)$$
has been defined with $B_{n-l} = \prod_{i=1}^k B^i_{n-l}$ and $B^i_{n-l} = A^i_{n_i -\alpha_i}$ where $\alpha_i = \max\{ t \,\, |\,\, a^i_t \in \{ \underline{b}_1, \dots, \underline{b}_l\} \}$.
It then follows that $\underline{b}_{l+1} = a^j_{\alpha_j +1}$ for some $1 \leq j \leq k$ and we put $B_{n-l-1} = A^1_{n_1 - \alpha_1} \times \dots \times A^j_{n_j - \alpha_j -1} \times \dots \times A^k_{n_k - \alpha_k}$ and

\begin{align}\label{shuffleelements}
b_{l + 1} = (1_{A^1_{\alpha_1}}, \dots, a^j_{\alpha_j + 1}, \dots, 1_{A^k_{\alpha_k}}): B_{n-l -1} \lra B_{n-l}.
\end{align}
We thus arrive at the element $$b = \beta^{}((a^i)_i) = \hat{b}_n =(b_1,\dots,b_n)\in \nnn_{n}(\prod_{i = 1}^k \aaa_i)$$ which concludes the definition of \eqref{action2}. 

\begin{remark}\label{remshuffle}
Suppose $\aaa$ is a category and $\phi: \prod_{i = 1}^k \aaa_i \lra \aaa$ is a functor. We naturally obtain an induced map
$\nnn_n(\prod_{i =1}^k \aaa_i) \lra \nnn_n(\aaa)$ which upon composition with \eqref{action2} gives rise to a \emph{$\phi$-shuffle} action
\begin{equation}\label{action3}
S_{(n_i)_i} \times \prod_{i = 1} ^k \nnn_{n_i}(\aaa_i) \lra \nnn_{n}(\aaa): (\beta, (a^i)_i) \longmapsto \beta^{(\phi)}((a^i)_i).
\end{equation}
Obviously, taking $\phi = 1_{\prod_{i = 1}^k \aaa_i}$, we recover the shuffle action \eqref{action2}. If $\phi$ is understood from the context, it will be omitted from the notation.
\end{remark}

\begin{example}\label{exshuffle}
Let $\AAA$ and $\BBB$ be small categories and put $\aaa_1 = \Fun(\AAA, \BBB)$, $\aaa_2 = \AAA$, $\aaa = \BBB$ and $$\phi: \Fun(\AAA, \BBB) \times \AAA \lra \BBB: (F, A) \longmapsto F(A).$$
Consider $a = (a_1: A_1 \lra A_0) \in \nnn_1(\AAA)$ and 
$$\epsilon = (\xymatrix{ {T_0} \ar[r]^{\epsilon_2} & {T_1} \ar[r]^{\epsilon_1} & {T_2} }) \in \nnn_2(\Fun(\AAA, \BBB)).$$
The three elements in $S_{2,1}$ correspond to the following three formal shuffles of $\underline{\epsilon}$ and $\underline{a}$:
$(a,\epsilon_1,\epsilon_2),(\epsilon_1,a,\epsilon_2)$ and $(\epsilon_1,\epsilon_2,a)$. The three corresponding shuffles in $\nnn_3(\Fun(\AAA, \BBB) \times \AAA)$ according to \eqref{action2} are given by:\\
$\xymatrix{ {T_0 \times A_0}\ar[r]^-{\epsilon_2 \times 1_{A_0}}&{T_1 \times A_0}\ar[r]^-{\epsilon_1 \times 1_{A_0}}&{T_2 \times A_0}\ar[r]^-{1_{T_2} \times a}&{T_2 \times A_1}  };\\ 
\xymatrix{ {T_0 \times A_0}\ar[r]^-{\epsilon_2 \times 1_{A_0}}&{T_1\times A_0}\ar[r]^-{1_{T_1} \times a}&{T_1 \times A_1}\ar[r]^-{\epsilon_1 \times 1_{A_1}}&{T_2 \times A_1}
}; \\ 
\xymatrix{ {T_0 \times A_0}\ar[r]^-{1_{T_0} \times a}&{T_0 \times A_1}\ar[r]^-{\epsilon_2 \times 1_{A_1}}&{T_1 \times A_1}\ar[r]^-{\epsilon_1 \times 1_{A_1}}&{T_2 \times A_1}
}.$

The three corresponding $\phi$-shuffles in $\nnn_3(\BBB)$ according to \eqref{action3} are given by:\\
$\xymatrix{ {T_0(A_0)}\ar[r]^-{\epsilon_2(A_0)}&{T_1{(A_0)}}\ar[r]^-{\epsilon_1(A_0)}&{T_2(A_0)}\ar[r]^-{T_2(a)}&{T_2(A_1)}  };\\ \xymatrix{ {T_0(A_0)}\ar[r]^-{\epsilon_2(A_0)}&{T_1{(A_0)}}\ar[r]^-{T_1(a)}&{T_1(A_1)}\ar[r]^-{\epsilon_1(A_1)}&{T_2(A_1)}
}; \\ \xymatrix{ {T_0(A_0)}\ar[r]^-{T_0(a)}&{T_0{(A_1)}}\ar[r]^-{\epsilon_2(A_1)}&{T_1(A_1)}\ar[r]^-{\epsilon_1(A_1)}&{T_2(A_1)}
}.$

\end{example}

\begin{remark}\label{shufflecompositionfunctor}
	Suppose that $\aaa_i=\mathsf{Fun}(\BBB_{k-i},\BBB_{k-i+1})$. Applying the natural composition of functors to each element $b_{l+1}$ in (\ref{shuffleelements}), we obtain
		\begin{align}
		b_{l+1}'= A^1_{\alpha_1}\circ \dots\circ a^j_{\alpha_j+1}\circ\dots \circ A^k_{\alpha_k} :  B'_{n-l-1}\lra B'_{n-l}
		\end{align} 
where $B_{n-l-1}' = A^1_{n_1 - \alpha_1}\circ \dots \circ A^j_{n_j - \alpha_j -1}  \circ \dots \circ A^k_{n_k - \alpha_k}$. Concatenating these morphisms, we obtain the simplex
$$\hat{b}'_n=(b'_1,\dots,b'_n)\in\nnn_n(\mathsf{Fun}(\BBB_0,\BBB_k)).$$ 
	
\end{remark}

\begin{example}\label{exshuffle2}
	Consider
	$$\epsilon = (\xymatrix{ {T_0} \ar[r]^{\epsilon_2} & {T_1} \ar[r]^{\epsilon_1} & {T_2} }) \in \nnn_2(\Fun(\BBB_0, \BBB_1))$$ and
	$$\xi = (\xymatrix{ {S_0} \ar[r]^{\xi} & {S_1}  }) \in \nnn_1(\Fun(\BBB_1, \BBB_2)).$$
	The shuffle products of $\xi$ and $\epsilon$ with respect to composition of functors corresponding to the formal sequences $(\xi\epsilon_1,\epsilon_2), (\epsilon_1,\xi,\epsilon_2), (\epsilon_1,\epsilon_2,\xi)$ are\\
	$\xymatrix{ {S_0T_0}\ar[r]^-{S_0\epsilon_2}&{S_0T_1}\ar[r]^-{S_0\epsilon_1}&{S_0T_2}\ar[r]^-{\xi T_2}&{S_1T_2}  };\\ \xymatrix{ {S_0T_0}\ar[r]^-{S_0\epsilon_2}&{S_0T_1}\ar[r]^-{\xi T_1}&{S_1T_1}\ar[r]^-{S_1\epsilon_1 }&{S_1T_2}  
	}; \\ \xymatrix{ {S_0T_0}\ar[r]^-{\xi T_0}&{S_1T_0}\ar[r]^-{S_1\epsilon_2}&{S_1T_1}\ar[r]^-{S_1\epsilon_1}&{S_1T_2}  
}.$

\end{example}

Now suppose $\beta \in \bar{S}_{(n_i)_i}$ is a conditioned shuffle. In this case it is possible to adapt the inductive procedure we just described in order to arrive at the datum, for $(a^i)_i$ as before, of a sequence
\begin{equation} \label{action3}
(\hat{c}_1, \dots, \hat{c}_k) \in \prod_{l = 1}^k \nnn_{\gamma_l}(\prod_{i = 1}^l \aaa_i)
\end{equation}
where the numbers $\gamma_l$ are determined by $\beta$ and satisfy $\sum_{l = 1}^k \gamma_l = n$. We put $\phi = 1$ and suppress it in the notations (the adaptation to arbitrary $\phi$ is easily made and will be used in the paper). Since $\beta$ is a conditioned shuffle, there are uniquely determined numbers $\gamma_l$ such that $\underline{b}_1 = a^1_1$, $\underline{b}_{\gamma_1 +1} = a^2_1$, \dots, $\underline{b}_{\sum_{i = 1}^l \gamma_i + 1} = a^{l+1}_1$, \dots, 
$\underline{b}_{\sum_{i = 1}^{k-1} \gamma_i + 1} = a^k_1$ and $\gamma_k = n - \sum_{i = 1}^{k-1} \gamma_i$. For $1 \leq l \leq k$ we now have that
for every $\sum_{i = 1}^{l-1} \gamma_i + 1 \leq \rho \leq \sum_{i = 1}^{l} \gamma_i$ there exists $1 \leq j \leq l$ and $t$ with $\underline{b}_{\rho} = a^j_t$. 
Moreover, for fixed $j$, there exists
$$a^{j, l} = (\xymatrix{ {A^j_{n_j -t +1 - m^l_j}} \ar[r]^-{a^j_{t + m^l_j -1}} & {\dots} \ar[r]^-{a^j_t} & {A^j_{n_j -t +1}}}) \in \nnn_{m^l_j}(\aaa_j)$$
such that the morphisms $a^j_s$ occuring in $a^{j,l}$ coincide precisely with the elements  occurring as $\underline{b}_{\rho}$ for $\sum_{i = 1}^{l-1} \gamma_i + 1 \leq \rho \leq \sum_{i = 1}^{l} \gamma_i$. Here we make the convention that if no $a^j_z$ occurs as such  $\underline{b}_{\rho}$, we put $a^{j,l} \in \nnn_0(\aaa_j)$ equal to the domain of $a^{j, l-1}$, or equal to $a^{j, l-1}$ in case $a^{j, l-1} \in \nnn_0(\aaa_j)$. We have $\sum_{j = 1}^l m^l_j = \gamma_l$.
As a consequence, there is a unique $\beta_l \in S_{(m^l_j)_j}$ such that 
$$\beta_l^{(0)}((\underline{a}^{j,l})_j) = (\underline{b}_{\rho})_{\sum_{i = 1}^{l-1} \gamma_i + 1 \leq \rho \leq \sum_{i = 1}^{l} \gamma_i}.$$
In \eqref{action3} we now put $\hat{c}_l = \beta_l((a^{j,l})_j) \in \nnn_{\gamma_l}(\prod_{i = 1}^l \aaa_i).$

\subsection{The Gerstenhaber-Schack complex}\label{parGScompl}
 Let $\uuu$ be a small category, $\aaa$ a prestack on $\uuu$ and $M$ a bimodule over $\aaa$. Let $\nnn(\uuu)$ denote the simplicial nerve of the small category  $\uuu$. Our standard notation for a $p$-simplex $\sigma \in \nnn(\uuu)_p$ is
 \begin{equation}\label{eqsigma0}
 \sigma = (\xymatrix{ {U_0} \ar[r]_-{u_1} & {U_{1}} \ar[r]_-{u_2} & {\dots} \ar[r]_-{u_{p-1}} & {U_{p-1}} \ar[r]_{u_p} & {U_p}}).
 \end{equation}
 
 If confusion can arise, we write $U_i = U^{\sigma}_i$ and $u_i = u^{\sigma}_i$ instead. We also write $\sigma=(u_1,\dots,u_p)$ for short.
 
 For $\sigma \in \nnn_p(\uuu)$, we obtain a functor
 $$\sigma^{\ast} = (u_p \dots u_2 u_1)^{\ast}: \aaa(U_p) \lra \aaa(U_0)$$
 and a functor
 $$\sigma^{\star} = u_1^{\ast}u_2^{\ast} \dots u_p^{\ast}: \aaa(U_p) \lra \aaa(U_0).$$

For each $1\le k\le p-1$, denote by $L_k(\sigma)$ and $R_k(\sigma)$ the following simplices $$\begin{matrix}L_k(\sigma)=(\xymatrix{ {U_0} \ar[r]_-{u_1} & {U_{1}} \ar[r]_-{u_2} & {\dots} \ar[r]_-{u_{p-1}} & {U_{k-1}} \ar[r]_{u_k} & {U_k}})\\ R_k(\sigma)=(\xymatrix{ {U_{k}} \ar[r]_-{u_{k+1}} & {U_{k+2}} \ar[r]_-{u_{k+2}} & {\dots} \ar[r]_-{u_{p-1}} & {U_{p-1}} \ar[r]_{u_p} & {U_p}})\end{matrix}$$
 
 We consider the following natural isomorphisms:
 
 \begin{equation}\label{cc}
 c^{\sigma,k} = c^{u_k\cdots u_1,u_{p}\cdots u_{k+1}}: (L_k\sigma)^{\ast} (R_k\sigma)^{\ast} \lra \sigma^{\ast}
 \end{equation}
 
 \begin{equation}\label{epsilon}
 \epsilon^{\sigma,k} = {u_1^{\ast}} \cdots u_{k-1}^*c^{u_k, u_{k+1}}u_{k+2}^* \cdots {u_p^{\ast}}: \sigma^{\star} \lra u_1^*\cdots(u_{k+1}u_{k})^*\cdots u_p^*
 \end{equation}
 For $A\in\mathrm{Ob}(\aaa(U_p))$, we write $c^{\sigma,k,A}=c^{\sigma,k}(A)$ and
 $\epsilon^{\sigma,k,A}=\epsilon^{\sigma,k}(A)$.
 
 For the category $\aaa(U)$, $U \in \uuu$, we use the following standard notation for a $q$-simplex $a \in \nnn(\aaa(U))_q$:
 \begin{equation}\label{equationa0}
     a=(\xymatrix{{A_0}\ar[r]^-{a_q}&{A_1} \ar[r]^-{a_{q-1}}&\cdots \ar[r]^-{a_2}&A_{q-1}\ar[r]^-{a_1}& A_q  }).
\end{equation}
We also write $a=(a_1,\dots,a_q)$ for short.

Let
	$$\CC^{\sigma, A}(\aaa,M) = \Hom_k\big(\aaa(U_p)(A_{q-1}, A_{q}) \otimes \dots \otimes \aaa(U_p)(A_{0}, A_{1}),M^{U_0}( \sigma^{\star}A_0,\sigma^{\ast}A_q)\big).$$
	and  put
	$$\CC^{\sigma, q}(\aaa,M) =  \prod_{A \in \aaa(U_p)^{q+1}} \CC^{\sigma, A}(\aaa,M),$$
	$$\CC^{p,q}(\aaa,M)=\prod_{\sigma\in\nnn_p(\uuu)}\CC^{\sigma,q}(\aaa,M).$$
Then we obtain the double object
\begin{equation}\label{double}
	\CC_\mathrm{GS}^n(\aaa,M)=\prod_{p+q=n}\CC^{p,q}(\aaa,M)
	\end{equation}
The usual Hochschild differential defines vertical maps
	$$d_{\mathrm{Hoch}}: \CC^{p, q-1}(\aaa) \lra \CC^{p, q}(\aaa).$$
Precisely, given $(\phi^\sigma)_\sigma\in\CC^{p,q}(\aaa,M)$, for each $p$-simplex $\sigma$ and for $(a_1,\dots,a_q)\in \aaa(U_p)(A_{q-1}, A_{q}) \otimes \dots \otimes \aaa(U_p)(A_{0}, A_{1})$, then we have
	$$(d_\mathrm{Hoch}\phi)^\sigma(a_1,\dots,a_q)=\displaystyle\sum_{i=0}^q(-1)^i(d_\h^i\phi)^\sigma(a_1,\dots,a_q)$$
where 
	$$(d_\mathrm{Hoch}^i\phi)^\sigma(a_1,\dots,a_q)=
		\begin{cases}
			\sigma^*(a_1)\phi^\sigma(a_2,\dots,a_q)& \text{ if } i=0 \\
			\phi^\sigma(a_1,\dots,a_{i}a_{i+1},\dots, a_q) &\text{ if } 1\le i\le q-1 \\ 
			\phi^\sigma(a_{q-1},\dots,a_1)\sigma^\star(a_q) & \text{ if } i=q .
		\end{cases}
	$$
We also write $\phi^\sigma(d_\mathrm{Hoch}^i(a_q,\dots,a_1))$ instead of $(d_\mathrm{Hoch}^i(\phi))^\sigma(a_q,\dots,a_1)$.\\

As a part of the simplicial structure of $\nnn(\uuu)$, we have maps
	$$\partial_i: \nnn_{p+1}(\uuu) \lra \nnn_p(\uuu): \sigma \longmapsto \partial_i \sigma$$
for $i = 0, 1, \dots, p+1$. For 
	$\sigma = (\xymatrix{ {U_0} \ar[r]_-{u_1} & {U_{1}} \ar[r]_-{u_2} & {\dots} \ar[r]_-{u_{p}} & {U_{p}} \ar[r]_{u_{p+1}} & {U_{p+1}}})$, 
we have
	$$\partial_{p+1}\sigma = (\xymatrix{ {U_0} \ar[r]_-{u_1} & {U_{1}} \ar[r]_-{u_2} & {\dots} \ar[r]_-{u_{p-1}} & {U_{p-1}} \ar[r]_{u_p} & {U_p}})$$
	$$\partial_{0}\sigma = (\xymatrix{ {U_1} \ar[r]_-{u_2} & {U_{2}} \ar[r]_-{u_3} & {\dots} \ar[r]_-{u_{p}} & {U_{p}} \ar[r]_{u_{p+1}} & {U_{p+1}}})$$
and
	$$\partial_i\sigma = (\xymatrix{ {U_{0}} \ar[r]_{u_1} & {\dots} \ar[r] & {U_{i-1}} \ar[r]_{u_{i+1} u_{i}} & {U_{i+1}} \ar[r] & {\dots} \ar[r]_{u_{p+1}} & {U_{p+1}}})$$
for $i = 1, \dots, p$. Each $\partial_i$ gives rise to a map
$$d_\mathrm{simp}^i: \CC^{p-1,q}(\aaa,M) \lra \CC^{p,q}(\aaa,M)$$
given by 
$$(d_\mathrm{simp}^i(\phi))^\sigma(a_1,\dots,a_q):=\begin{cases}
c^{\sigma,1,A_q}M^{u_1}\phi^{\partial_0\sigma}(a_1,\dots,a_q)  & \text{ if }\ \  i=0\\
\phi^{\partial_i\sigma}(a_1,\dots,a_q)\epsilon^{\sigma,i,A_0} & \text{ if }\ \  1\le i\le p\\
c^{\sigma,p-1,A_q}\phi^{\partial_p\sigma}(u_p^*a_1,\dots,u_p^*a_q)  & \text{ if }\ \  i=p.\\
\end{cases}$$
Hence we obtain the horizontal maps $$d_{\mathrm{simp}} = \sum_{i = 0}^{p} (-1)^i d_{\mathrm{simp}}^i : \CC^{p-1,q}(\aaa,M) \lra \CC^{p,q}(\aaa,M).$$
We define the maps 
$$ d_{\mathrm{GS}} = d_{\mathrm{Hoch}} + (-1)^nd_{\mathrm{simp}}:\CC^{n-1}(\aaa,M)\lra\CC^{n}(\aaa,M).$$

Now if $c^{u,v}=1$ for all $u:V\lra U, v:W\lra V$, then $\aaa$ is a presheaf of $k$-linear categories. It is easy to check that $d_\mathrm{Hoch}^2=d_\mathrm{simp}^2=d_\mathrm{Hoch}d_\mathrm{simp}-d_\mathrm{simp}d_\mathrm{Hoch}=0$, so $d_\mathrm{GS}^2=0$. 
In analogy with \cite{gerstenhaberschack1}, if $k$ is a field one shows that $(\CC^\bullet(\aaa,M),d_\mathrm{GS})$ computes $\Ext$ groups of bimodules: $$HH^n_{\mathrm{GS}}(\aaa,M) = H^n(\CC^\bullet(\aaa,M),d_\mathrm{GS})=\Ext_{\aaa-\aaa}^n(\aaa,M).$$  
Moreover, by analogous computations as in \cite[\S 2.21]{dinhvanliulowen}, it is seen that the second cohomology group $HH^2_{\mathrm{GS}}(\aaa)$ naturally controls the first order deformations of the presheaf $\aaa$ as a prestack. 

\subsection{The new differential}\label{parnewdif}
When $\aaa$ is a prestack with non-trivial twists $c^{u,v}$, then for $d_{\mathrm{GS}}$ defined as in the previous section, we have $d_\mathrm{GS}^2\ne0$ because $d_\mathrm{simp}^2\ne 0$. To fix this problem we add new components to $d_\mathrm{GS}$ to obtain the new  differential 
\begin{equation}\label{newdifin}
d=d_0+d_1+\dots+d_n:\CC_{\mathrm{GS}}^{n-1}(\aaa,M)\lra \CC_{\mathrm{GS}}^n(\aaa,M)
\end{equation}
where $d_0=d_\mathrm{Hoch},d_1=(-1)^nd_\mathrm{simp}$ as above. 
The cohomology with respect to the new differential is denoted $$HH^n_{\mathrm{GS}}(\aaa,M) = H^n\CC^{\bullet}_{\mathrm{GS}}(\aaa,M).$$

Let $\aaa$ be a prestack. Consider a simplex $\sigma=(u_1,\dots,u_n) \in \nnn_n(\uuu)$ with $n \geq 2$. For every $u:V\lra U,\ v:W\lra V$ we have the natural isomorphism $c^{u,v}:v^*u^*\lra(uv)^*$. From these isomorphisms we inductively construct a set $$\ppp(u_1, \dots, u_n) \subseteq \nnn_{n-1}(\Fun(\aaa(U_n), \aaa(U_0)))$$
of simplices $r$ with source $u_1^*u_2^*\cdots u_n^*$ and target $(u_nu_{n-1}\cdots u_1)^*$. 
Our standard notation for a simplex $r$ of natural transformations is
$$r = (\xymatrix{ {T_0} \ar[r]^{r_{n-1}} & {T_1} \ar[r]^{r_{n-2}} & {T_2} \ar[r] & {\dots} \ar[r]^{r_1} & {T_{n-1}} })$$
which is abbreviated to $r = (r_1, \dots, r_{n-1})$. 
Elements of $\ppp(u_1, \dots, u_n)$ are called \emph{paths} from $u_1^*u_2^*\cdots u_n^*$ to $(u_nu_{n-1}\cdots u_1)^*$. Further, we define a \emph{sign} map 
$$\mathrm{sign}: \ppp(u_1, \dots, u_n) \lra \{1, -1\}: r \longmapsto \mathrm{sign}(r).$$

We start with $n=2$. Consider $c^{u_1,u_2}:u_1^*u_2^*\lra (u_2u_1)^*$. We put $\mathcal{P}(u_1,u_2):=\{(c^{u_1,u_2})\}$ and we set $\mathrm{sign}(c^{u_1,u_2})=-1$.

For $n > 2$, given $\sigma=(u_1,\dots,u_n)$, for each $i=1, \dots, n-1$, consider the natural isomorphism
	$\epsilon^{\sigma,i} = {u_1^{\ast}} \cdots c^{u_i, u_{i+1}} \cdots {u_n^{\ast}}$ as defined in \eqref{epsilon}
and put 
 	$$\mathrm{sign}{(\epsilon_i)}=(-1)^i.$$ 
For each path $r=(r_1,\dots,r_{n-2})\in\mathcal{P}(u_1,\dots,u_{i-1},u_{i+1}u_i,u_{i+2},\dots,u_n)$, the simplex $(r_1,\dots,r_{n-2},\epsilon_i)$ is called a \emph{path} from $u_1^*u_2^*\cdots u_n^*$ to $(u_nu_{n-1}\cdots u_1)^*$ and
 $\mathcal{P}(u_1,\dots,u_n)$ is defined to be the set of all such paths. Thus,
	$$\mathcal{P}(u_1,\dots,u_n)=\big\{(r_1,\dots,r_{n-2},\epsilon^{\sigma,i}):\ 1\le i\le n-1 \text{ and } r\in \mathcal{P}(\partial_i\sigma) \big\}.$$
	For a path $r=(r_1,\dots,r_{n-1})$, we define
	  $$(-1)^r\equiv\mathrm{sign}(r)=\prod_{i=1}^{n-1}\mathrm{sign}(r_i).$$
	  For a permutation $\beta \in S_n$, we similarly denote $(-1)^{\beta} \equiv \mathrm{sign}(\beta)$ for the standard sign of permutations and denote $(-1)^{r+\beta}=(-1)^r(-1)^\beta.$

\begin{example}
Given $\sigma=(u_1,u_2,u_3)$, there are two paths from $u_1^*u_2^*u_3^*$ to $(u_3u_2u_1)^*$:
$$\mathcal{P}=\{r=(c^{u_2u_1,u_3},c^{u_1,u_2}u_3^*),\ s=(c^{u_1,u_3u_2},u_1^*c^{u_2,u_3}) \}$$ and $\mathrm{sign}(r)=1,\ \mathrm{sign}(s)=-1$. 
\end{example}

There are $(n-1)!$ paths in $\mathcal{P}(u_1,\dots,u_n)$, for each path $r=(r_1,r_2\dots,r_{n-1})$ denote the isomorphism $||r||=r_1r_2\cdots r_{n-1}$. 
\begin{lemma}\label{absolute}
	Given $n$-simplex $\sigma=(u_1,\dots,u_n)$. Let $r=(r_1,r_2\dots,r_{n-1})$ and $s=(s_1,s_2\dots,s_{n-1})$ be two arbitrary paths in $\mathcal{P}(u_1,\dots,u_n)$. Then $||r||=||s||$. 
\end{lemma}
\begin{proof}
	By the coherence condition (\ref{compatibility}) our lemma is true for $n=3$. For $n>3$,  we assume that $r_{n-1}=\epsilon^{\sigma,i}$ and $s_{n-1}=\epsilon^{\sigma,j}$ for some $i\le j$. If $i=j$ then $r_{n-1}=s_{n-1}$, by induction hypothesis we have $||r||=||s||$. If $i<j$, it is sufficient to prove that $||r||=||t||$ for some path $t=(t_1,\dots,t_{n-1})$ in which $t_{n-1}=\epsilon^{\sigma,i+1}$. Thus, let $h=(h_1,\dots,h_{n-2})$ be a path in  $\mathcal{P}(u_1,\dots,u_{i+1}u_i,\dots,u_n)$ such that $h_{n-2}=u_1^*\cdots u_{i-1}^*c^{(u_{i+1}u_i,u_{i+2})}u_{i+3}^*\cdots u_n^*$, 
	 by the induction hypothesis $$h_1\cdots h_{n-2}=r_1\cdots r_{n-2}.$$
	Let $t_{n-2}=u_1^*\cdots u_{i-1}^*c^{(u_i,u_{i+2}u_{i+1})}u_{i+3}^*\cdots u_n^*$, again by (\ref{compatibility}) we have the commutative diagram 
	$$\xymatrix{
		{u_1^*\cdots u_{i}^*u_{i+1}^*u_{i+2}^*\cdots u_n^*\  }\ar[r]^-{t_{n-1}}\ar[d]^-{r_{n-1}}&{\ \ u_1^*\cdots u_{i}^*(u_{i+2}u_{i+1})^*u_{i+3}^*\cdots u_n^* }\ar[d]^-{t_{n-2}}\\
		{u_1^*\cdots u_{i-1}^*(u_{i+1}u_i)^*u_{i+2}^*\cdots u_n^*\ }\ar[r]^-{h_{n-2}}&{\ u_1^*\cdots u_{i-1}^*(u_{i+2}u_{i+1}u_i)^*u_{i+3}^*\cdots u_n^*}
	}$$
	Choose $t=(h_1,\dots,h_{n-3},t_{n-2},t_{n-1})$, then $||t||=||(h,r_{n-1})||=||r||$.
\end{proof}

Given a simplex $\sigma=(u_1,\dots,u_n)$, let $r=(r_1,\dots,r_{n-1})$ be a path in $\ppp(\sigma)$. For each $1\le k\le n-2$, assume that $r_{k+1}=\epsilon^{\gamma,i}$ for some simplex $\gamma=(v_1,\dots,v_{k+2})$ and  $1\le i\le k+1$. Then $r_{k}=\epsilon^{\partial_i\gamma,j}$ for some $1\le j\le k$. We put 
$$\left[\begin{matrix}
r'_{k+1}=\epsilon^{\gamma,j} \text{ and } r'_{k}=\epsilon^{\partial_j\gamma, i-1} \text{ if } i> j;\\
r'_{k+1}=\epsilon^{\gamma,j+1} \text{ and } r'_{k}=\epsilon^{\partial_j\gamma, i} \text{ if } i\le j.
\end{matrix}\right.$$

Denote by $\mathrm{flip}(r,k)$ the path $(r_1,\dots,r_{k-1},r'_k,r'_{k+1},r_{k+2},\dots,r_{n-1})$ in $\ppp(\sigma)$. It is easy to see that $\mathrm{flip}(\mathrm{flip}(r,k),k)=r$ and
\begin{equation}\label{flipequation2}
\mathrm{sign}(\mathrm{flip}(r,k))=-\mathrm{sign}(r).
\end{equation}
 Due to Lemma \ref{absolute}, we have 
\begin{equation}\label{flipequation1}
r_k'r_{k+1}'=r_k r_{k+1}.
\end{equation}

In the next lemma, the shuffle product of natural transformations is taken with respect to the composition of functors as in Example (\ref{exshuffle2}).

\begin{lemma}\label{jointpaths}
	Given an $n$-simplex $\sigma=(u_1,\dots,u_n)$. Then, 
	\begin{enumerate}
		\item Consider two paths $r=(r_1,\dots,r_{n-k-1})\in\ppp(R_k(\sigma)),s=(s_1,\dots,s_{k-1})\in\ppp(L_k(\sigma))$. For each $\beta\in S_{n-k-1,k-1}$, the simplex $\omega=(c^{\sigma,k},\beta(r,s))$ is a path in $\ppp(\sigma)$. Moreover $$(-1)^\omega=(-1)^{n-k}(-1)^\beta(-1)^r(-1)^s.$$
		\item Consider a path $\omega=(\omega_1,\dots,\omega_{n-1})$ in $\ppp(\sigma)$ in which $\omega_1=c^{\sigma,k}$. There exist unique paths $r=(r_1,\dots,r_{n-k-1})\in\ppp(R_k(\sigma)),\ s=(s_1,\dots,s_{k-1})\in\ppp(L_k(\sigma))$ and $\beta\in S_{n-k-1,k-1}$ such that $\omega=(c^{\sigma,k},\beta(r,s))$.
	 \end{enumerate}
	 
\end{lemma}

 Now we are able to define the components $d_j(j\ge 2)$ of the differential $d$ from \eqref{newdifin} in formula \eqref{defnewdif} below with
 $d_j:\CC_{\mathrm{GS}}^{p,q}(\aaa,M)\lra \CC_{\mathrm{GS}}^{p+j,q+1-j}(\aaa,M)$.
 $$\xymatrix{
 	\CC^{p,q+1}&\\
 	\CC^{p,q}\ar[u]^-{d_0}\ar[r]^-{d_1}\ar[rrd]^-{d_2} \ar[rrrrddd]_{d_{q+1}}&\CC^{p+1,q}\\
 	& &\CC^{p+2,q-1}\\
 	& & & { \cdots}\\
 	& & & & \CC^{p + q + 1,0}
 }$$
 
Consider $\phi\in\CC_{\mathrm{GS}}^{p,q}(\aaa,M)$. Let   $\sigma=(u_1,\dots,u_{p+j})$ be a $(p+j)$-simplex as in $(\ref{eqsigma0})$. Given  $A_0,\dots,A_t\in \mathrm{Ob}(\aaa(U_{p+j}))$ where $t=q+1-j$, let $a=(a_1,\dots,a_t)$ where $a_i\in\aaa(U_p)(A_{t-i},A_{t-i+1})$ as in (\ref{equationa0}). We define 
\begin{equation} \label{defnewdif}
(d_j(\phi))^\sigma(a_1,\dots,a_t)=\sum_{\substack{r\in\mathcal{P}(R_{j}(\sigma))\\\beta\in S_{t,j-1}}}(-1)^r(-1)^\beta(-1)^tc^{\sigma,p,A_t}\phi^{L_p(\sigma)}(\beta(a,r))
\end{equation}
where $\beta(a,r)$ is the shuffle product by $\beta$ of $a=(a_1,\dots,a_t)$ and $r=(r_1,\dots,r_{j-1})$, with respect to the evaluation of functors (see Remark \ref{remshuffle} and Example \ref{exshuffle}).

\begin{theorem}\label{propnewdif}
$d\circ d=0$.
\end{theorem}
\begin{proof}
For $N\ge 2$, for each cochain $\phi\in\cgs{p,q+N-2}$, we show the component of $d(d (\phi))$ which lies in $\cgs{p+N,q}$ is zero. Given a simplex $\sigma=(u_1,\dots,u_{p+N})\in\nnn_{p+N}(\uuu)$ and objects $A_0,A_1,\dots,A_q\in\aaa(U_{p+N})$. Let $a=(a_1,\dots,a_q)$ where $a_i\in\aaa(U_{p+N})(A_{q-i},A_{q-i+1})$ as in (\ref{equationa0}). We show that
$$(d(d\phi))^\sigma(a)=\sum_{i=0}^N(d_{N-i}(d_i\phi))^\sigma(a)=0.$$
This equation is equivalent to 
\begin{equation}{\label{d21}}
(d_\h d_N\phi+d_{N-1}d_1\phi+d_1d_{N-1}\phi+\sum_{i=2}^{N-2}d_{N-i}d_i\phi)^\sigma(a)=-(d_Nd_\h\phi)^\sigma(a).
\end{equation}
By definition we have
\begin{eqnarray*}
-(d_Nd_\h\phi)^\sigma(a)&=&-\sum_{\substack{r\in\ppp(R_p(\sigma))\\ \beta\in S_{q,N-1}}}(-1)^q(-1)^r(-1)^\beta c^{\sigma,p,A_q}(d_\h\phi)^{L_p(\sigma)}(\beta(a,r))\\
&=&\sum_{i=0}^{q+N-1}\sum_{\substack{r\in\ppp(R_p(\sigma))\\ \beta\in S_{q,N-1}}}T(q,r,\beta,i)
\end{eqnarray*}
where  $$T(a,r,\beta,i)=-(-1)^{q+i}(-1)^r(-1)^\beta c^{\sigma,p,A_q}(d_\h^i\phi)^{L_p(\sigma)}(\beta(a,r)).$$

 We prove the equation $(\ref{d21})$ in the following steps: 
\begin{enumerate}
	\item For each term $T_1$ occurring in the expression of $d_\h d_N\phi$, there is a unique term $T(a,r,\beta,i)$ in $-(d_Nd_\h\phi)$ such that $T_1=T(a,r,\beta,i)$.
	\item For $j=2,\dots, (N-2)$, for each term $T_2$ occurring in $d_{N-j}d_j\phi$, there is a unique term $T(a,r,\beta',j')$ in $-(d_Nd_\h\phi)$ such that $T_2=T(a,r,\beta',j')$.
	\item After cancellation, for each term $T_3$ in $d_{N-1}d_1+d_1d_{N-1}$, there is a unique term $T(a,r,\beta,i)$ in $-(d_Nd_\h\phi)$ such that $T_3=T(a,r,\beta,i)$.
	\item After the cancellation with the terms $T_1,T_2,T_3$ as in step 1,2,3, denote $X$ the remaining terms in $-(d_Nd_\h\phi)$, then we show that $X=0$.
\end{enumerate}

{\em Step 1.}
 We have \begin{eqnarray*}
d_\h(d_N\phi)^\sigma(a)&=&\sum_{j=0}^q(-1)^j( d_N\phi)^\sigma(d_\h^j(a))\\
&=&\sum_{i=0}^q\sum_{\substack{r\in\ppp(R_p(\sigma)\\ \beta\in S_{q-1,N-1}}}T_1(d_\h^j(a),\beta,r,j)
\end{eqnarray*}
where 
	$$T_1(d_\h^j(a),r,\beta,j)=(-1)^j(-1)^{q-1}(-1)^r(-1)^\beta c^{\sigma,p,A_q}\phi^{L_p\sigma}(\beta(d_\h^j(a),r)).$$

{\em$\bullet$} Consider $j=1,\dots, q-1$. For each path $r\in\ppp(R_p(\sigma)$, each $\beta\in S_{q-1,N-1}$,  we write the formal sequence $\beta^{(0)}(d_\h^j(a),r)=(\beta_1,\dots,\beta_k,a_ja_{j+1},\beta_{k+2}\dots \beta_{q+N-2})$ for some $k$. There is a unique shuffle permutation $\beta'\in S_{q,N-1}$ such that $$\beta'^{(0)}(a,r)=(\underline{\beta}_1,\dots,\underline{\beta}_k,a_j,a_{j+1},\underline{\beta}_{k+2},\dots ,\underline{\beta}_{q+N-2}).$$
By decomposing $\beta'$ as
	\begin{align*}
		&(a_1,\dots,a_q,r_1,\dots,r_{N-1})\\
		& \overset{}{\lra}(a_1,\dots,a_{j-1},r_1,\dots,r_{k-j +1},a_j,a_{j+1},\dots,a_{q},r_{k-j+2},\dots,r_{N-1})\\
		& \lra (\underline{\beta}_1,\dots,\underline{\beta}_k,a_j,a_{j+1},\underline{\beta}_{k+2},\dots ,\underline{\beta}_{q+N-2}).
	\end{align*}
We see that 
	$$\mathrm{sign}(\beta')=(-1)^{k-j+1}\mathrm{sign}(\beta),$$ 
so we get $$T_1(d_\h^j(a),r,\beta,j)=T(a,r,\beta',k+1).$$

 {\em $\bullet$} Consider $j=0$ or $j=q$. For $j=0$, we have
	$$T_1(d_\h^0(a),\beta,r,0)=(-1)^{q-1}(-1)^r(-1)^\beta \sigma^*(a_1)c^{\sigma,p,A_{q-1}}\phi^{L_p}(\beta(a_2,\dots,a_q;r)).$$
Upon writing the formal sequence $\beta^{(0)}(a_2,\dots,a_q;r)=(\underline{\beta}_1,\dots,\underline{\beta}_{N+q-2})$, there is  a unique $\beta'\in S_{q,N-1}$ such that $\beta'^{(0)}(a,r)=(a_1,\underline{\beta}_1,\dots,\underline{\beta}_{q+N-2})$, and thus 
\begin{eqnarray*}
 T(a,r,\beta',0)&=&T_1(d_\h^0(a),\beta,r,0). \end{eqnarray*}
For $j=q$, we have
$$T_1(d_\h^q(a),\beta,r,0)=-(-1)^r(-1)^\beta c^{\sigma,p,A_{q}}\phi^{L_p}(\beta(a_2,\dots,a_q;r))\sigma^\star(a_q).$$
Assume that $\beta^{(0)}(a_1,\dots,a_{q-1};r)=(\underline{\beta}_1,\dots,\underline{\beta}_{N+q-2})$, there is a unique $\beta'\in S_{q,N-1}$ such that $\beta'^{(0)}(a,r)=(\underline{\beta}_1,\dots,\underline{\beta}_{q+N-2},a_q)$. We have 
\begin{eqnarray*}
	T(a,r,\beta',q+N)&=& T_1(d_\h^q(a),\beta,r,q).
\end{eqnarray*}

{\em Step 2.}
We write $$\sigma=(u_1,\dots,u_p,\dots,u_{p+N-j},\dots,u_{p+N}).$$ 
Let $\Delta=(u_1,\dots,u_p,\dots,u_{p+N-j})=L_{p+N-j}(\sigma)$. By definition, we have
	\begin{align*}
	(d_j(d_{N-j}\phi))^\sigma(a) &=\sum_{\substack{r\in\ppp(R_{p+N-j}(\sigma))\\\beta\in S_{q,j-1}}}(-1)^q(-1)^r(-1)^\beta c^{\sigma,p+N-j,A_q}(d_{N-j}\phi)^{L_{p+N-j}(\sigma)}(\beta(a,r))\\
 & =\sum_{\substack{r\in\ppp(R_{p+N-j}(\sigma))\\\beta\in S_{q,j-1}}}\sum_{\substack{s\in\ppp(R(\Delta,p))\\\gamma\in S_{q+j-1,N-j-1}}}T_2(a,r,\beta,s,\gamma)
	\end{align*}

where
	 $$T_2(a,r,\beta,s,\gamma)=(-1)^{j-1}(-1)^{r+s+\beta+\gamma} c^{\sigma,p+N-j,A_q}c^{\Delta,p,(R_{p+N-j}(\sigma))^*A_q}\phi^{L_p(\Delta)}(\gamma(\beta(a,r),s)).$$
The shuffle product is associative, hence 
	$$\gamma(\beta(a,r),s)=\beta(a,\gamma(r,s)).$$
Let $c_0=c^{\sigma,p+1}$, by Lemma \ref{jointpaths}, we have $\omega=(c_0,\gamma(r,s))$ is a path in $\ppp(R_p(\sigma))$. There is a unique $\beta'\in S_{q,N-1}$ such that
	 $$\beta'(a,\omega)=(c_0(A_q),\beta(a,\gamma(r,s))).$$
	  By computation
$(-1)^{j-1}(-1)^{r+s+\beta+\gamma}=-(-1)^q(-1)^{\omega+\beta'}$, by coherence (\ref{compatibility})
	  $$c^{\sigma,p+N-j,A_q}c^{\Delta,p,(R_{p+N-j}(\sigma))^*A_q}=c^{\sigma,p,A_q}c_0(A_q),$$
so we get
	\begin{align*}
		T(a,\omega,\beta',0)&=-(-1)^{q}(-1)^{\omega+\beta'} c^{\sigma,p,A_q}(d_\h^0\phi)^{L_p(\sigma)}\big(c_0(A_q),\beta(a,\gamma(r,s))\big)\\
		&=T_2(a,r,\beta,s,\gamma).
	\end{align*}
 
{\em Step 3.}
By definition we have 
\begin{align*}
	&(d_{N-1}(d_1\phi))^\sigma(a)\\
	&=\sum_{{\substack{r\in\ppp(R_{p+1}(\sigma))\\\beta\in S_{q,N-2}}}}(-1)^{q}(-1)^{r+\beta} c^{\sigma,p+1,A_q}((-1)^{p+q+N-1}d_\s\phi)^{L_{p+1}(\sigma)}(\beta(a,r))
	\\
	&=\sum_{{\substack{r\in\ppp(R_{p+1}(\sigma))\\\beta\in S_{q,N-2}}}}(-1)^{p+N-1}(-1)^{r+\beta}c^{\sigma,p+1,A_q}\bigg(d_\s^0\phi\\
	&+\sum_{i=1}^p(-1)^id_\s^i\phi+(-1)^{p+1}d_\s^{p+1}\phi\bigg)^{L_{p+1}(\sigma)}(\beta(a,r))
	\\
	&=\sum_{{\substack{r\in\ppp(R_{p+1}(\sigma))\\\beta\in S_{q,N-2}}}}\big( B(a,r,\beta)+\sum_{i=1}^{p} C(a,r,\beta,i)+D(a,r,\beta) \big)
\end{align*}
where 
\begin{eqnarray*}
    B(a,r,\beta)&=&(-1)^{p+N-1}(-1)^{r+\beta} c^{\sigma,p+1,A_q}c^{L_{p+1}(\sigma),1,(R_{p+1}(\sigma))^*A_q} M^{u_1}\big(\\ &&\phi^{\partial_0(L_{p+1}(\sigma))}(\beta(a,r))\big);\\
        C(a,r,\beta,i)&=&(-1)^{p+N+i-1}(-1)^{r+\beta} c^{\sigma,p+1,A_q}\phi^{\partial_iL_{p+1}(\sigma)}(\beta(a,r))\epsilon^{L_{p+1}(\sigma),i,(R_{p+1}(\sigma))^\star A_0};\\
    D(a,r,\beta)&=&(-1)^{N}(-1)^{r+\beta} c^{\sigma,p+1,A_q}c^{L_{p+1}(\sigma),p,(R_{p+1}(\sigma))^*A_q}\\
    &&\phi^{\partial_{p+1}L_{p+1}(\sigma)}(u_{p+1}^*(\beta(a,r))).
\end{eqnarray*}
On the other hand, we have
\begin{align*}
	&(d_{1}(d_{N-1}\phi))^\sigma(a)=(-1)^{p+q+N}(d_\s(d_{N-1}\phi))^\sigma(a)
	\\
	&=(-1)^{p+q+N}\bigg((d_\s^0(d_{N-1}\phi))^\sigma(a)+\sum_{i=1}^{p}(-1)^{i}(d_\s^i(d_{N-1}\phi))^\sigma(a)\\
	&+\sum_{i=p+1}^{p+N-1}(-1)^{i}(d_\s^i(d_{N-1}\phi))^\sigma(a)+(-1)^{p+N}(d_\s^{p+N}(d_{N-1}\phi))^\sigma(a)\bigg)
	\\
	&=\sum_{\substack{r\in\ppp(R_{p+1}(\partial_0\sigma))\\\beta\in S_{q,N-2}}}B'(a,r,\beta)+\sum_{i=1}^{p}\sum_{\substack{r\in\ppp(R_{p+1}(\partial_i\sigma))\\\beta\in S_{q,N-2}}}C'(a,r\
	,\beta,i)\\
	&+\sum_{i=p+1}^{p+N-1}\sum_{\substack{r\in\ppp(R_{p}(\partial_i\sigma))\\\beta\in S_{q,N-2}}}C''(a,r\
	,\beta,i)+\sum_{\substack{r\in\ppp(R_{p}\partial_{p+N}\sigma)\\\beta\in S_{q,N-2}}}D'(a,r,\beta)
\end{align*}
where 
\begin{eqnarray*}
	B'(a,r,\beta)&=&(-1)^{p+N}(-1)^{r+\beta} c^{\sigma,1,A_q}u_1^*(c^{\partial_0\sigma,p+1,A_q})M^{u_1}\left( \phi^{L_{p+1}(\partial_0\sigma)}(\beta(a,r))\right) ;
	\\
	C'(a,r,\beta,i)&=&(-1)^{p+N+i}(-1)^{r+\beta} c^{\partial_i\sigma,p+1,A_q}\phi^{L_{p+1}(\partial_i\sigma)}(\beta(a,r)) \epsilon^{\sigma,i,A_0} ;
	\\
	C''(a,r,\beta,i)&=&(-1)^{p+N+i}(-1)^{r+\beta} c^{\partial_i\sigma,p,A_q}\phi^{L_p(\partial_i\sigma)}(\beta(a,r)) \epsilon^{\sigma,i,A_0} ;
	\\
	D'(a,r,\beta)&=&(-1)^{r+\beta}c^{\sigma,p+N-1,A_q}c^{\partial_{p+N}\sigma,p,A_q}\phi^{L_p(\partial_{p+N}\sigma)}(\beta(u_{p+N}^*(a),r)).
\end{eqnarray*}

Since $R_{p+1}(\partial_0\sigma)=R_{p+1}(\sigma)$, by coherence  (\ref{compatibility}), we get 
	$$c^{\sigma,p+1,A_q}c^{L_{p+1}(\sigma),1,(R_{p+1}(\sigma))^*A_q} =c^{\sigma,1,A_q}u_1^*(c^{\partial_0\sigma,p+1,A_q}).$$
This implies $B(a,r,\beta)=-B'(a,r,\beta)$, and thus $$\sum_{\substack{r\in\ppp(R_{p+1}(\partial_0\sigma))\\\beta\in S_{q,N-2}}}B'(a,r,\beta)+\sum_{{\substack{r\in\ppp(R_{p+1}(\sigma))\\\beta\in S_{q,N-2}}}} B(a,r,\beta)=0.$$

For $i=1,\dots, p$, we have $\epsilon^{L_{p+1}(\sigma),i,(R_{p+1}(\sigma))^\star A_0}=\epsilon^{\sigma,i,A_0}$, $\partial_iL_{p+1}(\sigma)=L_{p+1}(\partial_i\sigma)$, and $c^{\sigma,p+1,A_q}=c^{\partial_i\sigma,p+1,A_q}$. Hence $$\sum_{i=1}^p\sum_{{\substack{r\in\ppp(R_{p+1}(\sigma))\\\beta\in S_{q,N-2}}}}C(a,r,\beta,i)+\sum_{i=1}^{p}\sum_{\substack{r\in\ppp(R_{p+1}(\partial_i\sigma))\\\beta\in S_{q,N-2}}}C'(a,r\
,\beta,i)=0 .$$

Now we obtain
\begin{align*}
&(d_{N-1}(d_1\phi))^\sigma(a)+(d_{1}(d_{N-1}\phi))^\sigma(a)\\
&=\sum_{{\substack{r\in\ppp(R_{p+1}(\sigma))\\\beta\in S_{q,N-2}}}} D(a,r,\beta)
+\sum_{i=p+1}^{p+N-1}\sum_{\substack{r\in\ppp(R_{p}(\partial_i\sigma))\\\beta\in S_{q,N-2}}}C''(a,r\,\beta,i)
+\sum_{\substack{r\in\ppp(R_{p}\partial_{p+N}\sigma)\\\beta\in S_{q,N-2}}}D'(a,r,\beta) .
\end{align*}
We complete step 3 by showing that every term at the right hand side of this equation is matched with a unique term in $-(d_Nd_\h\phi)^\sigma(a)$. Consider the term $D(a,r,\beta)$ for $r=(r_1,\dots,r_{N-2})\in\ppp(R_{p+1}(\sigma)) $ and $\beta\in S_{q,N-2}$. Let $c_0=c^{R_p(\sigma),1}$, denote $u_{p+1}^*r=(u_{p+1}^*r_1,\dots,u_{p+1}^*r_{N-2})$ then $s=(c_0,u_{p+1}^*r)$ is a path in $\ppp(R_p\sigma)$ and there is a unique $\beta'\in S_{q,N-1}$ such that 
	$$\beta'(a,s)=(c_0(A_q),u_{p+1}^*\beta(a,r)).$$
By computation we get $-(-1)^q(-1)^{s+\beta'}=(-1)^N(-1)^{r+\beta}$. By coherence (\ref{compatibility}), we have
	$$c^{\sigma,p,A_q}(L_p\sigma)^*c_0(A_q)=c^{\sigma,p+1,A_q}c^{L_{p+1}\sigma,p,(R_{p+1}\sigma)A_q}.$$
Hence, we obtain
	\begin{align*}
	T(a,s,\beta',0)&=-(-1)^q(-1)^{s+\beta'}c^{\sigma,p,A_q}(d_\h^0\phi)^{L_p\sigma}(c_0(A_q),u_{p+1}^*\beta(a,r))\\
	&=-(-1)^q(-1)^{s+\beta'}c^{\sigma,p,A_q}(L_p\sigma)^*c_0(A_q)\phi^{L_p\sigma}(u_{p+1}^*\beta(a,r))\\
&=D(a,r,\beta).
	\end{align*}

Consider the term $C''(a,r,\beta,i)$ for $r\in\ppp(R_p\partial_i\sigma),\beta\in S_{q,N-2}$ and $p+1\le i\le p+N-1$. Then $s=(r,\epsilon^{\sigma,i})$ is a path in $\ppp(R_p\sigma)$, there is a unique $\beta'\in S_{q,N-1}$ such that $$\beta'(a,s)=(\beta(a,r),\epsilon^{\sigma,i,A_0}).$$
By computation $(-1)^{p+N+i}(-1)^{r+\beta}=(-1)^N(-1)^{s+\beta'}$, thus we find
	\begin{align*}
	T(a,s,\beta',q+N-1)&=(-1)^N(-1)^{s+\beta'}c^{\sigma,p,A_q}(d_\h^{q+N-1}\phi)^{L_p\sigma}(\beta(a,r),\epsilon^{\sigma,i,A_0})\\
	&=C''(a,r,\beta,i).
	\end{align*}

Consider the term $D'(a,r,\beta)$ where $r=(r_1,\dots,r_{N-2})\in\ppp(R_p\partial_{p+N}\sigma)$ and $\beta\in S_{q,N-2}$. Let $c_0=c^{R_p\sigma,p+N-1}$, denote $ru_{p+N}^*=(r_1u_{p+N}^*,\dots,r_{N-2}u_{p+N}^*)$, then $s=(c_0,ru^*_{p+N})$ is a path in $\ppp(R_p\sigma)$. There is a unique $\beta'\in S_{q,N-1}$ such that 
	$$\beta'(a,s)=(c_0(A_q),\beta(u_{p+N}^*(a),r)).$$
Since $(-1)^q(-1)^{s+\beta'}=(-1)^{r+\beta}$, we obtain
	$$T(a,s,\beta',0)=D'(a,r,\beta).$$

{\em Step 4.} For $\beta\in S_{q,N-1},\ r\in\ppp(R_p\sigma)$, we write $\beta^{(0)}(a,r)=(\underline{\beta}_1,\dots,\underline{\beta}_{q+N-1})$. For each $k=1,\dots,(q+N-2)$, denote by $S_{q,N-1}^k$ the set of all $(q,N-1)$-shuffle permutations $\beta$ such that $$(\underline{\beta}_k,\underline{\beta}_{k+1})\ne (a_i,a_{i+1}),\ \forall \,\, {  i=1, \dots,q-1}.$$
After steps $1,2,3,$ now it is seen that 
$$-(d_Nd_\h\phi)^\sigma(a)=(d_\h d_N\phi+d_{N-1}d_1\phi+d_1d_{N-1}\phi+\sum_{i=2}^{N-2}d_{N-i}d_i\phi)^\sigma(a)+X$$
where $$X=\sum_{k=1}^{q+N-2}\sum_{\substack{r\in\ppp(R_p(\sigma))\\ \beta\in S^k_{q,N-1}}}T(q,r,\beta,k).$$
Recall that $$T(a,r,\beta,k)=(-1)^{q+1+k}(-1)^{r+\beta} c^{\sigma,p,A_q}(d_\h^k\phi)^{L_p(\sigma)}(\beta(a,r)).$$
Let $\beta\in S^k_{q,N-1},\ r=(r_1,\dots,r_{N-1})\in\ppp(R_p\sigma)$. In the expression $$\beta^{(0)}(a,r)=(\underline{\beta}_1,\dots,\underline{\beta}_k,\underline{\beta}_{k+1},\dots,\underline{\beta}_{q+N-1})$$
if $(\underline{\beta}_k,\underline{\beta}_{k+1})=(a_i,r_j)$ or $(\underline{\beta}_k,\underline{\beta}_{k+1})=(r_j,a_i)$ for some $(i,j)$, then take $\beta'=(k,k+1)\circ \beta$, then $$T(a,r,\beta,k)+T(a,r,\beta',k)=0.$$
Otherwise, $(\underline{\beta}_k,\underline{\beta}_{k+1})=(r_i,r_{i+1})$ for some $i$. Then, by (\ref{flipequation1}) and (\ref{flipequation2}), we get 
$$T(a,r,\beta,k)+T(a,\mathrm{flip}(r,k),\beta,k)=0.$$
Hence $X=0$, this completes our proof.
\end{proof}

\subsection{Normalized reduced cochains}\label{parnormred}
In this section, in analogy with \cite[\S 2.4]{dinhvanliulowen}, we study the subcomplex $\cgsnr{\bullet} \subseteq \cgs{\bullet}$ of normalized reduced cochains.
Let $\sigma=(u_1,\dots,u_p)$ be a $p$-simplex as in $(\ref{eqsigma0})$. The simplex $\sigma$ is  said to be {\em right $k$-degenerate} if $u_i=1_{U_i}$ for some $p-k+1\le i\le p$, $\sigma$ is said to be \emph{degenerate} if it is right $k$-degenerate for $k =  p$. For $A=(A_1,\dots, A_q)\in \aaa(U_p)$ and $a=(a_1,\dots,a_q)$ as in $(\ref{equationa0})$, $a$ is said to be {\em normal} if $a_i=1$ for some $i$.

Given a cochain $\phi=(\phi^\sigma)_\sigma\in\cgs{n}$, $\phi^\sigma$ is said to be {\em normalized} if $\phi^\sigma(a)=0$ as soon as $a$ is normal, and $\phi$ is said to be {\em normalized} if $\phi^\sigma$ is normalized for every simplex $\sigma$. The normalized cochains form a subcomplex $\cgsn{\bullet}$ of $\cgs{\bullet}$. The cochain $\phi$ is said to be \emph{right $k$-reduced} if $\phi^\sigma=0$ for every right $k$-degenerate simplex $\sigma$ and $\phi$ is said to be \emph{reduced} if $\phi^\sigma=0$ for every degenerate simplex $\sigma$.  The normalized reduced cochains further form a subcomplex $\cgsnr{\bullet}$ of $\cgsn{\bullet}$. 

Inspired by \cite[\S 2.4]{dinhvanliulowen},\cite[\S7]{gerstenhaberschack},  we first prove that the inclusion $\cgsn{\bullet}\hookrightarrow\cgs{\bullet}$ is a quasi-isomorphism. It is more subtle to prove $\cgsnr{\bullet}\hookrightarrow \cgsn{\bullet}$ is also a quasi-isomorphism. Due to the higher components of our new differential, the spectral sequence argument does not apply as in \cite[\S 2.4]{dinhvanliulowen}. As a single filtration is not sufficient, we use a double filtration instead. 
\begin{remark}
If $\aaa$ is a presheaf of $k$-linear categories, then the new differential $d$ on $\cgs{\bullet}$ does not reduce to $d_{\mathrm{GS}}$ from \S \ref{parGScompl}. However, on the quasi-isomorphic subcomplex $\cgsnr{\bullet} \subseteq \cgs{\bullet}$ of normalized reduced cochains, $d$ and $d_{\mathrm{GS}}$ do coincide in this case.
\end{remark}

The following Lemma is obvious. 
\begin{lemma}\label{normalizedreduced1}
	Given a cochain complex $(D^\bullet,\delta)$. Let $(D'^\bullet,\delta)$ be a subcomplex of $(D^\bullet,\delta)$. Assume that for every cochain $f\in D^n$, if $\delta(f)\in D'^{n+1}$ then there exists $h\in D^{n-1}$ such that $f-\delta(h)\in D'^n$, for all $n$. Then the inclusion $(D'^\bullet,\delta)\hookrightarrow (D^\bullet,\delta)$ is a quasi-isomorphism.
\end{lemma}

It is seen that for each simplex $\sigma$, $\cgs{\sigma,\bullet}$ is a cochain complex with the differential $d_\h$. By similar computations as in \cite[\S 7]{gerstenhaberschack} we obtain
\begin{lemma}\label{normalized0}
 Let $f\in\cgs{\sigma,n}$ be a cochain. Assume that $d_\h(f)$ is normalized, then there exists $h\in \cgs{\sigma,n-1}$ such that $f-d_\h(h)$ is  normalized.
\end{lemma}

Equip $\cgs{\bullet}$ with a filtration $$\cdots\subseteq F^p\CC^{n}\subseteq F^{p-1}\CC^{n}\subseteq\cdots\subseteq F^0\CC^{n}\subseteq F^{-1}\CC^n=\cgs{n}$$ by setting
\[F^j\CC^{n}=\{\phi = (\phi^{\sigma})_{\sigma} \in \cgs{n}\mid \phi^\sigma \text{ is normalzied if $|\sigma|\le j$ }\}.\]

Since $d(F^j\CC^{p})\subseteq F^j\CC^{p+1}$, $F^j\CC^{\bullet}$ is  a complex. There is a sequence of complexes
\begin{equation}
\cdots\hookrightarrow F^j\CC^{\bullet}\hookrightarrow F^{j-1}\CC^{\bullet}\hookrightarrow \cdots\hookrightarrow F^0\CC^{\bullet}.
\end{equation}

\begin{proposition}\label{inclusion1}
The following inclusions are quasi-isomorphisms:
\begin{enumerate}
\item $l\colon F^j\CC^{\bullet}\hookrightarrow F^{j-1}\CC^{\bullet}$;
\item $\cgsn{\bullet}\hookrightarrow \cgs{\bullet}$.
\end{enumerate}
\end{proposition}
\begin{proof}
	It suffices to prove that (1) is a quasi-isomorphism. By Lemma \ref{normalizedreduced1} it is sufficient to prove that for every cochain $\phi\in F^{j-1}\CC^n$, if $d(\phi)\in F^j\CC^{n+1}$ then there exists a cochain $\psi\in F^{j-1}\CC^{n-1}$ such that $\phi-d(\psi)\in F^j\CC^n$. Writing $\phi=(\phi_{p,q})_{p+q=n}$, we assume that $d(\phi)\in F^j\CC^{n+1}$. Let $\sigma$ be a $j$-simplex  and let $a=(a_1,\dots,a_{n+1-j})$ be normal, then $(d(\phi))^\sigma(a)=0$. By definition, we have 
	$$(d\phi)^\sigma(a)=\sum_{i=0}^j(d_i\phi_{j-i,n-j+i})^\sigma(a).$$
	Note that $(d_i\phi_{j-i,n-j+i})^\sigma(a)=0$ for $i>0$ as $\phi\in F^{j-1}\CC^n$.  Hence we get $$(d_\h\phi_{j,n-j})^\sigma(a)=0.$$
	By Lemma \ref{normalized0}, there exists $h^\sigma\in \CC^{\sigma,n-j-1} $ such that $\phi_{j,n-j}^\sigma-d_\h(h^\sigma)$ is normalized. We define $\psi^\sigma=h^\sigma$ if $|\sigma|=j$ and $\psi^\sigma=0$ otherwise. Thus $\psi\in F^{j-1}\CC^{n-1}$ and it is easy to see that $\phi-d(\psi)\in F^j\CC^n$.
\end{proof}

Now equip $\cgsn{\bullet}$ with a filtration $$\cdots\subseteq F'^p\bar{\CC}^{n}\subseteq F'^{p-1}\bar{\CC}^{n}\subseteq\cdots\subseteq F'^0\bar{\CC}^{n}=\cgsn{n}$$ by setting, for each $k\ge 1$,
\[
F'^k\bar{\CC}^{n}=\{\phi = (\phi^{\sigma})_{\sigma} \in \cgsn{n}\mid \phi^\sigma(a)=0 \ \forall a, \text{ if $\sigma$ is right $k$-degenerate}\}.
\]

\begin{lemma}\label{filtrationdegenerate1}
$	d(F^k\bar{\CC}^n) \subseteq F^k\bar{\CC}^{n+1}.	$
\end{lemma}
\begin{proof}
	Let $\phi\in F'^k\bar{\CC}^n$, $\sigma=(u_1,\dots,u_p)$ be a right $k$-degenerate $p$-simplex, and let $ a=(a_1,\dots,a_{n+1-p})$.
We need to prove $(d\phi)^\sigma(a)=0$. By definition $(d\phi)^\sigma=\sum_{i=0}^p(d_i\phi)^\sigma$. Obviously, we have $(d_0\phi)^\sigma=0$ and $(d_1\phi)^\sigma=0$. For $i\ge 2$ we have $$(d_i\phi)^\sigma(a)=(-1)^{n+1-p}(-1)^{r+\beta}\sum_{\substack{r\in\ppp(R_i\sigma);\beta\in S_{n+1-p,p-1}}}c^{\sigma,i}\phi^{L_i\sigma}(\beta(a,r)).$$
Because $\sigma$ is right $k$-degenerate, either $L_i\sigma$ is  right $k$-degenerate or $R_i\sigma$ is degenerate. If $R_i\sigma$ is degenerate, then for each path $r$, we have $r_j=1$ for some $j$. So $\beta(a,r)$ is normal, and we get $\phi^{L_i\sigma}(\beta(a,r))=0$.	
\end{proof}

 By Lemma \ref{filtrationdegenerate1} we obtain a sequence of complexes
\begin{equation}\label{filtration2}
\cdots\hookrightarrow F'^k\bar{\CC}^{\bullet}\hookrightarrow F'^{k-1}\bar{\CC}^{\bullet}\hookrightarrow \cdots\hookrightarrow F'^0\bar{\CC}^{\bullet}.
\end{equation}
Next, for each $k\ge 0$, we equip  $F'^k\bar{\CC}$ with a further  filtration 
$$F'^{k+1}\bar{\CC}^n=G^{n+1} F'^k\subseteq\bar{\CC}^{n}\cdots\subseteq G^{l+1} F'^k\bar{\CC}^{n}\subseteq G^{l} F'^{k}\bar{\CC}^{n}\subseteq\cdots\subseteq G^0F'^k\bar{\CC}^{n}=F'^k\bar{\CC}^{n}$$ 
by setting $$ G^{l} F'^{k}\bar{\CC}^{n}=\{\phi\in F'^k\bar{\CC}^n|\ \phi^\sigma=0 \text{ for $|\sigma|\ge n-l+1$ and $\sigma$ is right $(k+1)$-degenerate} \}.$$
By {analogous} computations as in Lemma \ref{filtrationdegenerate1}, we get $$d( G^{l} F'^{k}\bar{\CC}^{n})\subseteq  G^{l} F'^{k}\bar{\CC}^{n+1}.$$ 
Thus, for each $k$, we obtain a sequence of complexes 
\begin{equation}\label{filtration3}
\cdots\hookrightarrow G^{l+1}F'^k\bar{\CC}^{\bullet}\hookrightarrow G^lF'^{k}\bar{\CC}^{\bullet}\hookrightarrow \cdots\hookrightarrow F'^k\bar{\CC}^{\bullet}.
\end{equation}

\begin{lemma}\label{reduced1} 
	Let $\phi$ be right $k$-reduced cochain in $\cgsn{p,q}$. If $d_\s\phi$ is a right $(k+1)$-reduced cochain in  $\cgsn{p+1,q}$, then there exists a right $k$-reduced cochain $\psi\in\cgsn{p-1,q}$ such that $\phi-d_\s\psi$ is a right $(k+1)$-reduced cochain.
\end{lemma}

\begin{proposition}\label{inclusion2}The following inclusions are quasi-isomorphism:
	\begin{enumerate}
		\item  $G^{l+1} F'^{k}\bar{\CC}^{\bullet}\hookrightarrow G^{l} F'^{k}\bar{\CC}^{\bullet}$; 
		\item  $F'^{k+1}\bar{\CC}^{\bullet}\hookrightarrow F'^{k}\bar{\CC}^{\bullet}$;
		\item  $\cgsnr{\bullet}\hookrightarrow \cgsn{\bullet}$.
	\end{enumerate}

\end{proposition}
\begin{proof}
	For each $n$, the filtrations (\ref{filtration2}) and (\ref{filtration3}) are stationary, so we only need to prove (1). By Lemma \ref{normalizedreduced1}, it is sufficient to prove that for any cochain  $\phi=(\phi_{p,q})\in G^lF'^k\bar{\CC}^n$ which satisfies $d\phi\in G^{l+1}F'^k\bar{\CC}^{n+1}$, there exists a cochain $\psi\in G^lF'^k\bar{\CC}^{n-1}$ such that $\phi-d\psi\in G^{l+1}F'^k\bar{\CC}^n$. 
	
	Set $p=n-l+1$ and let  $\sigma$ be $(k+1)$-right degenerate $p$-simplex. By definition, we have $\phi_{p,n-p}^\sigma=0$. Assume that $(d\phi)^\sigma=0$, by computation as in Lemma \ref{filtrationdegenerate1}, we deduce 
	$$(d_\simp\phi_{p-1,n-p+1})^\sigma=0.$$
	Apply Lemma \ref{reduced1}, there exists $h\in F'^{k}\bar{\CC}^{p-2,n-p+1}$ such that
	 $$(\phi_{p-1,n-p+1}-d_1(h))^{\sigma'}=0$$
	 for every $(k+1)$-right degenerate $(p-1)$-simplex $\sigma'$.\\
	  We define $\psi^\sigma=h^\sigma$ if $|\sigma|=p-2$ and $\psi^\sigma=0$ elsewhere. It is seen that $\psi\in G^lF'^k\bar{\CC}^{n-1}$ and $\phi-d\psi\in G^{l+1}F'^k\bar{\CC}^n$ as desired.
\end{proof}

Combining Propositions \ref{inclusion1} and \ref{inclusion2}, we now obtain the following isomorphisms.

 \begin{proposition}\label{normalreducefinal} Let $M$ be an  $\aaa$-bimodule. Then
 $$H^n\cgsnr{\bullet}\simeq H^n\cgsn{\bullet}\simeq H^n\cgs{\bullet}.$$
 \end{proposition}
 
 \begin{remark}
If $\aaa$ is a presheaf of $k$-linear categories, then the new differential $d$ does not reduce to the old $d$ from \S \ref{parGScompl}. However, on the quasi-isomorphic subcomplex $\cgsnr{\bullet} \subseteq \cgs{\bullet}$ of normalized reduced cochains defined in \S \ref{parnormred}, they do coincide in this case.
\end{remark}

\subsection{First-order deformations of prestacks}\label{pardefo}
In this section, generalizing \cite[Thm 2.21]{dinhvanliulowen}, we prove that $HH^2_{\mathrm{GS}}$ classifies first order deformations of prestacks.

\begin{definition}\label{defprestack} (see Def 3.24 in \cite{lowenmap})
Let $(\aaa, m, f, c)$ be a prestack over $\uuu$.
\begin{enumerate}
\item
 A \emph{first order  deformation} of $\aaa$ is given by a prestack
$$(\bar{\aaa}, \bar{m}, \bar{f}, \bar{c}) = (\aaa[\epsilon], m + m_1\epsilon, f + f_1\epsilon, c + c_1\epsilon)$$
of $k[\epsilon]$-categories  where $(m_1, f_1, c_1)\in \CC^{0,2}(\aaa) \oplus \CC^{1,1}(\aaa) \oplus \CC^{2,0}(\aaa)$.

\item
 For two  deformations $(\bar{\aaa}, \bar{m}, \bar{f}, \bar{c})$ and $(\bar{\aaa}', \bar{m}', \bar{f}', \bar{c}')$ an \emph{equivalence of  deformations} is given by an isomorphism of the form $(g,\tau)=(1+g_1\epsilon,1+\tau_1\epsilon)$ where $(g_1,\tau_1)\in\CC^{0,1}(\aaa)\oplus\CC^{1,0}(\aaa)$.
\end{enumerate}
\end{definition}

\begin{theorem}\label{proptwist}
	Let $\aaa = (\aaa, m, f, c)$ be prestack with Gerstenhaber-Schack complex $(\CC_\mathrm{GS}^\bullet(\aaa),d)$. Then the second cohomology $HH^2_\mathrm{GS}(\aaa)$ classifies the first order deformation of $\aaa$. More concretely
	\begin{enumerate}
		\item
		\begin{enumerate}
			 For $(m_1, f_1, c_1)$ in $\CC^{0,2}(\aaa) \oplus \CC^{1,1}(\aaa) \oplus \CC^{2,0}(\aaa)$, we have that $(\aaa[\epsilon], \bar{m}=m + m_1\epsilon,\bar{f}= f + f_1\epsilon, \bar{c}=c+c_1\epsilon)$ is a first order  deformation of $\aaa$ if and only if $(m_1,f_1,c_1)\in\bar{\CC}_\mathrm{GS}'^2(\aaa)$ and $d(m_1, f_1, c_1) = 0$.
		\end{enumerate}
		\item
		\begin{enumerate}
			 For $(m_1, f_1, c_1)$ and $(m_1', f_1', c_1')$ in $Z^2\bar{\CC}_\mathrm{GS}'(\aaa)$, and $(g_1, -\tau_1) \in \CC^{0,1}(\aaa) \oplus \CC^{1,0}(\aaa)$, we have that $(g,\tau)=(1 + g_1\epsilon, 1 + \tau_1\epsilon)$ is an isomorphism between prestacks $\bar{\aaa}$ and $\bar{\aaa}'$ if and only if $(g_1,-\tau_1)\in \bar{\CC}_\mathrm{GS}'^1(\aaa)$ and $d(g_1, -\tau_1) = (m_1, f_1, c_1)-(m_1', f_1', c_1').$
			 	We have an isomorphism of sets
			\begin{equation}\label{eqclass}
			H^2\bar{\CC}_\mathrm{GS}'(\aaa) \lra \Def(\aaa).
			\end{equation}
			Hence, the second cohomology group $HH^2(\aaa)_{\mathrm{GS}} \cong H^2\bar{\CC}_\mathrm{GS}'(\aaa)$ classifies first order  deformations of $\aaa$ up to equivalence.
		\end{enumerate}
	\end{enumerate}
\end{theorem}
\begin{proof}

	\begin{enumerate}
	\item For each $U\in\uuu$, the composition $\bar{m}^U$ of $\aaa(U)$ is associative if and only if $$(d_\h m_1)^U=0.$$
	For each $a\in\aaa(U)(A,B)$, the unity condition $\bar{m}^U(1_B,a)=a=\bar{m}^U(a,1_A)$ holds if and only if $m_1^U(1_B,a)=m_1^U(a,1_A)=0$.\\
	
	For each $1$-simplex $\sigma=(V\overset{v}{\lra} U)$ and $(a,b)\in\aaa(U)(A,B)\times \aaa(U)(B,C)$. The condition $\bar{m}^V(\bar{f}(b),\bar{f}(a))=\bar{f}(\bar{m}^U(b,a))$ holds if and only if $$(d_\h f_1)^\sigma(b,a)-(d_\s m_1)^\sigma(b,a)=0.$$
	The condition $\bar{f}^\sigma(1_A)=1_{\bar{f}^\sigma(A)}$ is equivalent to $f_1^\sigma(1_A)=0$. The condition $\bar{f}^{1_U}=1_U$ holds if and only if $f_1^{1_U}=0$.\\
	
	For each $2$-simplex $\sigma=(W\overset{v}{\lra}V\overset{u}{\lra}U)$ and $a\in\aaa(U)(A,B)$, the condition $\bar{m}(\bar{c}^{u,v,B},\bar{f}^v\bar{f}^u(a))=\bar{m}(\bar{f}^{uv}(a),\bar{c}^{u,v,A})$ holds if and only if $$(d_\h c_1)^\sigma(a)-(d_\s f_1)^\sigma(a)+(d_2 m_1)^\sigma(a)=0.$$
	The condition that $\bar{c}^{\sigma}=1$ when $\sigma$ is degenerated  holds if and only if $c_1^\sigma=0$ if $\sigma$ is degenerated. 
	
	For each $3$-simplex $\sigma=(T\overset{w}{\lra}W\overset{v}{\lra}V\overset{u}{\lra}U)$, the compatibility of $\bar{c}$ holds if and only if $$-(d_\s c_1)^\sigma(A)+(d_2f_1)^\sigma(A)+(d_3m_1)^\sigma(A)=0.$$
	
	Recall that 
	\begin{align*}
	\ \ \ \ \ \ \ \ \ 	d(m_1,f_1,c_1)&=(d_\h m_1,\  d_\h f_1 -d_\s m_1,\ d_\h c_1 -d_\s f_1+d_2m_1,\\
		&\ \ \ -d_\s c_1+d_2f_1+d_2m_1).
	\end{align*}
	These facts yield that $(m_1,f_1,c_1)$ gives rise to a deformation of the prestack $\aaa$ if and only if it is a normalized reduced cocycle.
	
	\item For each $U\in\uuu$, we have that $g^U$ is a functor  if and only if $g_1^U(1)=0$ and $$d_\h (g_1)=m_1-m_1'.$$ 
	
	For each $1$-simplex $\sigma=(V\overset{u}{\lra}U)$ and $a\in\aaa(U)(A,B)$, the condition $m'^V(g^Vu^*(a),\tau^u)=m'^V(\tau^u,u'^* g^U(a))$ holds if and only if $$(d_\h g_1)^\sigma(a)+(d_\s (-\tau_1))^\sigma(a)=f_1^\sigma(a)-f_1'^\sigma(a).$$
The condition $m'^U(\tau^{1_U},1'_U)=g^U(1_U)$ holds if and only if $\tau_1^{1_U}=0$.	
	
	For each $2$-simplex $\sigma=(W\overset{v}{\lra}V\overset{u}{\lra}U)$ and $A\in\aaa(U)$, the condition $m'^W(\tau^{uv},c'^{u,v})=m'^W(g^W(c^{u,v}),\tau^v,v'^*(\tau^u))$ holds if and only if $$(d_\s (-\tau_1))^\sigma(A)+(d_2 g_1)^\sigma(A)=c_1^{\sigma}(A)-c_1'^\sigma(A).$$ 
	
	Hence  $(g,\tau) = (1 +g_1\epsilon, 1+ \tau_1\epsilon)$ is an isomorphism between $\aaa$ and $\aaa'$ if and only
if $(g_1,-\tau_1)$ is a normalized reduced cochain and
\begin{align*}
d(g_1,-\tau_1)&=\big(d_\h g_1,\ d_\h(-\tau_1)+d_\s g_1,\ d_\s(-\tau_1)+d_2g_1\big)\\
&=(m_1,f_1,c_1)-(m_1',f_1',c_1').
\end{align*}

	\end{enumerate}
\end{proof}

\section{Comparision of complexes}\label{parparcomp}

Let $\uuu$ be a small category, $\aaa$ a prestack on $\uuu$, and $M$ an $\aaa$-bimodule. In this section, we define cochain maps 
$$\fff:\cgs{\bullet}\lra \ctil{\bullet} \hspace{0,5cm} \text{and} \hspace{0,5cm} \GGG: \ctil{\bullet} \lra \cgs{\bullet}$$
between the Gerstenhaber-Schack complex $\cgs{\bullet}$ and the Hochschild complex $\CC_{\uuu}^\bullet(\tilde{\aaa},\tilde{M})$ as defined in \cite{lowenmap}. We prove that $\fff$ and $\GGG$ are inverse quasi-isomorphisms. In combination with \cite[Prop. 3.13]{lowenmap} and the Cohomology Comparison Theorem \cite[Thm. 1.1]{lowenvandenberghCCT} it follows that - as in the case of presheaves - if $k$ is a field then the cohomology of the complex $\cgs{\bullet}$ computes bimodule Ext groups. More precisely, we obtain
$$HH^n_{\mathrm{GS}}(\aaa, M) \cong H^n(\ctil{\bullet})\cong \mathrm{Ext}^n_{\tilde{\aaa}-\tilde{\aaa}}(\tilde{\aaa},\tilde{M})\cong \mathrm{Ext}^n_{\aaa-\aaa}(\aaa,M).$$
The differential on $\CC^{\bullet}_{\mathrm{GS}}(\aaa,M)$ being more involved than in the presheaf case, here too we have to come up with more subtle definitions for $\fff$ and $\GGG$, using partitions and making intensive use of the shuffle machinery. The proof of the resulting quasi-isomorphism has two parts. In Proposition \ref{propGF}, we prove that $\GGG\fff(\phi)=\phi$ for any normalized reduced cochain $\phi$. The harder work goes into Theorem \ref{homotopy}, in which we construct a homotopy $\fff \GGG \sim 1$. 
Our Theorem \ref{homotopy} has a powerful consequence, as by the Homotopy Transfer Theorem \cite[Theorem 10.3.9]{lodayvallette}, we can transfer the  dg Lie algebra structure present on $\CC_{\uuu}^\bullet(\tilde{\aaa})$ (see \cite{lowenmap}) in order to obtain an $L_{\infty}$-structure on $\CC^{\bullet}_{\mathrm{GS}}(\aaa)$. This $L_{\infty}$-structure determines the higher deformation theory of $\aaa$ as a prestack, which thus becomes equivalent to the higher deformation theory of the $\uuu$-graded category $\tilde{\aaa}$ described in \cite{lowenmap}. A more detailed elaboration of this $L_{\infty}$-structure, as well as a comparison with the $L_{\infty}$ deformation complex described in the literature in an operadic context \cite{fregiermarklyau},\cite{doubek},\cite{merkulovvallette} will appear in \cite{dinhvan}.

\subsection{The cochain map $\fff$}
Following \cite{lowenmap} the Hochschild complex $(\ctil{\bullet},\delta)$ of the $\uuu$-graded category $\tilde{\aaa}$ is defined  as
$$\ctil{\bullet}=\prod_{\substack{u_1,\dots,u_n\\A_0,A_1,\dots,A_n}}=\mathrm{Hom}_k\big(\otimes_{i=1}^n\tilde{\aaa}_{u_{n+1-i}}(A_{n-i},A_{n+1-i}),\tilde{\aaa}_{u_n\cdots u_1}(A_0,A_n))\big)$$
where $\delta$ is the usual Hochschild differential.

In order to define the cochain map $$\fff:\cgs{\bullet}\lra \ctil{\bullet}$$ we need to introduce the following notations. For each $n\in\mathbb{N}$ denote the set of all partitions of $n$ as 
\[\Part{n}=\{\bar{m}=(m_k,\dots,m_1)|\ m_k+\dots+m_1=n,\ k\ge 1,  m_i\ge 1\}\]
We define $(-1)^{\bar{m}}=(-1)^{n-k}$ for $\bar{m}=(m_k,\dots,m_1)$.

Let $\sigma=(u_1,\dots,u_n)$ be a $n$-simplex as in $(\ref{eqsigma0})$, denote $||\sigma||=u_n\cdots u_1$. For $i\le k$ denote by $\sigma[m_i]$ the $m_i$-simplex $(u_{m_k+\dots+m_{i+1}+1},\dots,u_{m_k+\dots+m_i})$. For example, we have $\sigma[m_k]=(u_1,\dots,u_{m_k})$ and $\sigma[m_{k-1}]=(u_{m_k+1},\dots,u_{m_k+m_{k-1}})$. Put $c^{\sigma,\bar{m}} = ||r||$ for an arbitrary $r\in\ppp(||\sigma[m_k]||,\dots,||\sigma[m_1]||)$.

Given
$A_i\in \mathrm{Ob}(\tilde{\aaa}(U_i))$, consider $\tilde{a}=(\tilde{a}_1,\dots,\tilde{a}_n)$ where 
$$\tilde{a}_i\in \tilde{\aaa}_{u_{n+1-i}}(A_{n-i}, A_{n+1-i}) = \aaa(U_{n-i})(A_{n-i}, u_{n+1-i}^{\ast}A_{n+1-i})$$ as follows:
\begin{equation}\label{tildea}
\xymatrix{{A_0}\ar[r]^-{\tilde{a}_n}&{A_1} \ar[r]^-{\tilde{a}_{n-1}}&\cdots \ar[r]^-{\tilde{a}_2}&A_{n-1}\ar[r]^-{\tilde{a}_1}& A_n \\
	{U_0} \ar[r]_{u_1} & {U_1} \ar[r]_{u_2} & {\dots} \ar[r]_{u_{n-1}} & {U_{n-1}} \ar[r]_{u_n} & {U_n.}}
\end{equation}

For each $i=1, \dots, n$, denote 
\begin{align*}
&\underline{\tilde{a}}_i= u_1^*\cdots u_{n-i}^* \tilde{a}_i\in \aaa(U_0)(u_1^*\cdots u_{n-i}^*A_{n-i},u_1^*\cdots u_{n+1-i}^*A_{n+1-i});\\
&\tilde{a}_{i,\dots,n}=\underline{\tilde{a}}_i\circ\dots\circ\underline{\tilde{a}}_n\in \aaa(U_0)(A_0,u_1^*\cdots u_{n+1-i}^*A_{n+1-i}).
\end{align*}
Given a partition $\bar{m}=(m_k,\dots,m_1)\in\Part{n}$, denote
$$\tilde{a}[m_i] = \underline{\tilde{a}}_{m_{i-1}+\dots+m_1+1}    \circ\dots\circ   \underline{\tilde{a}}_{m_i+\dots+m_1},$$
thus $\tilde{a}[m_1]=\underline{\tilde{a}}_1\circ\dots\circ\underline{\tilde{a}}_{m_1}$ and $\tilde{a}[m_k]=\tilde{a}_{n-m_k+1,\dots,n}$.\\

For $r=(r_1,\dots,r_{n-1})\in \ppp(\sigma)$, we obtain the following $n$-simplex in $\aaa(U_0)$:
$$(r(A_n),\tilde{a}_{1,\dots,n}) \equiv (r_1(A_n),\dots,r_{n-1}(A_n),\tilde{a}_{1,\dots,n}).$$
Now for each partition $\bar{m}=(m_k,\dots,m_1)$ of $n$ we define by induction a set 
$$\Seq(\sigma,\bar{m}) \equiv \Seq(\sigma,\tilde{a},\bar{m}) \subseteq \nnn_n(\aaa(U_0))(A_0, \sigma[m_k]^{\ast}\cdots \sigma[m_1]^{\ast}A_n)$$
along with a sign map $$\Seq(\sigma,\bar{m}) \lra \{ 1, -1\}: \xi \longmapsto \mathrm{sign}(\xi) \equiv (-1)^{\xi}.$$
Simultaneously, for each sequence $\xi \in \Seq(\sigma,\bar{m})$ we define the {\em  formal sequence} $\underline{\xi}$ of $\xi$, then denote the set of all these formal sequences
$$\underline{\Seq}(\sigma,\bar{m})=\{\underline{\xi}| \ \xi \in \Seq(\sigma,\bar{m})\}.$$
\begin{itemize}
	\item For $k=1$, $\bar{m}=(m_1)$ where $m_1=n$, we define  $$\Seq(\sigma,\bar{m})=\{(r(A_n), \tilde{a}_{1,\dots,n})\,\, |\,\, r\in\ppp(\sigma)\}.$$
	For each element $\xi=(r(A_n),\tilde{a}_{1,\dots,n})\in \Seq(\sigma,\bar{m})$ we define 
	$$\mathrm{sign}(\xi)=(-1)^r.$$
	The formal sequence of $\xi$ is defined to be $$\underline{\xi}=(r,\tilde{a}_{1,\dots,n}).$$
	\item For $k\ge 2$, $R_{m_k}\sigma$ is an $(n-m_k)$-simplex. Let $\xi=(\xi_1,\dots,\xi_{n-m_k})\in \Seq(R_{m_k}\sigma, (m_{k-1},\dots,m_1)) \subseteq \nnn_{n-m_k}(\aaa(U_{m_k}))(A_{m_k}, \sigma[m_{k-1}]^{\ast}\cdots \sigma[m_1]^{\ast}A_n)$. Let $\underline{\xi}=(\underline{\xi}_1,\dots,\underline{\xi}_{n-m_k})$ be the formal sequence of $\xi$.
	\begin{enumerate}
		\item[(i)] Case $m_k=1$. Let $u_1^*\xi=(u_1^*\xi_1,\dots,u_1^*\xi_{n-m_k})$, then we obtain the concatenation
		$$(u_1^*\xi,\tilde{a}_n) \in \nnn_n(\aaa(U_0))(A_0, \sigma[m_k]^{\ast}\cdots \sigma[m_1]^{\ast}A_n).$$ We define 
		$$\Seq(\sigma,\bar{m})=\{(u_1^*\xi,\tilde{a}_n)\,\, |\,\, \xi\in \Seq(R_{m_k}\sigma,(m_{k-1},\dots,m_1)) \}.$$
		For each element $\xi'=(u_1^*\xi,\tilde{a}_n)\in \Seq(\sigma,\bar{m})$, we define 
		$$\mathrm{sign}(\xi')=\mathrm{sign}(\xi).$$
		Now we define 
		the formal sequence of $\xi'$ to be $$\underline{\xi'}=(\underline{\xi},\tilde{a}_n).$$

		\item[(ii)] Case $m_k\ge 2$. For $s\in\ppp(L_{m_k}\sigma)$ and $\beta\in S_{n-m_k,m_k-1}$, we obtain 
		the shuffle $$\xi\underset{\beta}{*}s \in \nnn_{n-1}(\aaa(U_0))(\sigma[m_k]^{\star}A_{m_k}, \sigma[m_k]^{\ast}\cdots \sigma[m_1]^{\ast}A_n)$$
		taken with respect to evaluation of functors. Concatenation with $\tilde{a}_{n+1-m_k,\dots,n} \in \aaa(U_0)(A_0, \sigma[m_k]^{\star}A_{m_k})$ yields an
		$n$-simplex 
		$$(\xi\underset{\beta}{*}s,\tilde{a}_{n+1-m_k,\dots,n}) \in \nnn_n(\aaa(U_0))(A_0, \sigma[m_k]^{\ast}\cdots \sigma[m_1]^{\ast}A_n).$$ Put $m'=(m_{k-1},\dots,m_1)$. We define 
		\begin{align*}
		\Seq(\sigma,\bar{m})=&\{(\xi\underset{\beta}{*}r,\tilde{a}_{n+1-m_k,\dots,n})|\ \xi\in  \Seq(R_{m_k}\sigma, m'),r\in\ppp(R_{m_k}\sigma),\\
		&\beta\in S_{n-m_k,m_k-1}\}.
		\end{align*}
		
		For each element $\xi'=(\xi\underset{\beta}{*}r,\tilde{a}_{n+1-m_k,\dots,n})\in\Seq(\sigma,\bar{m})$ we define $$\mathrm{sign}(\xi')=(-1)^r(-1)^\beta\mathrm{sign}(\xi).$$
		Let $\beta(\underline{\xi},r)$ be the formal shuffle product of $\underline{\xi}$ and $r$. The formal sequence of $\xi'$ is defined to be 
		$$\underline{\xi'}=(\beta^{(0)}(\underline{\xi},r),\tilde{a}_{n+1-m_k,\dots,n}).$$
		
	\end{enumerate}
\end{itemize}

\begin{example}
	Consider a partition $m=(m_3,m_2,m_1)$ of $n$ where $m_i\ge 2$. Each element $\xi\in\Seq(\sigma,\bar{m})$ is of the form 
	$$\xi=\bigg(\big((r_1,\tilde{a}[m_1])\underset{\beta_1}{*}r_2,\tilde{a}[m_2]\big)\underset{\beta_2}{*}r_3,\tilde{a}[m_3]\bigg)$$ 
	where $r_1\in\ppp(\sigma[m_1]),\ r_2\in \ppp(\sigma[m_2]),\ r_3\in\ppp(\sigma[m_3])$ and $\beta_1\in S_{m_1,m_2-1},\beta_2\in S_{m_1+m_2,m_3-1}$. 
\end{example}

Now we are able to define the maps $F_p:\cgs{p,n-p}\lra \CC^n(\tilde{\aaa},\tilde{M})$. Let $\sigma=(u_1,\dots, u_n)$ be an $n$-simplex and $\tilde{a}=(\tilde{a}_1,\dots,\tilde{a_n})$ as in $(\ref{tildea})$. For each cochain $\phi=(\phi_{p,q})\in\cgs{n}$, we define
$$(\fff_p\phi_{p,n-p})(\tilde{a})=\sum_{\bar{m}\in\Part{n-p}}\sum_{\xi\in \Seq(R_p\sigma,\bar{m})}(-1)^{\bar{m}+\xi} \fff_p^{\sigma,\bar{m},A_n}\phi_{p,n-p}^{L_p\sigma}(\xi)\tilde{a}_{n+1-p,\dots,n}$$
where $\fff_p^{\sigma,\bar{m},A_n}=c^{\sigma,p,A_n}(L_p\sigma)^*c^{R_p\sigma,\bar{m},A_n}$. The map $\fff$ is as follows 
$$\fff(\phi)=\sum_{p+q=n}\fff_p(\phi_{p,q}).$$

\begin{proposition}\label{Fcommutes}
	The map $\fff$ commutes with differentials. More precisely, let $p+q=n-1$, for $\phi\in\cgs{p,q}$, then $F(d\phi)=\delta(F\phi)$. 
\end{proposition}
\begin{proof}
	Let $\sigma=(u_1,\dots,u_n)$ be a $n$-simplex as in $(\ref{eqsigma0})$, and let $\tilde{a}=(\tilde{a}_1,\dots,\tilde{a}_n)$ as in $(\ref{tildea})$. First, we prove that $\fff(d\phi)=\delta(\fff\phi)$ for the case $\phi\in \cgs{0,n-1}$. \\The  equation 
	\begin{equation}\label{Fcochain0}
	\sum_{i=0}^n(-1)^i \fff d_\h^i\phi+(-1)^n\fff (d^0_\s-d^1_\s)\phi+\sum_{i=2}^n\fff d_i\phi=\sum_{i=0}^n\delta_n\fff\phi
	\end{equation}
	holds true if the following equations hold true:
	\begin{itemize}
		\item [(i)]  $-(-1)^{n}\fff d_\h^n\phi=(-1)^{n+1}\fff d_\s^1\phi+\sum_{i=2}^n\fff d_i\phi $;\\
		\item[(ii)] $ \fff d_\s^0\phi=\delta_n\fff \phi$;\\
		\item[(iii)]  $\fff d_\h^0\phi+\sum_{i=1}^{n-1}(-1)^i\fff d_\h^i\phi=\sum_{i=0}^{n-1}(-1)^i\delta_i\fff\phi$.
	\end{itemize}
	
	{\em Step 1}. We prove the equation (i). Note that $L_0\sigma=(U_0)$ is a $0$-simplex, by definition, we have  \\ $$(-1)^{n+1}(\fff d_\h^n\phi)^\sigma(\tilde{a})=\sum_{\bar{m}\in\Part{n}}\sum_{\xi\in \Seq(\sigma,\bar{m})}T(\bar{m},\xi)$$ where 
	$$T(\bar{m},\xi)=(-1)^{n+1}(-1)^{\bar{m}+\xi}c^{\sigma,\bar{m},A_n}(d_\h^n\phi)^{U_0}(\xi).$$
	On the right hand side, we have 
	$$(-1)^{n+1}(\fff d_\s^1\phi)^\sigma(\tilde{a})=\sum_{\substack{\bar{m}'\in \Part{n-1}\\ \xi'\in \Seq(R_1\sigma,\bar{m}')}}(-1)^{n+1}(-1)^{\bar{m}'+\xi'}
	\fff_1^{\sigma,\bar{m}',A_n}(d_\s^1\phi)^{L_1\sigma}(\xi')\tilde{a}_n$$
	where 
	$$(d_\s^1\phi)^{L_1\sigma}(\xi')\tilde{a}_n=\phi^{U_0}(u_1^*\xi')\tilde{a}_n.$$
	
	For each $\bar{m}'=(m_k',\dots,m_1')\in\Part{n-1}$ and $\xi'\in \Seq(R_1\sigma,\bar{m}')$, let $\bar{m}=(1,m_k',\dots,m_1')\in\Part{n}$. Then by definition, there exists a unique element $\xi\in \Seq(\sigma,\bar{m})$ such that $\xi=(u_1^*\xi',\tilde{a}_n)$. Hence, we get 
	$$T(\xi,\bar{m})=(-1)^{n+1}(-1)^{\bar{m}'+\xi'}\fff_1^{\sigma,\bar{m}',A_n}(d_\s^1\phi)^{L_1\sigma}(\xi')\tilde{a}_n.$$
	So all the terms occurring in $(-1)^{n+1}(\fff d_\s^1\phi)^\sigma(\tilde{a})$ are canceled.
	
	For  $2\le i\le n$, we have
	$$(\fff d_i\phi)^{\sigma}(\tilde{a})=\sum_{\substack{\bar{m}'\in \Part{n-i}\\ \xi'\in \Seq(R_i\sigma,\bar{m}')}}(-1)^{\bar{m}'+\xi'}c^{\sigma,i,A_n}(L_i\sigma) ^*c^{R_i\sigma,\bar{m}',A_n}(d_i\phi)^{L_i\sigma}(\xi')\tilde{a}_{n+1-i,\dots,n}$$
	where 
	$$\displaystyle (d_i\phi)^{L_i\sigma}(\xi')=\sum_{\substack{r\in\ppp(L_i\sigma),\ \beta\in S_{n-i,i-1}}}(-1)^{n-i}(-1)^{r+\beta}\phi^{U_0}(\xi'\underset{\beta}{*}r).$$
	For each $\bar{m}'\in\Part{n-i}$, $r\in\ppp(L_i\sigma)$ and $ \beta\in S_{n-i,i-1}$, there exists a unique element $\xi\in\Seq(\sigma,\bar{m})$, where $\bar{m}=(n-i,\bar{m}')\in\Part{n}$, such that $\xi=(\xi'\underset{\beta}{*}r,\tilde{a}_{n+1-k,\dots,n})$. By inspection on signs, we get $$T(\xi,\bar{m})=(-1)^{\bar{m}'+\xi'}(-1)^{n-i}(-1)^{r+\beta}c^{\sigma,i,A_n}(L_i\sigma) ^*c^{R_i\sigma,\bar{m}',A_n}\phi^{U_0}(\xi'\underset{\beta}{*}r)\tilde{a}_{n+1-k,\dots,n}.$$	So every term occurring in $(\fff d_i\phi)^{\sigma}(\tilde{a})$ is canceled.
	
	\emph{Step 2.} The equation (ii) is obvious.  We prove the equation (iii). For $i=1,\dots,(n-1)$,  we have $$(\fff d_\h^i\phi)^\sigma(\tilde{a})=\sum_{\bar{m}\in\Part{n},\xi\in\Seq(\sigma,\bar{m})}(-1)^{\bar{m}+\xi}c^{\sigma,\bar{m},A_n}(d_\h^i\phi)^{U_0}(\xi).$$
	Let $\bar{m}=(m_k,\dots,m_1)$. Assume $\xi=(\xi_1,\dots,\xi_n)\in\Seq(\sigma,\bar{m})$, we have   $$(d_\h^i\phi)^{U_0}=\phi^{U_0}(\xi_1,\dots,\xi_{i}\xi_{i+1},\dots,\xi_n).$$
	Let $\underline{\xi}=(\underline{\xi}_1,\dots,\underline{\xi}_n)\in \underline{\Seq}(\sigma,\bar{m})$ be the formal sequence of $\xi$. Then
	$(\underline{\xi}_i,\underline{\xi}_{i+1})$ can only be one of the following cases
	$$(\underline{\xi}_i,\underline{\xi}_{i+1})=\left[\begin{matrix}
	(r_j,s_l) \text{ or } (s_l,r_j)  &\text{ for } r\in\ppp(\sigma[m_t]), s\in\ppp(\sigma[m_{t+1}]); \\
	(\tilde{a}[m_t],r_j) \text{ or } (r_j,\tilde{a}[m_t])  &\text{ for } r\in\ppp(\sigma[m_{t+1}]);\ \ \ \ \ \ \ \ \ \ \ \ \ \ \ \ \ \ \  \\
	(\tilde{a}[m_t],\tilde{a}[m_{t+1}]) &\text{ for some } t;   \ \ \ \ \ \ \ \ \ \ \ \ \ \ \ \ \ \  \ \ \ \ \ \ \ \ \ \ \ \ \   \\
	(r_{m_t-1},\tilde{a}[m_t])  &\text{ for } r\in\ppp(\sigma[m_{t}]). \ \ \ \ \ \ \ \ \ \ \ \ \ \ \ \ \ \  \ \ \ \   \\
	\end{matrix}\right.$$

	\emph{Case 1.} Assume that $(\underline{\xi}_i,\underline{\xi}_{i+1})=(r_j,s_l) \text{ or } (s_l,r_j)$, for some $r=(r{_1,\dots,r_{m_t-1}})\in\ppp(\sigma[m_t])$ and  $s=(s_1,\dots,s_{m_{t+1}-1})\in\ppp(\sigma[m_{t+1}])$. There exists a unique element $\xi'\in\Seq(\sigma,\bar{m})$ such that its formal sequence satisfies $$\underline{\xi'}=(\underline{\xi}_1,\dots,\underline{\xi}_{i-1},\underline{\xi}_{i+1},\underline{\xi}_{i},\underline{\xi}_{i+2},\dots,\underline{\xi}_n).$$ 
	It is obvious that $(-1)^{\xi'}=-(-1)^\xi$, hence
	$$	(-1)^{\bar{m}+\xi}c^{\sigma,\bar{m},A_n}(d_\h^i\phi)^{U_0}(\xi)+(-1)^{\bar{m}+\xi'}c^{\sigma,\bar{m},A_n}(d_\h^i\phi)^{U_0}(\xi')=0.$$
	The same argument applies for the cases $(\underline{\xi}_i,\underline{\xi}_{i+1})=	(\tilde{a}[m_t],r_j) \text{ or } (r_j,\tilde{a}[m_t]) $. \\
	
	\emph{Case 2.}   Assume $(\underline{\xi}_i,\underline{\xi}_{i+1})= (\tilde{a}[m_t],\tilde{a}[m_{t+1}]) $ for some $1\le t\le k-1$. Without loss of generality we assume that
	$$\underline{\xi}=(\underline{\xi}_1,\dots,\underline{\xi}_j,r,s,\tilde{a}[m_t],\tilde{a}[m_{t+1}],\underline{\xi}_{j+m_t+m_{t+1}+1},\dots,\underline{\xi}_n)$$
	for some paths  $r\in\ppp(\sigma[m_t]), s\in\ppp(\sigma[m_{t+1}])$.
	
	Denote  $\gamma=\sigma[m_{t}]\sqcup\sigma[m_{t+1}]$  the concatenation of the simplices $\sigma[m_{t+1}]$ and $\sigma[m_t]$, then $(c^{\gamma,m_{t+1}},r,s)$ is a path in $\ppp(\gamma)$.	
	We have $$(\fff d_\h^0\phi)^\sigma(\tilde{a})=\sum_{\bar{m}'\in\Part{n}}\sum_{\xi'\in \Seq(\sigma,\bar{m}')}(-1)^{\bar{m}'+\xi'}c^{\sigma,\bar{m}',A_n}(d_\h^0\phi)^{U_0}(\xi').$$
	Consider the partition $\bar{m}'=(m_k,\dots,m_{t+1}+m_t,\dots,m_1)$,  there exists a unique element $\xi'\in\Seq(\sigma,\bar{m}')$ such that its formal sequence satisfies
	\begin{align*}
	\underline{\xi'} &=(c^{\gamma,m_{t+1}},\underline{\xi}_1,\dots,\underline{\xi}_j,r,s,\tilde{a}[m_t],\tilde{a}[m_{t+1}],\underline{\xi}_{j+m_t+m_{t+1}+1},\dots,\underline{\xi}_n)\\
	&=(c^{\gamma,m_{t+1}},\underline{\xi}).
	\end{align*}   
	
	By computation, we get $(-1)^{\bar{m}'+\xi'}=-(-1)^{\bar{m}+\xi}$. Thus, we obtain  $$(-1)^{\bar{m}'+\xi'}c^{\sigma,\bar{m}',A_n}(d_\h^0\phi)^{U_0}(\xi')+(-1)^{\bar{m}+\xi}c^{\sigma,\bar{m},A_n}(d_\h^i\phi)^{U_0}(\xi)=0.$$

	\emph{Case 3.} Assume that $(\underline{\xi}_i,\underline{\xi}_{i+1})=(r_{m_t-1},\tilde{a}[m_t])  \text{ for some } r=(r_1,\dots,r_{m_t-1})\in\ppp(\sigma[m_{t}])$.  We have $r_{m_t-1}=\epsilon^{\sigma[m_t],j}$ for some $1\le j\le m_t-1$. Let $j'=n+1-(m_k+\dots+m_{t+1}+j)=m_1+\dots+m_t+1-j$.
	In the right hand side of equation (iii), we have 
	\begin{align*}
	(-1)^{j'}(\delta_{j'}\fff \phi)^\sigma(\tilde{a})&=(-1)^{j'}(\fff \phi)^{\partial_{n-j'}\sigma}(\partial_{j'}\tilde{a})\\
	&=\sum_{\bar{m}'\in\Part{n-1}}\sum_{\xi'\in \Seq(\partial_{n-j'}\sigma,\bar{m}')}(-1)^{j'}(-1)^{\bar{m}'+\xi'}c^{\partial_{n-j'}\sigma,\bar{m}',A_n}\phi^{U_0}(\xi').
	\end{align*}
	Choose $\bar{m}'=(m_k,\dots,m_t-1,\dots,m_1)\in\Part{n-1}$. There exists a unique element $\xi'\in\Seq(\partial_{n-j'}\sigma,\bar{m}')$ such that 
	$$(-1)^{j'}(-1)^{\bar{m}'+\xi'}c^{\partial_{n-j}\sigma,\bar{m}',A_n}\phi^{U_0}(\xi')=(-1)^{\bar{m}+\xi}c^{\sigma,\bar{m},A_n}(d_\h^i\phi)^{U_0}(\xi).$$
	After considering all cases 1,2,3 as above, we find that all the terms occurring in $\sum_{i=1}^{n-1}(\fff d_\h^i\phi)^\sigma(\tilde{a}) $ and $\sum_{i=1}^{n-1}(\delta_i\fff\phi)^\sigma(\tilde{a})$ are canceled. The remaining terms in $(\fff d_\h^0 \phi)^\sigma(\tilde{a})$ are only
	$$\sum_{\bar{m}'=(m'_k,\dots,m'_2,1)\in\Part{n}}\sum_{\xi'\in \Seq(\sigma,\bar{m}')}(-1)^{\bar{m}'+\xi'}c^{\sigma,\bar{m}',A_n}(d_\h^0\phi)^{U_0}(\xi')$$
	which are in turn canceled by all the terms in $(\delta_0\fff\phi)^\sigma(\tilde{a})$. We conclude that the equation (iii) holds.
	
	In the general case, we consider $\phi\in \cgs{p,n-1-p}$ for $p>0$. Applying the same arguments as above, we can prove the following equations hold true:  
	\begin{itemize}
		\item [(i')]  $-(-1)^{n-p}\fff d_\h^{n-p}\phi=(-1)^{n+1-p}\fff d_\s^{p+1}\phi+\sum_{i=2}^{n-p}\fff d_i\phi $;\\
		\item[(ii')] $ \fff d_\s^i\phi=\delta_{n-i}\fff \phi$ for $i=0,\dots,p$;\\
		\item[(iii'')]  $\fff d_\h^0\phi+\sum_{i=1}^{n-p-1}(-1)^i\fff d_\h^i\phi=\sum_{i=0}^{n-p-1}(-1)^i\delta_i\fff\phi$.
	\end{itemize}
	These equations yield 
	$$\sum_{i=0}^n(-1)^i \fff d_\h^i\phi+(-1)^n\fff (\sum_{i=0}^{p+1}d^i_\s)\phi+\sum_{i=2}^n\fff d_i\phi=\sum_{i=0}^n\delta_n\fff\phi,
	$$	
	which means $\fff(d\phi)=\delta(\fff\phi)$.
\end{proof}

\subsection{The cochain map $\GGG$}
In this section we define the cochain map $$\GGG:\ctil{\bullet} \lra \cgs{\bullet}.$$ Consider a $p$-simplex $\sigma = (u_1, \dots, u_p) \in \nnn_p(\uuu)$ as follows
 \begin{equation}\label{eqsigma3}
 \sigma = (\xymatrix{ {U_0} \ar[r]_-{u_1} & {U_{1}} \ar[r]_-{u_2} & {\dots} \ar[r]_-{u_{p-1}} & {U_{p-1}} \ar[r]_{u_p} & {U_p}})
 \end{equation}
and a $q$-simplex $a = (a_1, \dots, a_q) \in \nnn(\aaa(U_p))_q$ as follows
 \begin{equation}\label{equationa3}
     a=(\xymatrix{{A_0}\ar[r]^-{a_q}&{A_1} \ar[r]^-{a_{q-1}}&\cdots \ar[r]^-{a_2}&A_{q-1}\ar[r]^-{a_1}& A_q  }).
\end{equation}
Using conditioned shuffles, we will describe several ways to build a $(p+q)$-simplex in $\nnn_{p+q}(\tilde{\aaa})$ from these data. Let $\bar{m} = (m_k, \dots, m_1)$ be a partition of $p$ with $m_i \geq 1$ for all $i$ and let $\beta \in \bar{S}_{\bar{m}}$ be a conditioned $\bar{m}$-shuffle as defined in \S \ref{parshuffle}. For $1 \leq i \leq k$, let $r^i = (r^i_1, \dots, r^i_{m_i -1}) \in \ppp(\sigma[m_i])$ be a path and consider the associated $m_i$-simplex
\begin{align}\label{expath}
\bar{r}^i = (1_{\sigma[m_i]^{\ast}}, r^i_1, \dots, r^i_{m_i -1}) \in \nnn_{m_i}(\ccc_i).
\end{align}
where $$\ccc_i = \Fun(\aaa(U_{p - m_1 \dots - m_{i-1}}), \aaa(U_{p - m_1 \dots - m_i}).$$
First, consider the formal shuffle by $\beta$ of the associated tuples $({\bar{r}}^i)_i$ as described in \eqref{action1}. Assume that
$$\beta^{(0)}(({\bar{r}}^i)_i) = \underline{s} = (\underline{s}_1, \dots, \underline{s}_p).$$
{Since $\beta$ is a conditioned shuffle, there are uniquely determined numbers $\gamma_l \geq 1$, $1 \leq l \leq k$ such that $\underline{s}_1 = 1_{\sigma[m_1]^{\ast}}$, $\underline{s}_{\gamma_1 +1} = 1_{\sigma[m_2]^{\ast}}$, \dots, $\underline{s}_{\sum_{i = 1}^l \gamma_i + 1} = 1_{\sigma[l+1]^{\ast}}$, \dots, 
$\underline{s}_{\sum_{i = 1}^{k-1} \gamma_i + 1} = 1_{\sigma[m_k]^{\ast}}$ and $\gamma_k = p - \sum_{i = 1}^{k-1} \gamma_i$.
Following the pattern explained at the end of \S \ref{parshuffle}, we obtain the sequence 
\begin{equation*}
(\hat{c}^1, \dots, \hat{c}^k) \in \prod_{l = 1}^k \nnn_{\gamma_l}(\prod_{i = 1}^l \ccc_i).
\end{equation*}
 Using the composition of functors as in Remark \ref{shufflecompositionfunctor}, we obtain the following sequence which we define as the shuffle product of $(\bar{r}^i)_i$ by $\beta$  
\begin{equation*} 
\beta((\bar{r}^i)_i):=(\bar{c}^1, \dots, \bar{c}^k) \in \prod_{l = 1}^k \nnn_{\gamma_l}(\mathcal{D}_l)
\end{equation*}
where  $$\mathcal{D}_l = \Fun(\aaa(U_{p }), \aaa(U_{p - m_1 \dots - m_l}).$$
We denote by $\Seqq(\sigma,\bar{m})$ the set of all such conditioned shuffle products. Thus 
$$\Seqq(\sigma,\bar{m})=\{\beta((\bar{r}^i)_i)|\  \beta\in\bar{S}_{\bar{m}},\  \bar{r}^i=(1_{\sigma[m_i]^*},r^i),\ r^i\in\mathcal{P}(\sigma[m_i])\}.$$
For each $\zeta=\beta((\bar{r}^i)_i)\in\Seqq(\sigma,\bar{m})$, we denote the formal sequence $\beta^{(0)}((\bar{r}^i)_i)$ of $\zeta$ by $\underline{\zeta}$, and denote the set of all such formal sequence as
$$\underline{\Seqq}(\sigma,\bar{m})=\{\underline{\zeta}|\ \zeta\in\Seqq(\sigma,\bar{m})\}.$$
We define $$\mathrm{sign}(\zeta)=\mathrm{sign}(\beta((\bar{r}^i)_i))=(-1)^\beta\prod_{i=1}^k(-1)^{r^i}$$
and  equip this shuffle product with a certain underlying simplex denoted by $\simp(\beta((\bar{r}^i)_i))$. Writing $$\bar{c}^l=(\bar{c}^l_1,\dots,\bar{c}^l_{\gamma_l})$$ we define 
\begin{align*}
\simp(\bar{c}^l_1)&=(\xymatrix{U_{p - m_1 \dots - m_l} \ar[r]^-{||\sigma[m_l]||} & U_{p - m_1 \dots - m_{l-1}}});\\
\simp(\bar{c}^l_j)&=(\xymatrix{U_{p - m_1 \dots - m_l} \ar[r]^-{1} & U_{p - m_1 \dots - m_{l}}}),\ \ j>1.
\end{align*}
The simplex $\simp(\bar{c}^l)$ is obtained by concatenation of $(\simp(\bar{c}^l_j))_j$, the simplex $\simp(\beta((\bar{r}^i)_i))\equiv\simp(\bar{c}^1,\dots,\bar{c}^k)$ is obtained by concatenation of the simplices $(\simp(\bar{c}^l))_l$.

\begin{example}{\label{exshuffle4}}
		Consider the simplex $\sigma=(U_0\underset{u_1}{\lra}U_1\underset{u_2}{\lra}U_2\underset{u_3}{\lra}U_3\underset{u_1}{\lra}U_4)$ and the partition $\bar{m}=(m_2,m_1)=(2,2)$. 	 There are three conditioned formal shuffles $(1_{\sigma[m_1]^*}),c^{u_3,u_4},1_{\sigma[m_2]^*}, c^{u_1,u_2})$; $(1_{\sigma[m_1]^*}),1_{\sigma[m_2]^*},c^{u_3,u_4}, c^{u_1,u_2}); (1_{\sigma[m_1]^*}),1_{\sigma[m_2]^*},c^{u_1,u_2}, c^{u_3,u_4})  $. 
	 The set $\Seqq(\sigma,(m_2,m_1))$ consists of following sequences:
		$$\begin{pmatrix}
		\bullet&\overset{c^{u_1,u_2}u_3^*u_4^*}{\lra}&\bullet&\overset{1_{\sigma[m_2]^*}u_3^*u_4^*}{\lra}&\bullet&\overset{c^{u_3,u_4}}{\lra}&\bullet&\overset{1_{\sigma[m_1]^*}}{\lra}&\bullet\\
		U_0&\underset{1}{\lra}& U_0&\underset{u_2u_1}{\lra}& U_2&\underset{1}{\lra}& U_2 &\underset{u_4u_3}{\lra} &U_4
		\end{pmatrix}
		$$
		$$\begin{pmatrix}
		\bullet&\overset{c^{u_1,u_2}u_3^*u_4^*}{\lra}&\bullet&\overset{(u_2u_1)^*c^{u_3,u_4}}{\lra}&\bullet&\overset{1_{\sigma[m_2]^*}(u_4u_3)^*}{\lra}&\bullet&\overset{1_{\sigma[m_1]^*}}{\lra}&\bullet\\
		U_0&\underset{1}{\lra}& U_0 &\underset{1}{\lra}& U_0&\underset{u_2u_1}{\lra}&  U_2 &\underset{u_4u_3}{\lra} &U_4
		\end{pmatrix}
		$$
		$$\begin{pmatrix}
		\bullet&\overset{u_1^*u_2^*c^{u_3,u_4}}{\lra}&\bullet&\overset{c^{u_1,u_2}(u_4u_3)^*}{\lra}&\bullet&\overset{1_{\sigma[m_2]^*}(u_4u_3)^*}{\lra}&\bullet&\overset{1_{\sigma[m_1]^*}}{\lra}&\bullet\\
		U_0&\underset{1}{\lra}& U_0 &\underset{1}{\lra}& U_0&\underset{u_2u_1}{\lra}&  U_2 &\underset{u_4u_3}{\lra} &U_4
		\end{pmatrix}.
		$$	
		
\end{example}

Next consider a shuffle permutation $\omega \in S_{p,q}$. We are now to define the shuffle product of $a$ and $(\bar{c}^1, \dots, \bar{c}^k)$ by $\omega$ to be the element
$$(\hat{b}^0,\hat{b}^1,\dots,\hat{b}^k)\equiv a \underset{\omega}{\ast} (\bar{c}^1, \dots, \bar{c}^k)  \in \nnn_{p+q}(\tilde{\aaa}).$$ 
The formal shuffle product $\omega^{(0)}(a,\beta^{(0)}(({\bar{r}}^i)_i))$ is called the formal sequence of $(\hat{b}^0,\hat{b}^1,\dots,\hat{b}^k)$.
First consider the formal shuffle 
$$\omega^{(0)}(a;(\bar{c}^1, \dots, \bar{c}^k)) = (b_1, \dots, b_{p + q}).$$
Since $\omega$ is shuffle, there are unique numbers $t_1,\dots,t_{k+1}$ such that $b_{t_1+1}=\bar{c}^1_1$, $b_{t_1+t_2+1}=\bar{c}^2_1$ , $\dots$ , $b_{\sum_{i=1}^kt_i +1}=\bar{c}^k_1$ and $t_{k+1}=p+q-\sum_{i=1}^kt_i $. Following the procedure at the end of section 3.1, for $0\le l\le k$ consider 
$$a^l=(a^l_1,\dots,a^l_{j_l})=\{b_{\sum_{i=1}^lt_i+1},\dots,b_{\sum_{i=1}^{l+1}t_i}\}\cap\{a_1,\dots,a_q\}.$$ 
Obviously $a^0=(a_1,\dots,a_{t_1})$.
There is unique shuffle $\omega_l\in S_{j_l,\gamma_l}$ such that the formal shuffle product of $a^l$ and $\bar{c}^l$ by $\omega$ is exactly
$$(b_{\sum_{i=1}^lt_i+1},\dots,b_{\sum_{i=1}^{l+1}t_i}).$$ 
Now we put $\hat{b}^0=a^0$. For $l=1 \dots k$, take the shuffle product $\omega_l(a^l,\hat{c}^l)$ with respect to evaluation of functors as in Example \ref{exshuffle}, and put
$$\hat{b}^l=(\hat{b}^l_1,\dots,\hat{b}^l_{j_l+\gamma_l})=\omega_l(a^l,\hat{c}^l).$$

Now we associate the underlying simplex to $(\hat{b}^0,\hat{b}^1,\dots,\hat{b}^k)$ to show that 
$$(\hat{b}^0,\hat{b}^1,\dots,\hat{b}^k)\in \nnn_{p+q}(\tilde{\aaa})(\sigma^{\star}A_0, A_q).$$

We have $\hat{b}^l_1=\sigma[m_l]^*T_{l-1}(A_{\alpha_l})$ for a certain $T_{l-1}\in \mathcal{D}_{l-1}$ and a certain $A_{\alpha_l}\in\{A_0,\dots,A_q\}$.  Thus it can be regarded as an element of $\nnn_1(\tilde{\aaa})$ as follows:
$$\xymatrix{ \sigma[m_l]^{\ast}T_{l-1}A_{\alpha_l} \ar[r]^-{1} & T_{l-1}A_{\alpha_l} 
	\\
	U_{p - m_1 \dots - m_l} \ar[r]^-{||\sigma[m_l]||} & U_{p - m_1 \dots - m_{l-1}.}
 }$$
We consider $\hat{b}^l_j=\aaa_{U_{p - m_1 \dots - m_l}}(B,B')$ as an element of 
$\nnn_1(\tilde{\aaa})$ as follows:

$$\xymatrix{ B \ar[r]^-{\hat{b}^l_j} & B'
	\\
	U_{p - m_1 \dots - m_l} \ar[r]^-{1} & U_{p - m_1 \dots - m_{l}.}
}$$
 Put 
\begin{align*}
\simp(\hat{b}^l_1)&=(\xymatrix{U_{p - m_1 \dots - m_l} \ar[r]^-{||\sigma[m_l]||} & U_{p - m_1 \dots - m_{l-1}}}),\ \ \  l\ge 1;\\
\simp(\hat{b}^l_j)&=(\xymatrix{U_{p - m_1 \dots - m_l} \ar[r]^-{1} & U_{p - m_1 \dots - m_{l}}}),\ \ j>1;\\
\simp(\hat{b}^0_j)&=(\xymatrix{U_{p} \ar[r]^-{1} & U_{p}}),\ \ \ j\ge 0.
\end{align*}
By concatenation of all these 1-simplices we obtain 
the simplex $\simp(\hat{b}^0,\hat{b}^1,\dots,\hat{b}^k)$ of $(\hat{b}^0,\hat{b}^1,\dots,\hat{b}^k)$. 

\begin{example} Let $\sigma=(U_0\underset{u_1}{\lra}U_1\underset{u_2}{\lra}U_2)$ and $a\in\aaa(U_2)(A_0,A_1)$. Let $\bar{m}=(2)$, then $\Seqq(\sigma,(2))$  consist only of the sequence $(1_{(u_2u_1)^*},c^{u_1,u_2})$: 
$$\begin{pmatrix}
\bullet&\overset{c^{u_1,u_2}}{\lra}&\bullet&\overset{1_{(u_2u_1)^*}}{\lra}&\bullet\\
U_0&\underset{1}{\lra}& U_0&\underset{u_2u_1}{\lra}& U_2
\end{pmatrix}.
$$	
	 The following are shuffle products of ${a}$ and $(1_{u_2u_1},c^{u_1,u_2})$:
		$$\begin{pmatrix}
	u_1^*u_2^*A_0&\overset{c^{u_1,u_2,A_0}}{\lra}&(u_2u_1)^*A_0&\overset{1_{(u_2u_1)^*}(A_0)}{\lra}&A_0&\overset{a}{\lra}&A_1\\
	U_0&\underset{1}{\lra}& U_0&\underset{u_2u_1}{\lra}& U_2&\underset{1}{\lra}& U_2 
	\end{pmatrix}
	$$	
	$$\begin{pmatrix}
	u_1^*u_2^*A_0&\overset{c^{u_1,u_2,A_0}}{\lra}&(u_2u_1)^*A_0&\overset{(u_2u_1)^*a}{\lra}&(u_2u_1)^*A_1&\overset{1_{(u_2u_1)^*}(A_1)}{\lra}&A_1\\
	U_0&\underset{1}{\lra}& U_0&\underset{1}{\lra}& U_0&\underset{u_2u_1}{\lra}& U_2 
	\end{pmatrix}
	$$
	$$\begin{pmatrix}
	u_1^*u_2^*A_0&\overset{u_1^*u_2^*a}{\lra}&u_1^*u_2^*A_1&\overset{c^{u_1,u_2,A_1}}{\lra}&(u_2u_1)^*A_1&\overset{1_{(u_2u_1)^*}(A_1)}{\lra}&A_1\\
	U_0&\underset{1}{\lra}& U_0&\underset{1}{\lra}& U_0&\underset{u_2u_1}{\lra}& U_2 
	\end{pmatrix}.
	$$
\end{example}

For each cochain $\psi \in \CC^{p + q}_{\uuu}(\tilde{\aaa}, \tilde{M})$, we now define

$$\GGG(\psi)^\sigma(a)=\sum_{\substack{\bar{m}\in\Part{n}\\ \zeta\in\Seqq(\sigma,\bar{m})}}
\sum_{\beta\in S_{q,p}}(-1)^\beta(-1)^\zeta\psi^{\simp({a}\underset{\beta}{*}\zeta)}({a}\underset{\beta}{*}\zeta)$$

}

\begin{proposition}\label{Gcommutes}
	The map $\GGG$ commutes with differentials. Precisely, for $\psi\in\ctil{n-1}$, we have $dG(\psi)=G\delta(\psi)$.
\end{proposition}
\begin{proof}
	Assume that  $p+q=n$. Let  $\sigma=(u_1,\dots,u_p)$ be a $p$-simplex  as in (\ref{eqsigma3}) and $a=(a_1,\dots,a_q)$ as in (\ref{equationa3}). We prove $(dG(\psi))^\sigma(a)=(G\delta(\psi))^\sigma(a)$.\\ Put 
	\begin{eqnarray*}
	\mathrm{LHS}&=&(d_0\GGG\psi)^\sigma(a)+(-1)^n(d_\s\GGG\psi)^\sigma(a)+(d_2\GGG\psi)^\sigma(a))+\cdots+(d_p\GGG\psi)^\sigma(a);\\
	\mathrm{RHS}&=&(\GGG\delta_0\psi)^\sigma(a)-(\GGG\delta_1\psi)^\sigma(a)+\cdots+(-1)^n(\GGG\delta_n\psi)^\sigma(a).
	\end{eqnarray*}
We have 
		  $$(-1)^i(\GGG\delta_i\psi)^\sigma(a)=\sum_{\substack{\bar{m}\in\Part{p},\ \beta \in S_{q,p}\\ \zeta\in \Seqq(\sigma,\bar{m})  }}(-1)^i(-1)^\beta(-1)^\zeta(\delta_i\psi)^{\simp({a}\underset{\beta}{*}\zeta)}({a}\underset{\beta}{*}\zeta).$$
		  Denote 
		  $$T(i,\bar{m},\beta,\zeta)=(-1)^i(-1)^\beta(-1)^\zeta(\delta_i\psi)^{\simp({a}\underset{\beta}{*}\zeta)}({a}\underset{\beta}{*}\zeta).$$
		  	To prove that $\mathrm{LHS}=\mathrm{RHS}$, we show that  each term $T$ appearing in the expansion of $\mathrm{RHS}$ is either  matched with a unique term  in the expansion of $\mathrm{LHS}$ or canceled out with a term $-T$ in $\mathrm{RHS}$. Simultaneously, this process also shows that every term in
LHS is cancelled out.

	Take a partition $\bar{m}=(m_k,\dots,m_1)\in\Part{p}$. Fix $\beta\in S_{q,p}$ and $ \zeta\in \Seqq(\sigma,\bar{m})$. By definition, there are a unique $\gamma\in \bar{S}_{\bar{m}}$, $r^{m_i}=(r^{m_i}_1,\dots,r^{m_i}_{m_i-1})\in\ppp(\sigma[m_i])$, $i=1,\dots,k$; such that $\zeta$ is the shuffle product
	\begin{align}\label{equationzeta}
	\zeta=\gamma\big((\bar{r}^{m_i})_{i=1,\dots,k}\big)
	\end{align}
	where $\bar{r}^{m_i}=(1_{\sigma[m_i]},r^{m_i})$ as in (\ref{expath}).
	
	We denote the shuffle product
		$${a}\underset{\beta}{*}\zeta=\alpha=(\alpha_1,\dots,\alpha_n)$$
	and its formal sequence $$\underline{\alpha}=(\underline{\alpha}_1,\dots,\underline{\alpha}_n).$$

	\emph{Step 1.} We consider the term $T(0,\bar{m},\beta,\zeta)$ in $\mathrm{RHS}$. We have  
		$$(\delta_0\psi)^{\simp({a}\underset{\beta}{*}\zeta)}(a\underset{\beta}{*}\zeta)=\mu(\alpha_1,\psi^{\simp(\partial_0\alpha)}(\alpha_2,\dots,\alpha_n))$$
		where $\mu$ is the composition in the map-graded category $\tilde{\aaa}$.  
		  There are only three cases $\underline{\alpha}_1={a}_1$, $\underline{\alpha}_1=1_{u_p}$ or $\underline{\alpha}_1=1_{\sigma[m_1]^*}$ where $m_1\ge 2$.
	\begin{itemize}
		\item Consider the case $\underline{\alpha}_1={a}_1$. We have $\simp(\alpha_1)=(U_p\underset{1}{\lra} U_p)$. In the $\mathrm{LHS}$, we consider 
		\begin{eqnarray*}
		(d_\h^0\GGG\psi)^\sigma(a)&=&\sigma^*(a_1)(\GGG\psi)^\sigma(a_2,\dots,a_q)\\
		&=&\sum_{\substack{\bar{m}'\in\Part{p},\ \beta' \in S_{q-1,p}\\ \zeta'\in \Seqq(\sigma,\bar{m}')  }}
		(-1)^{\beta'}(-1)^{\zeta'}\sigma^*(a_1)\psi^{\simp({\partial_0a}\underset{\beta'}{*}\zeta')}({\partial_0a}\underset{\beta'}{*}\zeta').
		\end{eqnarray*}
		Choose $\bar{m}'=\bar{m}$ and $\zeta'=\zeta\in\Seqq(\sigma,\bar{m}')$. Then there exists  a unique  $\beta'\in S_{q-1,p}$ such that  $$({a}_2,\dots,{a}_q)\underset{\beta'}{*}\zeta'=(\alpha_2,\dots,\alpha_n).$$
		We have $(-1)^{\beta'}=(-1)^\beta$. Hence $$T(0,\bar{m},\beta,\zeta)=(-1)^{\beta'+\zeta'}\sigma^*(a_1)\psi^{\simp({\partial_0a}\underset{\beta'}{*}\zeta')}({\partial_0a}\underset{\beta'}{*}\zeta').$$		
		\item Consider the case $\underline{\alpha}_1=1_{u_p}$. We have $m_1=1$, 
		 $\simp(\alpha_1)=(U_{p-1}\underset{u_p}{\lra} U_p)$, and 
				 
		 $$(\delta_0\psi)^{\simp({a}\underset{\beta}{*}\zeta)}({a}\underset{\beta}{*}\zeta)=c^{\sigma,p-1,A_q}\psi^{\simp(\partial_0\alpha)}(\alpha_2,\dots,\alpha_n).$$
		 In the $\mathrm{LHS}$, we have 
		 \begin{eqnarray*}
		    &&(-1)^{n+p}(d_\s^p\GGG\psi)^\sigma(a)=(-1)^{n+p}c^{\sigma,p-1,A_q}(\GGG\psi)^{\partial_p\sigma}(u_p^*a)\\
		    &&=\sum_{\substack{\bar{m}'\in\Part{p-1},\ \beta' \in S_{q,p-1}\\ \zeta'\in \Seqq(\sigma,\bar{m}')  }}(-1)^{n+p+	\beta'+\zeta'}c^{\sigma,p-1,A_q}\psi^{\simp(u_p^*a\underset{\beta'}{*}\zeta')}(u_p^*a\underset{\beta'}{*}\zeta').
		    \end{eqnarray*}
		    Choose $\bar{m}'=(m_k,\dots,m_2)\in\Part{p-1}$. There exist unique $\zeta'\in\Seqq(\partial_p\sigma,\bar{m}')$ and $\beta'\in S_{q,p-1}$ such that  $${u_p^*a}\underset{\beta'}{*}\zeta'=(\alpha_2,\dots,\alpha_n).$$
		    Note that $(-1)^q(-1)^{\beta'}=(-1)^\beta $, so $(-1)^{n+p+	\beta'+\zeta'}=(-1)^\beta(-1)^\zeta$. This implies
		    $$T(0,\bar{m},\beta,\zeta)=(-1)^{n+p+	\beta'+\zeta'}c^{\sigma,p-1,A_q}\psi^{\simp(u_p^*a\underset{\beta'}{*}\zeta')}(u_p^*a\underset{\beta'}{*}\zeta').$$
		    
		    \item Consider the case $\underline{\alpha}_1=1_{\sigma[m_1]^*}$. We have $\simp(\alpha_1)=(U_{p-m_1}\underset{u_p\cdots u_{p-m_1+1}}{\lra}U_p)$, and 
		     $$(\delta_0\psi)^{\simp({a}\underset{\beta}{*}\zeta)}({a}\underset{\beta}{*}\zeta)=c^{\sigma,p-m_1,A_q}\psi^{\simp(\partial_0\alpha)}(\alpha_2,\dots,\alpha_n).$$
		  We have,  in $\mathrm{RHS}$, the terms
		    \begin{align*}
		    	(d_{m_1}\GGG\psi)^\sigma(a)&=\sum_{\substack{r\in\ppp(\sigma[m_1]),\ \beta'\in S_{q,m_1-1}}}(-1)^q(-1)^r(-1)^{\beta'}c^{\sigma,p-m_1,A_q}(\GGG\psi)^{L_{p-m_1}\sigma}(\beta'(a,r))\\
		    &=\sum_{\substack{r\in\ppp(\sigma[m_1]),\ \beta'\in S_{q,m_1-1}   ,\ \beta''\in S_{q+m_1-1,p-m_1}\\\bar{m}'\in\Part{p-m_1},\ \zeta'\in\Seqq(L_{p-m_1}\sigma,\bar{m}') }}(-1)^{q+r+\beta'+\beta''+\zeta'} \times\\
		    & \ \ \ c^{\sigma,p-m_1,A_q}\psi^{\simp(\beta'(a,r)\underset{\beta''}{*}\zeta')}(\beta'(a,r)\underset{\beta''}{*}\zeta').
		    \end{align*} 
		  
		    Let  $\bar{m}'=(m_k,\dots,m_2)\in\Part{p-m_1}$. We consider the element $\zeta'$ in $\Seqq(L_{p-m_1}\sigma,\bar{m}')$ of the form
		    $$\zeta'=\gamma_1(\bar{r}^{m_2},\dots,\bar{r}^{m_k})$$
		    where $\gamma_1\in\bar{S}_{\bar{m}'}$. Choose $r=r^{m_1}$, there exist unique $\gamma_1\in\bar{S}_{\bar{m}'},\ \beta'\in S_{q,m_1-1},\ \beta''\in S_{q+m_1-1,p-m_1}S$ such that 
		    $$\beta'(a,r)\underset{\beta''}{*}\zeta'=(\alpha_2,\dots,\alpha_n).$$
		    By inspection on signs, we get $(-1)^{q+r+\beta'+\beta''+\zeta'}=(-1)^\beta(-1)^\zeta$. Therefore
		    $$T(0,\bar{m},\beta,\zeta)=(-1)^{q+r+\beta'+\beta''+\zeta'}c^{\sigma,p-m_1,A_q}\psi^{\simp(\beta'(a,r)\underset{\beta''}{*}\zeta')}(\beta'(a,r)\underset{\beta''}{*}\zeta').$$
	\end{itemize}
	
	\emph{Step 2.} We consider the term $T(n,\bar{m},\beta,\zeta)$ in $\mathrm{RHS}$. We have  
		$$(-1)^n(\delta_n\psi)^{\simp({a}\underset{\beta}{*}\zeta)}({a}\underset{\beta}{*}\zeta)=(-1)^n\mu(\psi^{\simp(\partial_0\alpha)}(\alpha_1,\dots,\alpha_{n-1}),\alpha_n).$$
		  There are only three cases: $\underline{\alpha}_n={a}_n$, $\underline{\alpha}_n=1_{u_1^*}$ or $\underline{\alpha}_n=r^{m_i}_{m_i-1}$ where $r^{m_i}=(r^{m_i}_1,\dots,r^{m_i}_{m_i-1})\in\ppp(\sigma[m_i])$.

	\begin{itemize}

		\item Consider the case $\underline{\alpha}_n={a}_q$. Then $\simp(\alpha_n)=(U_0\underset{1}{\lra}U_0)$. In $\mathrm{LHS}$, we have 
		$$(-1)^q(d_\h^q\GGG\psi)^\sigma(a)=\sum_{\substack{ \bar{m}'\in\Part{p},\beta'\in\ S_{q-1,p}\\ 
		\zeta'\in\Seqq(\sigma,\bar{m}')		 }}(-1)^{q+\beta'+\zeta'}\psi^{\simp({\partial_q a}\underset{\beta'}{*}\zeta')}(\partial_q a\underset{\beta'}{*}\zeta')\sigma^{\star}(a_q).$$
	Choose $\bar{m}'=\bar{m}\in\Part{p}$ and $\zeta'=\zeta\in\Seqq(\sigma,\bar{m}')$. There exists  a unique $\beta'\in S_{q-1,p}$ such that $$({a}_1,\dots,{a}_{q-1})\underset{\beta'}{*}\zeta'=(\alpha_1,\dots,\alpha_{n-1}).$$
	Note that  $(-1)^p(-1)^{\beta'}=(-1)^\beta$, so $(-1)^{q+\beta'+\zeta'}=(-1)^{n+\beta+\zeta}$. This implies $$T(n,\bar{m},\beta,\zeta)=(-1)^{q+\beta'+\zeta'}\psi^{\simp({\partial_qa}\underset{\beta'}{*}\zeta')}({\partial_qa}\underset{\beta'}{*}\zeta')\sigma^{\star}(a_q).$$

	\item Consider the case $\underline{\alpha}_n=(1_{u_1^*})$.  Then $m_k=1$ and $\simp(\alpha_n)=(U_0\underset{1}{\lra}U_1)$, so $\underline{\zeta}=(\underline{\eta},1_{u_1^*})$ for some $\underline{\eta}\in\underline{\Seqq}(\partial_0\sigma,(m_{k-1},\dots,m_1))$. We have 
	$$(\delta_n\psi)^{\simp({a}\underset{\beta}{*}\zeta)}({a}\underset{\beta}{*}\zeta)=c^{\sigma,1,A_q}{ u_1^*}\psi^{\simp(\partial_n\alpha)}(\alpha_1,\dots,\alpha_{n-1}).$$
	
	In $\mathrm{LHS}$ we have
\begin{eqnarray*}
 (-1)^n(d_\s^0\GGG\psi)^\sigma(a)&=&(-1)^nc^{\sigma,1,A_q}M^{u_1}(\GGG\psi)^{\partial_0\sigma}(a)\\
 &=& \sum_{\substack{\bar{m}'\in\Part{p-1}, \beta'\in S_{q,p-1}\\ \zeta'\in\Seqq(\partial_0\sigma,\bar{m}') }}(-1)^{n+\beta'+\zeta'}c^{\sigma,1,A_q}M^{u_1}\psi^{\simp({a}\underset{\beta'}{*}\zeta')}({a}\underset{\beta'}{*}\zeta').
\end{eqnarray*}
 Take $\bar{m}'=(m_{k-1},\dots,m_1)\in\Part{p-1}$ and $\zeta'=\eta$, there exists a unique $\beta'\in S_{q,p-1}$ such that 
	$$({a}_1,\dots,{a}_{q})\underset{\beta'}{*}\zeta'=(\alpha_1,\dots,\alpha_{n-1}).$$
	We have $(-1)^{\beta'}=(-1)^\beta$, so $(-1)^{n+\beta'+\zeta'}=(-1)^{n+\beta+\zeta}$, hence $$T(n,\bar{m},\beta,\zeta)=(-1)^{n+\beta'+\zeta'}c^{\sigma,1,A_q}M^{u_1}\psi^{\simp({a}\underset{\beta'}{*}\zeta')}({a}\underset{\beta'}{*}\zeta').$$
	
	\item Consider the case $\underline{\alpha}_n=r^{m_i}_{m_i-1}$. Then $\alpha_n=\epsilon^{\sigma,j_0}(A_0)$ for some $j_0$,  $\simp(\alpha_n)=(U_0\underset{1}{\lra}U_0)$. We have 
	$$(\delta_n\psi)^{\simp({a}\underset{\beta}{*}\zeta)}({a}\underset{\beta}{*}\zeta)=\psi^{\simp(\partial_n\alpha)}(\alpha_1,\dots,\alpha_{n-1})\epsilon^{\sigma,j_0}(A_0).$$
	 In $\mathrm{LHS}$ we have $$(-1)^{n+j_0}(d_\s^{j_0}\GGG\psi)^\sigma(a)=\sum_{\substack{\bar{m}'\in\Part{p-1}, \beta'\in S_{q,p-1}\\ \zeta'\in\Seqq(\partial_{j_0}\sigma,\bar{m}') }}(-1)^{n+j_0+\beta'+\zeta'}\psi^{\simp({a}\underset{\beta'}{*}\zeta')}({a}\underset{\beta'}{*}\zeta')\epsilon^{\sigma,j_0}(A_0).$$
	 Take $\bar{m}'=(m_k,\dots,m_{i+1},m_i-1,m_{i-1},\dots,m_1)\in\Part{p-1}$. There exist unique $\zeta'\in \Seqq(\partial_{j_0}\sigma,\bar{m}')$ and $\beta\in S_{q,p-1}$ such that 
		$$({a}_1,\dots,{a}_{q})\underset{\beta'}{*}\zeta'=(\alpha_1,\dots,\alpha_{n-1}).$$
	Since $(-1)^{n+j_0+\beta'+\zeta'}=(-1)^{n+\beta+\zeta}$, we get
	$$T(n,\bar{m},\beta,\zeta)=(-1)^{n+j_0+\beta'+\zeta'}\psi^{\simp({a}\underset{\beta'}{*}\zeta')}({a}\underset{\beta'}{*}\zeta')\epsilon^{\sigma,j_0}(A_0).$$
		
	\end{itemize}

	\emph{Step 3.} Considering the term $T(i,\bar{m},\beta,\zeta)$ in $\mathrm{RHS}$  for $i=1..(n-1)$, we have
	    $$(\delta_i\psi)^{\simp({a}\underset{\beta}{*}\zeta)}({a}\underset{\beta}{*}\zeta)=\psi^{\simp(\partial_i\alpha)}(\alpha_1,\dots,\alpha_{i-1},\alpha_i\alpha_{i+1},\alpha_{i+2},\dots,\alpha_{n}).$$
	 Denote 
	 $$\Gamma=\{1_{u_i^*},1_{\sigma[m_j]^*},r^{m_t}_{l}|\ r^{m_t}=(r^{m_t}_{1},\dots,r^{m_t}_{m_t-1})\in\ppp(\sigma[m_t]),\ i,j,t,l\ge 1  \}.$$
	  We consider the following case
	 \begin{itemize}
	 	\item[(i)] Assume that $\{\underline{\alpha}_i,\underline{\alpha}_{i+1}\}\cap\{\tilde{a}_1,\dots,\tilde{a}_n\}\ne\emptyset$ then:
	 	\begin{itemize}
	 		\item[(a)] If $(\underline{\alpha}_i,\underline{\alpha}_{i+1})=(a_j,a_{j+1})$ for some $j$, we look at $d_\h^j\GGG\psi$ in $\mathrm{LHS}$,
	 		$$(-1)^j(d_\h^j\GGG\psi)^\sigma(a)=\sum_{\substack{ \bar{m}'\in\Part{p},\beta'\in\ S_{q-1,p}\\ 
	 				\zeta'\in\Seqq(\sigma,\bar{m}')		 }}(-1)^{j+\beta'+\zeta'}\psi^{\simp({\partial_ja}\underset{\beta'}{*}\zeta')}({\partial_ja}\underset{\beta'}{*}\zeta').$$
	 		Choose $\bar{m}'=\bar{m}$ and $\zeta'=\zeta$. There exists a unique $\beta'\in S_{q-1,p}$ such that $${\partial_ja}\underset{\beta'}{*}\zeta'=\partial_i\alpha.$$
	 		Since $(-1)^{j+\zeta'+\beta'}=(-1)^{i+\beta+\zeta}$, we get
	 		$$T(i,\bar{m},\beta,\zeta)=(-1)^{j+\beta'+\zeta'}\psi^{\simp({\partial_ja}\underset{\beta'}{*}\zeta')}({\partial_ja}\underset{\beta'}{*}\zeta').$$
	 		\item[(b)] If $\{\underline{\alpha}_i,\underline{\alpha}_{i+1}\}=\{a_j,b\}$ for some $b\in\Gamma$, there exists a unique $\beta'\in S_{q,p}$ such that 
	 		$$\beta'^{(0)}({a},\zeta)=(\underline{\alpha}_1,\dots,\underline{\alpha}_{i-1},\underline{\alpha}_{i+1},\underline{\alpha}_i,\underline{\alpha}_{i+2},\dots,\underline{\alpha}_n).$$
	 		Note that $(-1)^{\beta'}=-(-1)^\beta$. This implies $$T(i,\bar{m},\beta,\zeta)+T(i,\bar{m},\beta',\zeta)=0.$$
	 	\end{itemize}
	 	\item[(ii)] Assume that $\{\underline{\alpha}_i,\underline{\alpha}_{i+1}\}\subseteq \Gamma$, then :
	 	\begin{itemize}
	 		\item[(a)] If $\{\underline{\alpha}_i,\underline{\alpha}_{i+1}\}=\{r^{m_t}_{j},1_{u_l^*}  \}$, then we repeat the argument in (i'').
	 		\item[(b)] If $\{\underline{\alpha}_i,\underline{\alpha}_{i+1}\}=\{r^{m_t}_{j},r^{m_s}_{l}  \}$, then if $s\ne t$ we repeat the argument in (i''). Else $s=t$ so $(\underline{\alpha}_i,\underline{\alpha}_{i+1})=(r^{m_t}_{j},r^{m_t}_{l})$, then $l=j+1$. 
	 	   In the formula (\ref{equationzeta}) of $\zeta$, we keep $\bar{r}^{m_{t'}}$ when $t'\ne t$ and replace $\bar{r}^{m_t}=(1_{\sigma[m_t]^*},r^{m_t})$ by $(1_{\sigma[m_t]^*},\mathrm{flip}(r^{m_t},j))$ to obtain the new element $\eta\in \Seqq(\sigma,\bar{m})$.
	 	
	 		Then $(-1)^\eta=-(-1)^\zeta$ and by (\ref{flipequation1}) we get $$\partial_i({a}\underset{\beta}{*}\eta)=\partial_i\alpha.$$
	 	This implies $$T(i,\bar{m},\beta,\eta)+T(i,\bar{m},\beta,\zeta)=0.$$
	 		
	 		\item[(c)] If $\{\underline{\alpha}_i,\underline{\alpha}_{i+1}\}=\{1_{\sigma[m_j]},r^{m_t}_{v} \}$ where $r^{m_j}=(r^{m_j}_{1},\dots,r^{m_j}_{m_j-1})\in\ppp(\sigma[m_j])$, then if $j\ne t$ we again repeat the argument in (i''). If $j=t$ then $(\underline{\alpha}_i,\underline{\alpha}_{i+1})=(1_{\sigma[m_j]},r^{m_j}_{1} )$ . Assume that $r^ {m_j}_{1}=c^{\sigma[m_j],l}$ for some $l$. Let 
	 			$$\Delta=m_j-l, \Delta'=l.$$
	 	 We decompose $\sigma[m_j]=\sigma[\Delta']\sqcup \sigma[\Delta]$ as concatenation of $\sigma[\Delta']$ and $\sigma[\Delta']$. By Lemma (\ref{jointpaths}) there exist paths $r^\Delta\in\ppp(\sigma[\Delta]), r^{\Delta'}\in\ppp(\sigma[\Delta'])$ and $\beta_0\in S_{\Delta-1,\Delta'-1}$ such that $$(r^{m_j}_{1},r^\Delta\underset{\beta_0}{*}r^{\Delta'})=r^{m_j}.$$ 
	 		Choose the new partition $\bar{m}'=(m_k,\dots,m_{j+1},\Delta',\Delta,m_{j-1},\dots,m_1)\in\Part{p}$. 
	 		 There exists a unique conditioned shuffle permutation $\gamma\in\bar{S}_{\bar{m}}$ such that\\ $\gamma^{(0)}(\bar{r}^{m_1},\dots,\bar{r}^{m_i-1},\bar{r}^{\Delta},\bar{r}^{\Delta'},\bar{r}^{m_{i+1}},\dots,\bar{r}^{m_k} )\\=(\underline{\alpha}_1,\dots,\underline{\alpha}_{i-1},1_{\sigma[\Delta]},1_{\sigma[\Delta']},\underline{\alpha}_{i+1},\dots,\underline{\alpha}_n).$\\
	 		Let $\eta=\gamma^{(0)}(\bar{r}^{m_1},\dots,\bar{r}^{m_i-1},\bar{r}^{\Delta},\bar{r}^{\Delta'},\bar{r}^{m_{i+1}},\dots,\bar{r}^{m_k} )\in \Seqq(\sigma,\bar{m}')$.	 	
	 		Since $(-1)^\eta=-(-1)^\zeta$, and 
	 		$$\partial_i({a}\underset{\beta}{*}\eta)=\partial_i\alpha,$$
	 		we get $$T(i,\bar{m},\beta,\zeta)+T(i,\bar{m}',\beta,\eta)=0.$$
	 		
	 		\end{itemize}
	 \end{itemize}

\end{proof}

\subsection{$\fff$ and $\GGG$ are quasi-inverse}
We construct homotopy maps $$\{T_{n+1}:\ctil{n+1}\lra\ctil{n}\}$$ to show that $\fff\GGG\sim 1$, then we prove directly that $\GGG\fff(\phi)=\phi$ for any normalized reduced cochain $\phi$. Hence we conclude that both $\fff$ and $\GGG$ are quasi-isomorphisms, in particular, we have
 $$HH^n_{\mathrm{GS}}(\aaa, M) = H^n\cgs{\bullet}\cong H^n(\ctil{\bullet}) = HH^n_{\uuu}(\tilde{\aaa}, \tilde{M}).$$

For each $n$-simplex $\sigma=(u_1,\dots,u_n)$ as in (\ref{eqsigma0}), let $A=(A_i)_{i=0}^n$ where $A_i\in\tilde{\aaa}(U_i)$. Denote
$$\tilde{\aaa}_{\sigma,A}=\tilde{\aaa}_{u_n}(A_{n-1},A_n)\otimes\dots\otimes \tilde{\aaa}_{u_1}(A_0,A_1).$$
Let
$$\Lambda=\{ x\in \tilde{\aaa}_{\sigma,A}|\ \sigma \in\nnn(\uuu),\ A_i\in\aaa(\sigma(i))  \}.$$
Denote by $\langle\Lambda\rangle$ the free abelian group generated by $\Lambda$. Given $\Psi\in\ctil{n}$ and $x=\sum_{\sigma,A} x_{\sigma,A}\in\langle\Lambda\rangle$ where $x_{\sigma,A}\in\tilde{\aaa}_{\sigma,A}$, then we set 
$$\Psi(x)=\sum_{\sigma,A}\Psi(x_{\sigma,A})$$
in which $\Psi(x_{\sigma,A})=0$ if $\sigma\notin \nnn_n(\uuu)$. \\

Let $\sigma=(u_1,\dots,u_n)$ and $\gamma=(v_1,\dots,v_m)$ be simplices as in (\ref{eqsigma0}). Let $A=(A_i)_{i=0}^n$ where $A_i\in\tilde{\aaa}(U_i)$ and $B=(B_i)_{i=0}^m$ where $B_i\in \tilde{\aaa}(V_i)$. 
Given $\tilde{a}=(\tilde{a}_1,\dots,\tilde{a}_n)\in\tilde{\aaa}_{\sigma,A}$ and $\tilde{b}=(\tilde{b}_1,\dots,\tilde{b}_m)\in\tilde{\aaa}_{\gamma,B}$ as in $(\ref{tildea})$, we have $$\simp(\tilde{a})=\sigma;\ \simp(\tilde{b})=\gamma.$$
Assume $A_n=B_0$ and $U_n=V_0$, we define the concatenation
\begin{align*}
	\tilde{b}\sqcup\tilde{a}&=(\tilde{b}_1,\dots,\tilde{b}_m,\tilde{a}_1,\dots,\tilde{a}_n);\\
	\simp(\tilde{b}\sqcup\tilde{a})&=\simp(\tilde{b})\sqcup \simp(\tilde{a})=\gamma\sqcup\sigma .
\end{align*}
We have $\simp(\tilde{a}_i)=(U_{n-i}\underset{u_{n-i+1}}{\lra}U_{n-i+1})$, and so $$\simp(\tilde{a})=\sigma=\simp(\tilde{a}_1)\sqcup\cdots\sqcup \simp(\tilde{a}_n).$$

We use the following notations
\begin{eqnarray*}
	\partial_0(\tilde{a})&=&(\tilde{a}_2,\dots,\tilde{a}_n),\ \mathrm{simp}(\partial_0(\tilde{a}))=\partial_n\sigma;\\
	\partial_i(\tilde{a})&=&(\tilde{a}_1,\dots,\tilde{a}_{i-1},\mu(\tilde{a_i},\tilde{a}_{i+1}),\tilde{a}_{i+2},\dots,\tilde{a}_n),\  \mathrm{simp}(\partial_i(\tilde{a}))=\partial_{n-i}\sigma;\\
	\partial_n(\tilde{a})&=&(\tilde{a}_{1},\dots,\tilde{a}_{n-1}),\ \mathrm{simp}(\partial_n(\tilde{a}))=\partial_0\sigma;\\
	R_p\tilde{a}&=&(\tilde{a}_1,\dots,\tilde{a}_{n-p});\\
	\bar{a}_{n+1-p,\dots,n}&=&\tilde{a}_{n+1-p,\dots,n},\ \mathrm{simp}(\bar{a}_{n+1-p,\dots,n})=(U_0\overset{1}{\lra}U_0).
\end{eqnarray*}
In the abelian group $\langle\Lambda\rangle$, we put 
\begin{eqnarray*}
	\omega_{n,p}(\sigma,\tilde{a})&=&\sum_{\substack{\bar{m}\in\Part{n-p}\\\xi\in\Seq(R_p\sigma, R_p\tilde{a},\bar{m}) }}\sum_{\substack{\bar{m}'\in\Part{p},\beta\in S_{n-p,p}\\ \zeta\in \Seqq(L_p\sigma,\bar{m}') }}(-1)^{\bar{m}+\xi+\zeta+\beta}\xi\underset{\beta}{*}\zeta\sqcup \bar{a}_{n+1-p,\dots, n}\ ;\\
	\omega_n(\sigma,\tilde{a})&=&\sum_{p=1}^n\omega_{n,p}(\sigma,\tilde{a});\\
	\Delta_n(\sigma,\tilde{a})&=&\sum_{\substack{\bar{m}\in\Part{n}\\ \xi\in\Seq(\sigma,\bar{m}) }}(-1)^{\bar{m}+\xi}\xi-(\tilde{a}_1,\dots,\tilde{a}_n).
\end{eqnarray*}
By induction, we set $$\Omega_n(\sigma,\tilde{a})=(-1)^{n+1}\omega_n(\sigma,\tilde{a})+\Omega_{n-1}(\partial_0\sigma,\partial_n\tilde{a})\sqcup\tilde{a}_n,\ \mathrm{for } \ n\ge 2,$$
if $n=1$ then $(\sigma,\tilde{a})$ is represented as $$(\sigma,\tilde{a})=\begin{pmatrix}
A_0\overset{\tilde{a}}{\lra} A_1\\ 
U_0\underset{u_1}{\lra}U_1
\end{pmatrix}$$
and we set $$\Omega_1(\sigma,\tilde{a}) = \begin{pmatrix}
A_0\overset{\tilde{a}}{\lra} u_1^*A_1\overset{1_{u_1^*A_1}}{\lra} A_1\\ 
U_0\underset{1_{U_0}}{\lra}U_0\underset{u_1}{\lra}U_1
\end{pmatrix}.$$\\

Now we define the homotopy maps $\{T_{n+1}:\ctil{n+1}\lra\ctil{n}\}$ as follows:
\begin{align*}
&T_1=0 ,\\ &(T_{n+1}\Psi)^\sigma(\tilde{a})=\Psi(\Omega_n(\sigma,\tilde{a})),\ n\ge1.
\end{align*}
From now on, for simplicity we write $\Omega_n(\tilde{a})$ and $\omega_n(\tilde{a})$ instead of $\Omega_n(\sigma,\tilde{a})$ and $\omega_n(\sigma,\tilde{a})$.

\begin{theorem}\label{homotopy}
	Let $\Psi$ be a cochain in $\ctil{n}$, then we have
	\begin{equation}\label{homoeqn}
	\fff\GGG(\Psi)-\Psi=\delta T_n\Psi+T_{n+1}\delta\Psi
	\end{equation}
\end{theorem}
\begin{proof}
	We have
	\begin{align*}
&	(\fff\GGG\Psi)^\sigma(\tilde{a})=(\fff_0\GGG\Psi)^\sigma(\tilde{a})+\sum_{p=1}^n(\fff_p\GGG\Psi)^\sigma(\tilde{a})\\
	& =(\fff_0\GGG\Psi)^\sigma(\tilde{a})+\sum_{\substack{\bar{m}\in\Part{n-p}\\\xi\in\Seq(R_p\sigma, R_p\tilde{a},\bar{m}) }}\sum_{\substack{\bar{m}'\in\Part{p},\beta\in S_{n-p,p}\\ \zeta\in \Seqq(L_p\sigma,\bar{m}') }}(-1)^{\xi+\zeta+\beta}\Psi(\xi\underset{\beta}{*}\zeta)\tilde{a}_{n+1-p,\dots, n}\\
	&=(\fff_0\GGG\Psi)^\sigma(\tilde{a})+\delta_{n+1}\Psi(\omega_n(\sigma,\tilde{a})).
	\end{align*}
	
	Moreover, we have
	\begin{align*}
	(-1)^{n+1}(T_{n+1}\delta_{n+1}\Psi)^\sigma(\tilde{a}) &=\delta_{n+1}\Psi(\omega_n(\sigma,\tilde{a}))+(-1)^{n+1}\delta_{n+1}\Psi(\Omega_{n-1}(\partial_{0}\sigma,\partial_n\tilde{a})\sqcup\tilde{a}_n)\\
	&=\delta_{n+1}\Psi(\omega_n(\sigma,\tilde{a}))+(-1)^{n+1}\Psi(\Omega_{n-1}(\partial_{0}\sigma,\partial_n\tilde{a}))\tilde{a}_n 
	\end{align*}
	and 
	$$(-1)^n(\delta_nT_n\Psi)^\sigma(\tilde{a})=(-1)^n(T_n\Psi)^{\partial_0\sigma}(\partial_n\tilde{a})\tilde{a}_n=(-1)^n\Psi(\Omega_{n-1}(\partial_{0}\sigma,\partial_n\tilde{a}))\tilde{a}_n .$$
	This implies $$(-1)^{n+1}(T_{n+1}\delta_{n+1}\Psi)^\sigma(\tilde{a})+(-1)^n(\delta_nT_n\Psi)^\sigma(\tilde{a})=\delta_{n+1}\Psi(\omega_n(\sigma,\tilde{a})).$$
	So the equation (\ref{homoeqn}) is equivalent to the equation 
	\begin{eqnarray}
	(\fff_0\GGG\Psi)^\sigma(\tilde{a})-\Psi^\sigma(\tilde{a})&=&\sum_{i=0}^n(-1)^i(T_{n+1}\delta_i\Psi)^\sigma(\tilde{a})+\sum_{i=0}^{n-1}(-1)^i(\delta_iT_n\Psi)^\sigma(\tilde{a})\nonumber.
	\end{eqnarray}
	This equation holds true due to Lemma \ref{lemmahomoeqn1}. 
\end{proof}

\begin{lemma}\label{lemmahomoeqn1}
	Let $\sigma=(u_1,\dots,u_{n})$ be a simplex and  $\tilde{a}=(\tilde{a}_1,\dots,\tilde{a}_n)$ as in (\ref{tildea}), then we have
	\begin{eqnarray}\label{homoeqn1}
	\sum_{i=0}^n(-1)^i\partial_i\Omega_n(\tilde{a}_1,\dots,\tilde{a}_n)+\sum_{i=0}^{n-1}(-1)^i\Omega_{n-1}(\partial_i(\tilde{a}_1,\dots,\tilde{a}_n))=\Delta_n(\tilde{a}).
	\end{eqnarray}
\end{lemma}
\begin{proof}
	The equation (\ref{homoeqn1}) is equivalent to 
	\begin{align*}
\sum_{i=0}^{n-1}(-1)^i\Omega_{n-1}(\partial_i\tilde{a})=&\Delta_n(\tilde{a})-\sum_{i=0}^n(-1)^{i+n+1}\partial_i\omega_n(\tilde{a})-\sum_{i=0}^n(-1)^{i}\partial_i(\Omega_{n-1}(\tilde{a}_1,\dots,\tilde{a}_{n-1})\sqcup\tilde{a}_n). 
	\end{align*}
	Assume that the equation (\ref{homoeqn1}) holds for $n$. We now prove it holds for $n+1$. Assume $\tilde{a}=(\tilde{a}_1,\dots,\tilde{a}_{n+1})$. Let
	$$\displaystyle B=\sum_{i=0}^{n+1}(-1)^i\partial_i\Omega_{n+1}(\tilde{a}_1,\dots,\tilde{a}_{n+1}),\ \mathrm{and}\ 
	C=\sum_{i=0}^{n}(-1)^i\Omega_{n}(\partial_i(\tilde{a}_1,\dots,\tilde{a}_{n+1})).$$
	We need to prove \begin{align}\label{homoeqn2}
	B+C=\Delta(\tilde{a}_1,\dots,\tilde{a}_{n+1}).
	\end{align}
	By definition, we have 
	\begin{eqnarray*}
		B&=&\sum_{i=0}^{n+1}(-1)^{i+n+2}\partial_i\omega_{n+1}(\tilde{a}_1,\dots,\tilde{a}_{n+1})+\sum_{i=0}^{n+1}(-1)^i\partial_i(\Omega_n(\tilde{a}_1,\dots,\tilde{a}_n)\sqcup \tilde{a}_{n+1})\\
		&=&\sum_{i=0}^{n+1}(-1)^{i+n+2}\partial_i\omega_{n+1}(\tilde{a}_1,\dots,\tilde{a}_{n+1})+B_1+B_2
	\end{eqnarray*}
	where
	\begin{align*}
	B_1&=\sum_{i=0}^{n+1}(-1)^{i+n+1}\partial_i(\omega_n(\tilde{a}_1,\dots,\tilde{a}_n)\sqcup \tilde{a}_{n+1})\\
	B_2&=\sum_{i=0}^{n+1}(-1)^{i}\partial_i(\Omega_{n-1}(\tilde{a}_1,\dots,\tilde{a}_{n-1})\sqcup\tilde{a}_n\sqcup\tilde{a}_{n+1}).
	\end{align*}
	We also have 
	\begin{align*}
	&C=\sum_{i=0}^n(-1)^{i+n+1}\omega_n(\partial_i(\tilde{a}_1,\dots,\tilde{a}_{n+1}))+\sum_{i=0}^{n-1}\Omega_{n-1}(\partial_i(\tilde{a}_1,\dots,\tilde{a}_n))\sqcup\tilde{a}_{n+1}\\
	&+(-1)^n\Omega_{n-1}(\tilde{a}_1,\dots,\tilde{a}_{n-1})\sqcup\tilde{a}_{n,n+1}.
	\end{align*}
	By induction hypothesis, we have 
	\begin{align*}
	&\sum_{i=0}^{n-1}(-1)^i\Omega_{n-1}(\partial_i(\tilde{a}_1,\dots,\tilde{a}_n))\sqcup\tilde{a}_{n+1}\\
	&=\Delta_n(\tilde{a}_1,\dots,\tilde{a}_n)\sqcup\tilde{a}_{n+1}-\sum_{i=0}^n(-1)^{i+n+1}\partial_i\omega_n(\tilde{a}_1,\dots,\tilde{a}_n)\sqcup\tilde{a}_{n+1}\\
	&\ \ \ -\sum_{i=0}^n(-1)^i\partial_i(\Omega_{n-1}(\tilde{a}_1,\dots,\tilde{a}_{n-1})\sqcup\tilde{a}_n)\sqcup\tilde{a}_{n+1}\\
	&=\Delta_n(\tilde{a}_1,\dots,\tilde{a}_n)\sqcup\tilde{a}_{n+1}-(B_1-\partial_{n+1}(\omega_n(\tilde{a}_1,\dots,\tilde{a}_n)\sqcup \tilde{a}_{n+1}))\\
	&\ \ \ - (B_2-(-1)^{n+1}\Omega_{n-1}(\tilde{a}_1,\dots,\tilde{a}_{n-1})\sqcup\tilde{a}_{n,n+1} ) .
	\end{align*}
	This implies 
	\begin{align*}
	B+C&=\sum_{i=0}^{n+1}(-1)^{i+n+2}\partial_i\omega_{n+1}(\tilde{a}_1,\dots,\tilde{a}_{n+1})+\sum_{i=0}^n(-1)^{i+n+1}\omega_n(\partial_i(\tilde{a}_1,\dots,\tilde{a}_{n+1}))\\
	&\ \ \ +\partial_{n+1}(\omega_n(\tilde{a}_1,\dots,\tilde{a}_n)\sqcup \tilde{a}_{n+1})+\Delta_n(\tilde{a}_1,\dots,\tilde{a}_n)\sqcup\tilde{a}_{n+1}.
	\end{align*}
	Thus, by Lemma (\ref{homomain1}), we obtain the equation (\ref{homoeqn2}).
\end{proof}

\begin{lemma}\label{homomain1} 	Let $\sigma=(u_1,\dots,u_{n+1})$ be a simplex and  $\tilde{a}=(\tilde{a}_1,\dots,\tilde{a}_{n+1})$ as in (\ref{tildea}), then we have
	\begin{align*}
	&\sum_{i=0}^{n+1}(-1)^{i+n+2}\partial_i\omega_{n+1}(\tilde{a}_1,\dots,\tilde{a}_{n+1})+\sum_{i=0}^n(-1)^{i+n+1}\omega_n(\partial_i(\tilde{a}_1,\dots,\tilde{a}_{n+1}))\\
	&\  +\partial_{n+1}(\omega_n(\tilde{a}_1,\dots,\tilde{a}_n)\sqcup \tilde{a}_{n+1})+\Delta_n(\tilde{a}_1,\dots,\tilde{a}_n)\sqcup\tilde{a}_{n+1}
	-\Delta(\tilde{a}_1,\dots,\tilde{a}_{n+1})=0.
	\end{align*}
\end{lemma}
\begin{proof}
	We denote \begin{align*}
	&B=\sum_{i=0}^{n+1}(-1)^{i+n+2}\partial_i\omega_{n+1}(\tilde{a}_1,\dots,\tilde{a}_{n+1});\ \ C=\sum_{i=0}^n(-1)^{i+n+1}\omega_n(\partial_i(\tilde{a}_1,\dots,\tilde{a}_{n+1}));\\
	&D=\partial_{n+1}(\omega_n(\tilde{a}_1,\dots,\tilde{a}_n)\sqcup \tilde{a}_{n+1}); \ E=\Delta_n(\tilde{a}_1,\dots,\tilde{a}_n)\sqcup\tilde{a}_{n+1}
	-\Delta(\tilde{a}_1,\dots,\tilde{a}_{n+1}).
	\end{align*}
	We prove that each of the terms appearing in the expansion of $B$ is canceled out against a unique term in $C,D$, or $E$, and vice-versa. The cancellation is as follows: 
	$$\xymatrix{B\ar[r]\ar[d]\ar[rd]&\ar[l]C\\ D\ar[u]&\ar[lu]E}.$$
	We write 
	\begin{align*}
	&B_p=\sum_{i=0}^{n+1}(-1)^{i+n+2}\partial_i\omega_{n+1,p}(\tilde{a})\\
	&=\sum_{i=0}^{n+1}\sum_{\substack{\bar{m}\in\Part{n+1-p}\\\xi\in\Seq(R_p\sigma, R_p\tilde{a},\bar{m}) }}\sum_{\substack{\bar{m}'\in\Part{p},\beta\in S_{n+1-p,p}\\ \zeta\in \Seqq(L_p\sigma,\bar{m}') }}(-1)^{i+n+2}(-1)^{\bar{m}+\xi+\zeta+\beta}\partial_i B_p(\tilde{a},\xi,\zeta,\beta)
	\end{align*} 
	where
	$$ B_p(\tilde{a},\xi,\zeta,\beta)=\xi\underset{\beta}{*}\zeta\sqcup \bar{a}_{n+2-p,\dots, n+1}$$
	and write
	\begin{align*}
	&C_p=\sum_{i=0}^{n}(-1)^{i+n+1}\omega_{n,p}(\partial_i\tilde{a})\\
	&=\sum_{i=0}^{n}\sum_{\substack{\bar{m}\in\Part{n-p}\\\xi\in\Seq(R_p\partial_{n+1-i}\sigma, R_p\partial_i\tilde{a},\bar{m}) }}\sum_{\substack{\bar{m}'\in\Part{p},\beta\in S_{n-p,p}\\ \zeta\in \Seqq(L_p\partial_{n+1-i}\sigma,\bar{m}') }}(-1)^{i+n+1}(-1)^{\bar{m}+\xi+\zeta+\beta} C_p(\partial_i\tilde{a},\xi,\zeta,\beta)
	\end{align*} 
	where 
	$$  C_p(\partial_i\tilde{a},\xi,\zeta,\beta)=\xi\underset{\beta}{*}\zeta\sqcup \overline{(\partial_i\tilde{a})}_{n+1-p,\dots, n}.$$
	We also write
	\begin{align*}
	E_0&=-\Delta(\tilde{a}_1,\dots,\tilde{a}_{n+1})
	=\sum_{\substack{\bar{m}\in\Part{n+1}\\
			\xi\in\Seq(\sigma,\bar{m})}}(-1)^{\bar{m}+\xi+1}\xi ;\\
	E_1&=\Delta(\tilde{a}_1,\dots,\tilde{a}_{n})\sqcup \tilde{a}_{n+1}
	=\sum_{\substack{\bar{m}\in\Part{n}\\
			\xi\in\Seq(\partial_0\sigma,\bar{m})}}(-1)^{\bar{m}+\xi}\xi\sqcup\tilde{a}_{n+1}.
	\end{align*}
	and
	\begin{align*}
	&D_p=\partial_{n+1}\omega_{n,p}(\tilde{a}_1,\dots,\tilde{a}_n)\sqcup\tilde{a}_{n+1}\\
	&=\sum_{\substack{\bar{m}\in\Part{n-p}\\\xi\in\Seq(R_p\partial_0\sigma, R_p\partial_{n+1}\tilde{a},\bar{m}) }}\sum_{\substack{\bar{m}'\in\Part{p},\beta\in S_{n-p,p}\\ \zeta\in \Seqq(L_p\partial_0\sigma,\bar{m}') }}(-1)^{\bar{m}+\xi+\zeta+\beta}\partial_{n+1} D_p((\tilde{a}_1,\dots,\tilde{a}_n),\xi,\zeta,\beta)
	\end{align*} 
	where
	$$D_p(\partial_{n+1}\tilde{a},\xi,\zeta,\beta)=\xi\underset{\beta}{*}\zeta\sqcup \bar{a}_{n+1-p,\dots, n}\sqcup\tilde{a}_{n+1}.$$
	
	Assume that $\bar{m}=(m_l,\dots,m_1)\in\Part{n+1-p}$ and $\bar{m}'=(m_k',\dots,m_1')\in\Part{p}$. Let $\xi\in\Seq(R_p\sigma,\bar{m}),\ \zeta\in\Seqq(L_p\sigma,\bar{m}')$ and $\beta\in S_{n+1-p,p}$. We denote 
	$$B_p(\tilde{a},\xi,\zeta,\beta)=(b_1,\dots,b_{n+2}),$$
	and denote $(\underline{b}_1,\dots,\underline{b}_{n+2})$ the formal sequence of $B_p(\tilde{a},\xi,\zeta,\beta)$.
	
	\emph{Step 1.} Consider the case $i=0$, then $\partial_0(B_p(\tilde{a},\xi,\zeta,\beta))=(b_2,\dots,b_{n+2})$.  There are only the following three cases:
	$$\underline{b}_1=\left[\begin{matrix}
	\tilde{a}_1;&\ \ \ \ \ \ \ \ \ \ \\
	r^{m_t}_{1} & \text{ where } r^{m_t}=(r^{m_t}_{1},\dots,r^{m_t}_{m_t-1})\in \ppp(\sigma[m_t]);\\
	1_{(L_p\sigma)[m_1']^*}.&
	\end{matrix}\right.
	$$
	\begin{itemize}
		\item[(i)] Assume $\underline{b}_1=\tilde{a}_1$. Then $m_1=1$, and we choose $\tilde{m}=(m_l,\dots,m_2)\in \Part{n-p}$. There exists a unique element $\xi'\in \Seq(R_p(\partial_{n+1}\sigma),\tilde{m})$ such that $b_1\sqcup\xi'=\xi$. There exists a unique $\beta'\in S_{n-p,p}$ such that 
		$$(b_1\sqcup \xi')\underset{\beta'}{*}\zeta=\xi\underset{\beta}{*}\zeta$$
		Since $(-1)^{\tilde{m}+\xi'+\beta'}=(-1)^{\bar{m}+\xi+\beta}$, we get
		$$(-1)^{n+2}(-1)^{\bar{m}+\xi+\zeta+\beta}\partial_0B_p(\tilde{a},\xi,\zeta,\beta)+(-1)^{n+1}(-1)^{\tilde{m}+\xi'+\zeta+\beta'}C_p(\partial_0\tilde{a},\xi',\zeta,\beta')=0.$$
		\item[(ii)] Assume $\underline{b}_1=r^{m_t}_{1}$ for some $t$, where $r^{m_t}=(r^{m_t}_{1},\dots,r^{m_t}_{m_t-1})\in \ppp(\sigma[m_t])$. Then $r^{m_t}_1=c^{(R_p\sigma)[m_t],j}$ for some $j$. Set $\Delta=j,\ \Delta'= m_t-j.$		
		Using the analogous argument as in Case 2 of Step 2 in the proof of Proposition \ref{Fcommutes}, considering the partition $$\tilde{m}=(m_l,\dots,m_{t+1},\Delta,\Delta',m_{t-1},\dots,m_1)\in \Part{n+1-p}$$
		we find unique $ \xi'\in \Seq(R_p\sigma,\tilde{m})$ and $1\le j_0 \le n+1$ such that 
		$$(-1)^{n+2}(-1)^{\bar{m}+\xi+\zeta+\beta}\partial_0B_p(\tilde{a},\xi,\zeta,\beta)+(-1)^{j_0+n+2}(-1)^{\tilde{m}+\xi'+\zeta+\beta}\partial_{j_0}B_p(\tilde{a},\xi,\zeta,\beta)=0.$$
		\item[(iii)] Assume $\underline{b}_1=1_{(L_p\sigma)[m_1']^*}$. \\
		If $m'_1<p$, choose $\tilde{m}'=(m_k',\dots,m_2')\in \Part{p-m_1'}$ and $\tilde{m}=(m_1',m_l,\dots,m_1)\in\Part{n+1-p+m_1'}$. There exists unique $\xi'\in\Seq(R_{p-m_1'}\sigma,\tilde{m}),\ \zeta'\in \Seqq(L_{p-m_1'}\sigma,\tilde{m}')$ and $\beta'\in S_{n+1-p+m_1',p-m_1'}$ such that 
		$$\xi'\underset{\beta'}{*}\zeta'\sqcup \bar{a}_{n+2-p+m_1',\dots,n+1}=(b_2,\dots,b_{n+1})\sqcup\bar{a}_{n+2-p,\dots,n+1-p+m_1'}\sqcup \bar{a}_{n+2-p+m_1',\dots,n+1}$$
		so we get  $$\partial_{n+1}(\xi'\underset{\beta'}{*}\zeta'\sqcup \bar{a}_{n+2-p+m_1',\dots,n+1})=(b_2,\dots,b_{n+2}).$$ 
		By a sign computations, we obtain 
		$$(-1)^{n+2}(-1)^{\bar{m}+\xi+\zeta+\beta}\partial_0B_p(\tilde{a},\xi,\zeta,\beta)-(-1)^{\tilde{m}+\xi'+\zeta'+\beta'}\partial_{n+1}B_{p-m_1'}(\tilde{a},\xi',\zeta',\beta')=0.$$
		If $m_1'=p$, then $\bar{m}'=(m_1')\in\Part{p}$ and thus $\simp(b_i)=(U_0\overset{1}{\lra}U_0)$ for $i=2,\dots,(n+1)$. Let $\tilde{m}=(m_1',m_l,\dots,m_1)\in \Part{n+1}$. It is seen that $\xi'=(b_2,\dots,b_{n+1})\in\Seq(\sigma,\tilde{m})$, by a sign computation we have $$(-1)^{n+2}(-1)^{\bar{m}+\xi+\zeta+\beta}\partial_0B_p(\tilde{a},\xi,\zeta,\beta)=(-1)^{\tilde{m}+\xi'}\xi'.$$
		Hence the term $(-1)^{n+2}(-1)^{\bar{m}+\xi+\zeta+\beta}\partial_0B_p(\tilde{a},\xi,\zeta,\beta)$ is killed by the term $(-1)^{1+\tilde{m}+\xi'}\xi'$ in $E_0$. In this way, when $p$ runs through $\{1,\dots,n+1\}$, every term in $E_0$ is eliminated except the term $-(\tilde{a}_1,\dots,\tilde{a}_{n+1})$ which is eliminated by the term $(\tilde{a}_1,\dots,\tilde{a}_{n})\sqcup\tilde{a}_{n+1}$ in $E_1$.
	\end{itemize}
	
	\emph{Step 2.} Consider the case $1\le i\le n$. We write $\xi=(\xi_1,\dots,\xi_{n+1-p})$ and $\zeta=(\zeta_1,\dots,\zeta_p)$. We have 
	$$\partial_iB_p(\tilde{a},\xi,\zeta,\beta)=(b_1,\dots,b_{i-1},\mu(b_i,b_{i+1}),b_{i+2},\dots,b_{n+2}).$$
	There are only the following three cases:
	$$\left[\begin{matrix}
	&\{\underline{b}_i,\underline{b}_{i+1}\}=\{\xi_j,\zeta_{j'}\} \text{ for some } j,j';\\
	&\{\underline{b}_i,\underline{b}_{i+1}\} \subseteq \{\xi_1,\dots,\xi_{n+1-p}\};\ \ \ \ \ \ \ \ \\
	&\{\underline{b}_i,\underline{b}_{i+1}\} \subseteq \{\zeta_1,\dots,\zeta_p\}.\ \ \ \ \ \ \ \ \ \ \ \ \ \ 
	\end{matrix}\right.
	$$
	\begin{itemize}
		\item Assume $\{\underline{b}_i,\underline{b}_{i+1}\}=\{\xi_j,\zeta_{j'}\}$. Choose $\beta'=(i,i+1)\circ\beta$ then
		$$(-1)^{i+n+2}(-1)^{\bar{m}+\xi+\zeta+\beta}\partial_iB_p(\tilde{a},\xi,\zeta,\beta)+(-1)^{i+n+2}(-1)^{\bar{m}+\xi+\zeta+\beta'}\partial_iB_p(\tilde{a},\xi,\zeta,\beta')=0.$$
		\item Assume $\{\underline{b}_i,\underline{b}_{i+1}\} \subseteq \{\xi_1,\dots,\xi_{n+1-p}\}$. We repeat the arguments of Step 2 in the Proposition \ref{Fcommutes}.
		\item Assume $\{\underline{b}_i,\underline{b}_{i+1}\} \subseteq \{\zeta_1,\dots,\zeta_p\}$. We repeat the arguments of Step 3 in the Proposition \ref{Gcommutes}.
	\end{itemize}
	
	\emph{Step 3.} Consider the case $i=n+1$. We have
	\begin{align*}
	&\partial_{n+1}B_p(\tilde{a},\xi,\zeta,\beta)=\partial_{n+1}((b_1,\dots,b_{n+1})\sqcup\bar{a}_{n+2-p,\dots,n+1})\\
	&=(b_1,\dots,b_n,\mu(b_{n+1},\bar{a}_{n+2-p,\dots,n+1})).
	\end{align*}
	There are only the following three cases for $\underline{b}_{n+1}$:
	$$\underline{b}_{n+1}=\left[\begin{matrix}
	\tilde{a}[m_l];&\\
	r^{m'_t}_{m'_t-1}& \text{ where } r^{m'_t}=(r^{m'_t}_{1},\dots,r^{m'_t}_{m'_t-1})\in\ppp(L_p\sigma[m'_t]);\\
	1_{(L_p\sigma)[m'_k]^*} &\text{ where } m'_k=1 \text{ and } \bar{m}'=(1,m'_{k-1},\dots,m'_1)\in\Part{p}.
	\end{matrix}\right.$$
	\begin{itemize}
		\item Assume $\underline{b}_{n+1}=\tilde{a}[m_l]$. We apply the argument in (iii) of Step 1, then every term of this form is killed.
		
		\item Assume $\underline{b}_{n+1}=r^{m'_t}_{m'_t-1}$. We assume that $r^{m'_t}_{m'_t-1}=\epsilon^{L_p\sigma[m'_t],j}$ for some $j$. Then 
		$$\mu(r^{m'_t}_{1},\bar{a}_{n+2-p,\dots,n+1})=\overline{(\partial_j\tilde{a})}_{n+2-p,\dots,n+1}.$$
		In $C$ we consider terms $C_p(\partial_j\tilde{a},\xi',\zeta',\beta')$. There exists unique $(\xi',\zeta',\beta')$ such that 
		$$-(-1)^{\bar{m}+\xi+\zeta+\beta}\partial_{n+1}B_p(\tilde{a},\xi,\zeta,\beta)+(-1)^{n+1+j}(-1)^{\tilde{m}+\xi'+\zeta'+\beta'}C_p(\partial_j\tilde{a},\xi',\zeta',\beta')=0$$
		Combining with Step 2 and (i) in Step 1, we see that every term in $C$ is killed.
		\item Assume that $\underline{b}_{n+1}=1_{(L_p\sigma)[m'_k]^*}$ where $\bar{m}'=(1,m'_{k-1},\dots,m'_1)\in\Part{p}$. Thus we have $\underline{b}_{n+1}=1_{u_1^*}$ and $\simp(b_{n+1})=(U_0\overset{u_1}{\lra}U_1)$.\\
		 If $p=1$, then we have $\bar{m}'=(1)$, $\zeta=1_{L_p\sigma[1]^*}=b_{n+1}$, $\beta=1$ and $\simp(b_i)=(U_1\overset{1}{\lra}U_1)$ for $i\le n$. So $\partial_{n+1}B_1(\tilde{a},\xi,\zeta,\beta)=\xi\sqcup \tilde{a}_{n+1}$. We have $(-1)^{\bar{m}+\xi}\xi\sqcup \tilde{a}_{n+1}$ is a term in $E_1$ and 
		$$-(-1)^{\bar{m}+\xi+\zeta+\beta}\partial_{n+1}B_1(\tilde{a},\xi,\zeta,\beta)+(-1)^{\bar{m}+\xi}\xi\sqcup \tilde{a}_{n+1}=0.$$
		So we see that  every term in $E_1$ is killed.\\
		If $p>1$, recall that $\bar{m}=(m_l,\dots,m_1)$. We show that  the term
		$$-(-1)^{\bar{m}+\xi+\zeta+\beta}\partial_{n+1}B_p(\tilde{a},\xi,\zeta,\beta)$$ is killed by a term in $D$. Thus in the expression of $D_p$, we choose $\tilde{m}'=(m_{k-1}',\dots,m_1')\in\Part{p-1}$, $\tilde{m}=\bar{m}\in \Part{n+1-p}$, $\xi'=\xi$ and $\zeta'=(\zeta_1,\dots,\zeta_{p-1})$. There exists a unique $\beta'\in S_{n+1-p,p-1}$ such that
		$$-(-1)^{\bar{m}+\xi+\zeta+\beta}\partial_{n+1}B_p(\tilde{a},\xi,\zeta,\beta)+(-1)^{\tilde{m}+\xi'+\zeta'+\beta'}D_{p-1}(\partial_{n+1}\tilde{a},\xi',\zeta',\beta')=0.$$
		When $p$ varies, we see that every term in $D$ is killed.
	\end{itemize}
	
\end{proof}

\begin{proposition}\label{propGF}
	Let $\phi$ be a normalized reduced cochain in $\cgsnr{n}$ then we have $$\GGG\fff(\phi)=\phi.$$
\end{proposition}
\begin{proof}
	Assume that  $p+q=n$. Let $\sigma=(u_1,\dots,u_p)$ be a $p$-simplex as in (\ref{eqsigma0}) and let $A_0,A_1,\dots,A_q\in \mathrm{Ob}(\aaa(U_p))$. Take $a=(a_1,\dots,a_q)$ where $a_i\in\aaa(U_p)(A_{n-i},A_{n+1-i})$ as in (\ref{equationa0}), we show that 
	\begin{equation}
	(\GGG\fff(\phi))^\sigma(a)=\phi^\sigma(a).
	\end{equation} 
	We have 
	\begin{align*}
	(\GGG\fff\phi)^\sigma(a)&=\sum_{i=0}^n(\GGG\fff_i\phi)^\sigma(a)\\
	&=\sum_{i=0}^n\sum_{\substack{\bar{m}'\in\Part{p},\beta\in S_{p,q}\\ \zeta\in\Seqq(\sigma,\bar{m}') }}(-1)^{\beta+\zeta}(\fff_i\phi)^{\simp({a}\underset{\beta}{*}\zeta)}({a}\underset{\beta}{*}\zeta)
	\end{align*}
	Denoting $\tilde{b}=({a}\underset{\beta}{*}\zeta)$, we have
	\begin{align*}
	&(\fff_i\phi)^{\simp({a}\underset{\beta}{*}\zeta)}({a}\underset{\beta}{*}\zeta)
	=\sum_{\substack{\bar{m}\in\Part{n-i}\\\xi\in \Seq(R_i\simp(\tilde{b}),\bar{m})}}(-1)^{\bar{m}+\xi}\fff_i^{\simp(\tilde{b}),\bar{m},A_n}\phi^{L_i(\simp(\tilde{b}))}(\xi)\tilde{b}_{n+1-i,\dots,n}.
	\end{align*}
	\begin{itemize}
		\item When $i>p$, then $L_i(\simp(\tilde{b}))$ is degenerate, so $\phi^{L_i(\simp(\tilde{b}))}=0$, and thus $(\GGG\fff_i\phi)^\sigma(a)=0$.
		\item When $i<p$, then it is seen that $\xi$ is normal, so $\phi^{L_i(\simp(\tilde{b}))}(\xi)=0$, and thus $(\GGG\fff_i\phi)^\sigma(a)=0$.\\
		\item We have $(\GGG\fff\phi)^\sigma(a)=(\GGG\fff_p\phi)^\sigma(a)$. $L_p(\simp(\tilde{b}))$ is non-degenerate if and only if $\bar{m}'=(1,1,\dots,1)$ and $\beta=1$. Then $L_p(\simp(\tilde{b}))=\sigma$, we have that $\xi$ is not normal if and only if $\bar{m}=(1,1,\dots,1)$. Thus we get $(\GGG\fff_p\phi)^\sigma(a)=\phi^\sigma(a)$.
	\end{itemize}

\end{proof}

\def\cprime{$'$} \def\cprime{$'$}
\providecommand{\bysame}{\leavevmode\hbox to3em{\hrulefill}\thinspace}
\providecommand{\MR}{\relax\ifhmode\unskip\space\fi MR }
\providecommand{\MRhref}[2]{%
  \href{http://www.ams.org/mathscinet-getitem?mr=#1}{#2}
}
\providecommand{\href}[2]{#2}

\end{document}